\documentclass[twoside,11pt]{article}

\usepackage{blindtext}

\usepackage{jmlr2e}

\usepackage[utf8]{inputenc} 
\usepackage[T1]{fontenc}    
\usepackage{hyperref}       
\usepackage{url}            
\usepackage{booktabs}       
\usepackage{amsfonts}       
\usepackage{nicefrac}       
\usepackage{microtype}      
\usepackage[table]{xcolor}         
\usepackage{tocvsec2}
\usepackage{titletoc}
\usepackage{setspace}
\usepackage{xspace}
\usepackage{caption}

\let\proof\relax

\usepackage{amsmath}
\usepackage{amsthm}
\usepackage{amssymb}

\usepackage{multicol}
\usepackage{multirow}
\usepackage{makecell}
\usepackage{enumitem}
\usepackage{wrapfig}
\usepackage{dsfont}

\usepackage{graphicx,epstopdf}
\usepackage{subfigure} 

\usepackage{algorithm}
\usepackage{algorithmic}

\usepackage{threeparttable, booktabs}
\usepackage{mathtools}
\usepackage{diagbox}
\mathtoolsset{showonlyrefs=false}

\newtheorem{assumption}{Assumption}
\newtheorem{definition}{Definition}
\newtheorem{remark}{Remark}
\newtheorem{lemma}{Lemma}
\newtheorem{proposition}{Proposition}
\newtheorem{theorem1}{Theorem}
\newtheorem{corollary}[theorem1]{Corollary}

\newcommand{\ours}{D-SOBA }
\newcommand{\oursFO}{D-SOBA-FO }
\newcommand{\oursSO}{D-SOBA-SO }

\newcommand{\cD}{{\mathcal{D}}}

\newcommand{\cN}{{\mathcal{N}}}
\newcommand{\cO}{{\mathcal{O}}}

\newcommand{\EE}{\mathbb{E}}
\newcommand{\RR}{\mathbb{R}}

\def\one{\mathds{1}}

\DeclareMathOperator*{\argmin}{arg\,min}

\newcommand{\eg}{\textit{e.g.}}
\newcommand{\ie}{\textit{i.e.}}

\usepackage{xcolor,pifont}
\newcommand{\cmark}{\textcolor{green}{\ding{52}}}
\newcommand{\xmark}{\textcolor{red}{\ding{55}}}

\newcommand{\vvvert}{{\vert\kern-0.25ex\vert\kern-0.25ex\vert}}

\usepackage{lastpage}
\jmlrheading{26}{2025}{1-\pageref{LastPage}}{4/25; Revised 9/25}{11/25}{25-0677}{Boao Kong, Shuchen Zhu, Songtao Lu, Xinmeng Huang, Kun Yuan}

\ShortHeadings{Decentralized Bilevel Optimization: A Perspective from Transient Complexity}{Kong, Zhu, Lu, Huang, and Yuan}
\firstpageno{1}

\begin{document}

\title{Decentralized Bilevel Optimization: A Perspective from Transient Iteration Complexity}

\author{\name Boao Kong$^*$ \email kongboao@stu.pku.edu.cn \\
       \addr Center for Data Science, Peking University\\
       Beijing, China
       \AND
       \name Shuchen Zhu$^*$\email shuchenzhu@stu.pku.edu.cn \\
       \addr Center for Data Science, Peking University\\
       Beijing, China
       \AND
       \name Songtao Lu \email stlu@cse.cuhk.edu.hk \\
       \addr Department of Computer Science and Engineering\\
       The Chinese University of Hong Kong\\
       Hong Kong SAR
       \AND
       \name Xinmeng Huang \email xinmengh@sas.upenn.edu \\
       \addr Graduate Group in Applied Mathematics and Computational Science\\
       University of Pennsylvania\\
       Philadelphia, USA
       \AND
       \name Kun Yuan$^\dagger$ \email kunyuan@pku.edu.cn \\
       \addr Center for Machine Learning Research, Peking University\\
       AI for Science Institute\\
       Beijing, China
       }
       
\editor{Shiqian Ma}

\maketitle
\def\thefootnote{*}\footnotetext{Equal contribution.}
\def\thefootnote{$\dagger$}\footnotetext{Corresponding author.}
\begin{abstract}
Stochastic bilevel optimization (SBO) is becoming increasingly essential in machine learning due to its versatility in handling nested structures. To address large-scale SBO, decentralized approaches have emerged as effective paradigms in which nodes communicate with immediate neighbors without a central server, thereby improving communication efficiency and enhancing algorithmic robustness. However, most decentralized SBO algorithms focus solely on asymptotic convergence rates, overlooking transient iteration complexity-the number of iterations required before asymptotic rates dominate, which results in limited understanding of the influence of network topology, data heterogeneity, and the nested bilevel algorithmic structures. To address this issue, this paper introduces \textbf{D-SOBA}, a \textbf{D}ecentralized \textbf{S}tochastic \textbf{O}ne-loop \textbf{B}ilevel \textbf{A}lgorithm framework. D-SOBA comprises two variants: D-SOBA-SO, which incorporates second-order Hessian and Jacobian matrices, and D-SOBA-FO, which relies entirely on first-order gradients. We provide a comprehensive non-asymptotic convergence analysis and establish the transient iteration complexity of D-SOBA. This provides the first theoretical understanding of how network topology, data heterogeneity, and nested bilevel structures influence decentralized SBO. Extensive experimental results demonstrate the efficiency and theoretical advantages of D-SOBA.
\end{abstract}


\section{Introduction}

Stochastic bilevel optimization, which tackles problems with nested optimization structures, has gained growing interest. This two-level structure provides a flexible and potent framework for addressing a broad range of tasks, spanning from hyperparameter optimization \citep{bertinetto2019meta,franceschi2018bilevel,snell2017prototypical} to reinforcement learning \citep{hong2023two}, adversarial learning \citep{zhang2022revisiting}, continue learning \citep{borsos2020coresets}, and imitation learning \citep{arora2020provable}. State-of-the-art performance in these tasks is typically achieved with extremely large  datasets, which necessitates efficient distributed algorithms for stochastic bilevel optimization across multiple computing nodes. 

This paper considers $N$ nodes connected through a given graph topology. Each node $i$ privately owns an upper-level loss function $f_i: \mathbb{R}^d \times \mathbb{R}^p  \rightarrow \mathbb{R}$ and a lower-level loss function $g_i: \mathbb{R}^d \times \mathbb{R}^p \rightarrow \mathbb{R}$. The goal of all nodes is to collaboratively find a solution to the following distributed stochastic bilevel optimization problem:
\begin{subequations}
\label{prob:general}
\begin{align}
\min_{x \in \RR^d} \quad & \Phi(x) := f(x,y^\star(x)):=\frac{1}{N}\sum_{i=1}^N f_i\left(x, y^\star(x)\right) \label{eq:upper} \\
\mathrm{s.t.} \quad &  y^\star(x) := \argmin_{y \in \RR^p}  \left\{ g(x,y)  :=  \frac{1}{N}\sum_{i=1}^N g_i(x, y)\right\}  \label{eq:lower} 
\end{align}
\end{subequations}
where $f_i$ and $g_i$ are defined as the expectation of $F(x,y;\xi_i)$ and $G(x,y;\zeta_i)$: 
\begin{align}\label{eq:f-and-g}
f_i(x, y) := \EE_{\xi_i \sim \cD_{f_i}}[F(x, y;\xi_i)], \quad
g_i(x, y) := \EE_{\zeta_i \sim \cD_{g_i}}[G(x, y; \zeta_i)].
\end{align}
The random variables $\xi_i$ and $\zeta_i$ represent data samples available at node $i$, drawn from the local upper-level distribution $\cD_{f_i}$ and the lower-level distribution $\cD_{g_i}$, respectively. Throughout this paper, we assume that these local data distributions vary across nodes, potentially causing data heterogeneity issues during training.

\subsection{Decentralized Bilevel Algorithms}

Conventional distributed approaches for solving problem \eqref{prob:general} typically follow a centralized paradigm \citep{huang2023achieving, tarzanagh2022fednest, yang2023simfbo}, where a central server synchronizes updates across the network by computing a globally averaged upper- or lower-level stochastic gradient to update the model parameters. This global averaging step is commonly implemented using mechanisms such as Parameter Server \citep{li2014scaling, smola2010architecture} or Ring-Allreduce \citep{ring-allreduce}. However, these techniques incur substantial bandwidth costs or high latency \citep[Table I]{ying2021bluefog}, which severely limits the scalability of centralized bilevel optimization.

Decentralized bilevel optimization is a new paradigm to approach problem \eqref{prob:general} eliminating the global averaging step, wherein each node maintains a local model updated by communicating with its immediate neighbors.
This distinctive feature results in a noteworthy reduction in communication costs compared to centralized approaches. Beyond the communication efficiency, decentralized bilevel algorithms also exhibit enhanced robustness to node and link failures, persisting in converging to the desired solution as long as the graph remains connected. In contrast, centralized bilevel algorithms inevitably break down when the central server crashes.
For these reasons, considerable research efforts, {\em e.g.}, \citep{chen2022decentralized,chen2023decentralized,dong2023single,gao2023convergence,liu2022interact,lu2022stochastic,niu2023distributed,yang2022decentralized,zhang2023communication}, have been dedicated to developing decentralized bilevel algorithms with theoretical guarantees and empirical effectiveness. 

\subsection{Limitations and Open Questions}
\label{sec:intro-limitation}

\textbf{Limitations in existing literature.} Significant progress has been made in developing decentralized stochastic bilevel algorithms, with most studies focusing on their \textit{asymptotic} convergence rates, which are achieved after a sufficiently large number of iterations. References \citep{lu2022stochastic,yang2022decentralized} demonstrate that decentralized bilevel algorithms attain the same asymptotic convergence rates as their centralized bilevel counterparts or decentralized single-level algorithms, highlighting their strong theoretical guarantees. However, an exclusive focus on asymptotic convergence rates presents several critical limitations:

\begin{itemize}
\item[L1.] In practical scenarios, computational constraints—such as limited time and resource budgets—often restrict the number of iterations that can be performed, preventing algorithms from reaching their asymptotic stage. It remains unknown the number of transient iterations required before the algorithm attains its asymptotic rate.

\item[L2.] Asymptotic convergence rates offer limited insight into the impact of nested algorithmic structures, network topologies, and data heterogeneity on decentralized bilevel algorithms, as these factors are typically associated with the algorithms' transient convergence stage rather than its asymptotic steady-state rate.

\item[L3.] Asymptotic convergence rates are inadequate as a metric for distinguishing between various decentralized bilevel algorithms, as most existing methods share the same asymptotic convergence rate~\citep{lu2022stochastic,yang2022decentralized}. This makes it challenging to determine which algorithm delivers superior theoretical performance. 
\end{itemize}

\noindent \textbf{Open questions.} To address these limitations and advance our understanding of decentralized bilevel optimization, several fundamental open questions need to be answered:
\begin{itemize}
    \item[Q1.] Can we quantify the transient iteration complexity of decentralized stochastic bilevel algorithms—that is, the number of iterations needed for the asymptotic convergence rate to dominate? Furthermore, can transient iteration complexity serve as a criterion to differentiate between various decentralized bilevel algorithms?

    \item[Q2.] Can we elucidate how network topologies and data heterogeneity influence transient iteration complexity? What benefits can we achieve by enhancing network connectivity and reducing data heterogeneity?

    \item[Q3.] Can we clarify how the nested bilevel algorithmic structure affects transient iteration complexity? Does the lower-level optimization introduce substantial challenges to the non-asymptotic phase?
\end{itemize}

\subsection{Main Results}
This paper proposes a new decentralized stochastic bilevel optimization method, focusing on its non-asymptotic convergence analysis and transient iteration complexity. This case study sheds light on the aforementioned open questions.

\begin{wrapfigure}{r}{0.45\textwidth}
\centering
\vspace{-13pt}
\includegraphics[width=0.38\textwidth]{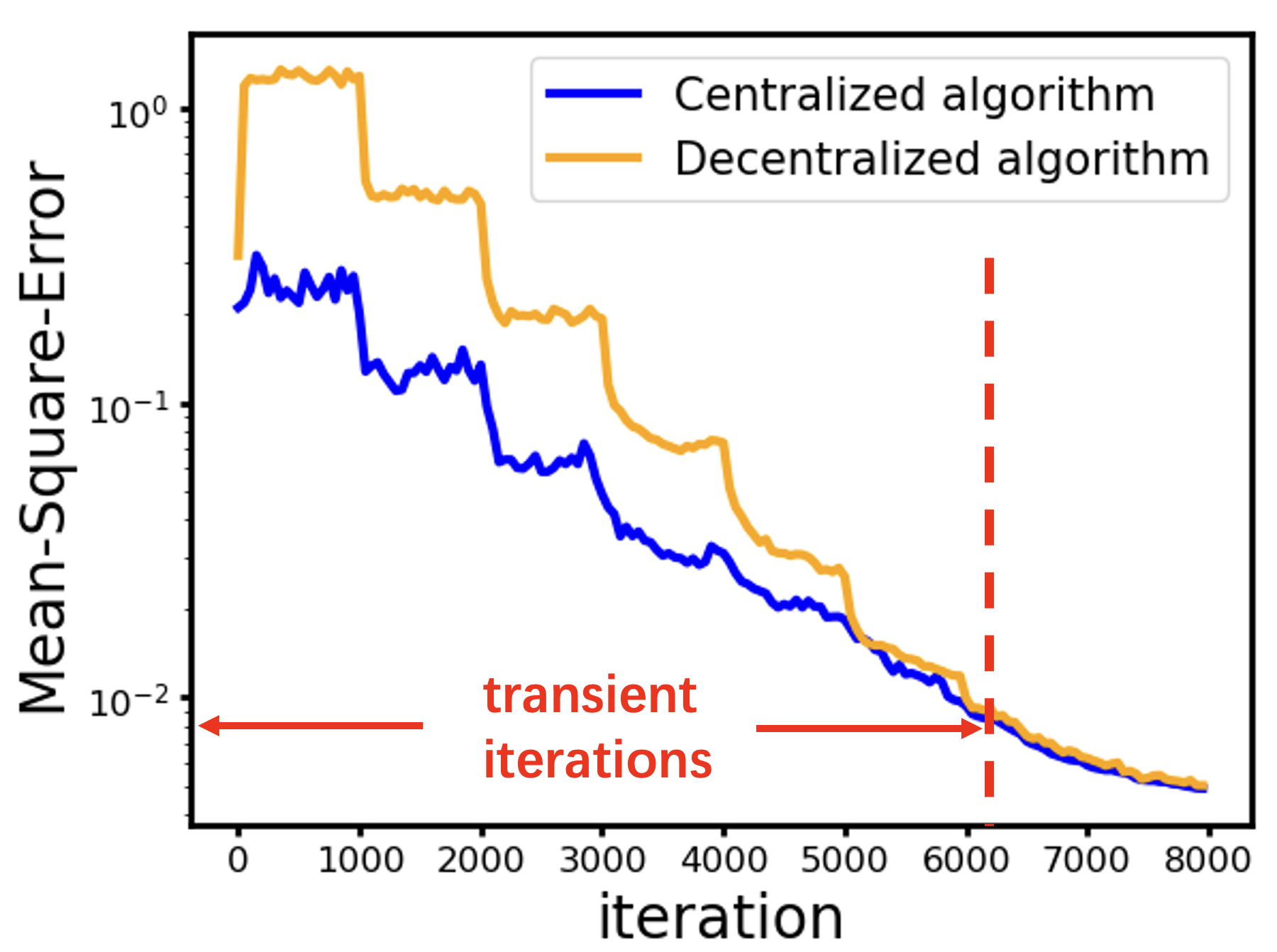}
\captionsetup{justification=raggedright}
\caption{\small Decentralized algorithm (with stage-wise decaying learning rate) has to experience sufficiently massive transient iterations to achieve the same asymptotic rate as centralized approach.}
 \vspace{-10pt}
\label{fig: tran_iter}
\end{wrapfigure}

\vspace{1mm}
\noindent \textbf{Transient iteration complexity.} Before presenting the main results, we will formally introduce the concept of transient iteration complexity. Although stochastic decentralized bilevel algorithms can asymptotically achieve the same linear speedup rate of $\mathcal{O}(1/\sqrt{NT})$ as their centralized counterparts~\citep{gao2023convergence,lu2022stochastic,yang2022decentralized}, they typically require more iterations to reach this stage due to inexact averaging in decentralized communication settings. \textit{Transient iteration complexity} refers to the number of iterations a decentralized algorithm undergoes before the asymptotic convergence rate becomes dominant, as illustrated in Fig.~\ref{fig: tran_iter}. A lower transient complexity indicates that the decentralized approach can more rapidly align with the performance of its centralized counterpart.

\begin{table*}[t!]
\small
\caption{\small Comparison between different decentralized stochastic bilevel algorithms. Notation $T$ indicates the number of (outer) iterations, $1-\rho \in (0, 1]$ measures the connectivity of the underlying graph, $N$ is the number of nodes, $G^2$ denotes gradient upper bound, and $b^2$ denotes the magnitude of gradient dissimilarity. We also list the result of single-level DSGD in the bottom line for reference.}
\vspace{1mm}
\centering
\begin{threeparttable}
\begin{tabular}{lccccccl}
\toprule
\small
{Algorithm} & \hspace{-0.4cm}{$\text{S.L.}^\blacklozenge$} & \hspace{-0.4cm}{L.S.$^\Diamond$} & \hspace{-0.4cm}{F.O.$^\heartsuit$} & \hspace{-0.32cm}{\begin{tabular}[c]{@{}c@{}}Asymptotic\\complexity$^\dagger$\end{tabular}} & {\begin{tabular}[c]{@{}c@{}}\hspace{-0.7cm}Transient \\ \hspace{-0.7cm}complexity$^\ddagger$\end{tabular}} & \multicolumn{1}{c}{\hspace{-0.3cm}{Assumption}$^\triangleleft$}                                                        \\ \midrule
\citep{chen2022decentralized}                & \hspace{-0.2cm}\xmark                                       & \hspace{-0.2cm}\xmark    & \hspace{-0.2cm}\xmark                                                & \hspace{-0.35cm}$\frac{1}{\varepsilon^3}$                          & \hspace{-0.7cm} N. A.                                                                    & \hspace{-0.3cm}LC $f_i$                                                             \vspace{1mm}      \\
\citep{chen2023decentralized}                 & \hspace{-0.2cm}\xmark                                       & \hspace{-0.2cm}\xmark                 & \hspace{-0.2cm}\xmark                                                & \hspace{-0.35cm}$\frac{1}{\varepsilon^2}\log(\frac{1}{\varepsilon})$                          & \hspace{-0.7cm} N. A.                                                                    & \hspace{-0.3cm}LC $f_i$                                                                 \vspace{1mm}    \\
\citep{lu2022stochastic}                 & \hspace{-0.2cm}\xmark                                       & \hspace{-0.2cm}\cmark      & \hspace{-0.2cm}\xmark                                              & \hspace{-0.35cm}$\frac{1}{N \varepsilon^2}\log(\frac{1}{\varepsilon})$                        & \hspace{-0.7cm} N. A.                                                                    & \hspace{-0.3cm}LC $f_i$                                                                   \vspace{1mm}   \\
\citep{yang2022decentralized}               & \hspace{-0.2cm}\xmark                                       & \hspace{-0.2cm}\cmark          & \hspace{-0.2cm}\xmark                                          & \hspace{-0.35cm}$\frac{1}{N \varepsilon^2}\log(\frac{1}{\varepsilon})$                        &\hspace{-0.7cm} $\frac{N^3 G^4}{(1-\rho)^4}^{\triangleright}$                                                 & \hspace{-0.3cm}BG $\nabla f_i$                                                                     \vspace{1mm}   \\
\citep[Thm 1]{gao2023convergence}$^\ast$                & \hspace{-0.2cm}\xmark                                       & \hspace{-0.2cm}\xmark         & \hspace{-0.2cm}\xmark                                     & \hspace{-0.35cm}$\frac{1}{(1-\rho)^2\varepsilon^2}\log(\frac{1}{\varepsilon})$                &\hspace{-0.7cm} N. A.                                                                     & \hspace{-0.3cm}BG $\nabla f_i$                                                                 \vspace{1mm}       \\
\citep[Thm 2]{gao2023convergence}$^\ast$                & \hspace{-0.2cm}\xmark                                       & \hspace{-0.2cm}\cmark         & \hspace{-0.2cm}\xmark                                           & \hspace{-0.35cm}$\frac{1}{N\varepsilon^2}\log(\frac{1}{\varepsilon})$                         &\hspace{-0.7cm}$\frac{N^3}{(1-\rho)^8}$                                                 & \hspace{-0.42cm} \begin{tabular}[l]{@{}l@{}}BG $\nabla f_i, \nabla g_i$ \\ BG $\nabla^2 g_i$\end{tabular}  \vspace{1mm}   \\
\citep{wang2024fully}                & \hspace{-0.2cm}\xmark                                       & \hspace{-0.2cm}\cmark          & \hspace{-0.2cm}\cmark                                          & \hspace{-0.35cm}$\frac{1}{(1-\rho^2)^{7/2}N \varepsilon^7}$                        &\hspace{-0.7cm} N. A.                                                 & \hspace{-0.3cm}LC $f_i\,^\clubsuit$                             \vspace{1mm}   \\
\rowcolor{green!30}\textbf{D-SOBA-SO(ours)}               & \hspace{-0.2cm}\cmark                                     & \hspace{-0.2cm}\cmark         & \hspace{-0.2cm}\xmark                                                   & \hspace{-0.35cm}$\boldsymbol{\frac{1}{N\varepsilon^2}}$                                                 & $\hspace{-0.9cm} \boldsymbol{\max\left\{ \frac{N^3}{(1-\rho)^2}, \frac{N^3b^2}{(1-\rho)^4}\right\}}$        & \hspace{-0.3cm}\begin{tabular}[c]{@{}l@{}}\textbf{BGD} $\nabla f_i, \nabla g_i$ \\ \textbf{BGD} $\nabla^2g_i$\end{tabular}       \vspace{-0.5mm}    
\\ \rowcolor{orange!30}\textbf{D-SOBA-FO(ours)}               & \hspace{-0.2cm}\cmark                                     & \hspace{-0.2cm}\cmark       & \hspace{-0.2cm}\cmark                                                     & \hspace{-0.35cm}$\boldsymbol{\frac{1}{N\varepsilon^2}}$                                                 & $\hspace{-0.9cm} \boldsymbol{\max\left\{ \frac{N^3}{(1-\rho)^2}, \frac{N^3b^2}{(1-\rho)^4}\right\}}$        & \hspace{-0.3cm}\begin{tabular}[c]{@{}l@{}}\textbf{BGD} $\nabla f_i, \nabla g_i$ \\ \textbf{BGD} $\nabla^2g_i$\end{tabular}       \vspace{-0.5mm}    
\\\midrule
\multicolumn{1}{l}{\begin{tabular}[c]{@{}l@{}}Single-level DSGD \\ \citep{chen2021accelerating}\end{tabular}} 
& \hspace{-0.2cm}\cmark                                     & \hspace{-0.2cm}\cmark                                     & \hspace{-0.2cm}N. A.                                                     & \hspace{-0.35cm}${\frac{1}{N\varepsilon^2}}$                                                 & $\hspace{-0.8cm} {\max\left\{ \frac{N^3}{(1-\rho)^2}, \frac{N^3b^2}{(1-\rho)^4}\right\}}$        & \hspace{-0.3cm}{BGD} $\nabla f_i$  \\ \bottomrule
\end{tabular}
\begin{tablenotes}
\footnotesize
    \item[$\blacklozenge$] The algorithm has the single-loop structure. 
    \item[$\Diamond$] The algorithm achieves {linear speedup}, i.e., it achieves the asymptotic rate of $1/{\sqrt{NT}}$.
    \item[$\heartsuit$] The algorithm only uses first-order gradients and does not need computation of the Hessian/Jacobian matrix or the Hessian/Jacobian-vector product.
    \item[$\dagger$] \#gradient/Hessian evaluations to achieve an $\varepsilon$-stationary solution when $\epsilon \to 0$ (smaller is better).
    \item[$\ddagger$] \#transient iterations an algorithm experiences before the asymptotic rate dominates (smaller is better).
    \item[$\triangleleft$]Additional assumptions beyond convexity, smoothness, and stochastic gradient. Lipschitz continuous (LC) function and bounded gradients (BG) are more restrictive than bounded gradient dissimilarity (BGD).

    \item[$\triangleright$] $G$ is the uniform upper bound of gradients assumed in \citep{yang2022decentralized}. It typically holds that $b \ll G$ where $b$ gauges the magnitude of the gradient dissimilarity, \ie, $\frac{1}{N}\sum_{i=1}^N\|\nabla f_i - \nabla f\|^2 \le b^2$.
    
\item[$\clubsuit$]The algorithm relies on the  Lipschitz continuity  assumption about $f_i$ with respect to $y$, and the 
 condition $\|\frac{1}{N}\sum_{i=1}^N\nabla_1f_i(x_i^{(t)},y_i^{(t)})+\alpha(\nabla_1g_i(x_i^{(t)},y_i^{(t)})-\nabla_1g_i(x_i^{(t)},v_i^{(t)}))\|^2=\mathcal{O}(\alpha^2)$, where $v_i^{(t)}$ is an intermediate variable.  In general cases, this condition essentially requires that each $\nabla_1 f_i$ is uniformly bounded.

\item[$\ast$]Asymptotic rate/transient complexity are not given in \citep{gao2023convergence}. We derive them in Appendix \ref{analysis of gao}. 
\end{tablenotes}
\end{threeparttable}
\label{table:comparison}
\end{table*}

\vspace{1mm}
\noindent \textbf{Contributions.} This paper addresses the above open questions through the results below:
\begin{itemize}
    \item[C1.] We propose \textbf{D-SOBA}, a \underline{\bf D}ecentralized \underline{\bf S}tochastic \underline{\bf O}ne-loop \underline{\bf B}ilevel \underline{\bf A}lgorithm framework. A key feature of D-SOBA is its elimination of the inner loop typically required for estimating the lower-level solution $y^\star(x)$ or the Hessian inverse of $g_i(x,y)$. The D-SOBA framework is versatile and can be implemented in two variants: \textbf{D-SOBA-SO} and \textbf{D-SOBA-FO}, which differ in their approaches to Hessian/Jacobian-vector product evaluations. {D-SOBA-SO} maintains \underline{\bf S}econd-\underline{\bf O}rder stochastic Hessian/Jacobian-vector products, while {D-SOBA-FO} performs hypergradient evaluation using fully \underline{\bf F}irst-\underline{\bf O}rder gradients through finite-difference techniques \citep{allen2018neon2,carmon2018accelerated}. By avoiding direct computation and storage of Hessian/Jacobian matrices or their vector products, D-SOBA-FO offers greater efficiency and practicality in high-dimensional applications such as meta-learning \citep{franceschi2018bilevel,rajeswaran2019meta,kayaalp2022dif,NEURIPS2022_1a82986c}. 
    
    \item[C2.] We provide a comprehensive non-asymptotic convergence analysis and derive the transient iteration complexity for the \ours framework. On one hand, we demonstrate that \ours achieves an asymptotic gradient complexity of $\mathcal{O}(1/(N\epsilon^2))$ where $\epsilon$ is the desired convergence accuracy, outperforming existing results by at least a factor of $\log(1/\epsilon)$. On the other hand, the transient iteration complexity allows us to clarify the theoretical superiority of \ours over other bilevel algorithms in the transient iteration regime. This addresses open question Q1. Notably, all convergence results for the \ours framework extend naturally to both D-SOBA-SO and D-SOBA-FO. 

    \item[C3.] Our transient iteration complexity is the {\em first} result in decentralized bilevel optimizaiton to quantify how network topology and data heterogeneity jointly affect the non-asymptotic convergence stage. It highlights that severe data heterogeneity exacerbates the impact of network topology, while poorly connected networks amplify the adverse effects of data heterogeneity. Neglecting either factor result in an incomplete transient-stage analysis. Additionally, we quantify the theoretical benefits achievable through enhanced network connectivity and reduced data heterogeneity, thereby addressing question Q2. All these results apply to both {D-SOBA-SO} and {D-SOBA-FO}.

    \item[C4.] Our comprehensive convergence analysis elucidates the impact of the nested bilevel algorithmic structure on transient iteration complexity. On one hand, we demonstrate that decentralized single-level and bilevel algorithms share the same asymptotic convergence rate and transient iteration complexity with respect to network topology and data heterogeneity. On the other hand, we clarify that the nested algorithmic structure significantly exacerbates the influence of the lower-level optimization's condition number. This highlights the challenges inherent to decentralized bilevel optimization, thereby addressing open question Q3.
\end{itemize}

\noindent \textbf{Comparison with existing algorithms.} All our established results, along with those of existing decentralized stochastic bilevel optimization algorithms, are summarized in Table \ref{table:comparison}. Notably, both D-SOBA-SO and D-SOBA-FO achieve an asymptotic gradient complexity of \(\mathcal{O}(1/(N\epsilon^2))\), surpassing existing methods by at least a factor of \(\log(1/\epsilon)\). In particular, the fully first-order decentralized bilevel algorithm proposed in \citep{wang2024fully} attains a complexity of \(\mathcal{O}(1/(N\epsilon^7))\), which is less efficient than our proposed D-SOBA-FO. Furthermore, the transient iteration complexity of the proposed algorithms is considerably better than that of \citep{yang2022decentralized} and \citep{gao2023convergence} when the data heterogeneity \( b^2 \) is trivial and the network topology is sparse (i.e., \( 1-\rho \to 0 \)). In addition, it is observed that our algorithms attain the same asymptotic and transient-iteration complexities as single-level DSGD \citep{chen2021accelerating} when the influence of the condition number can be neglected. Finally, Table~\ref{table:comparison} shows our results hold under more relaxed assumptions than those of existing methods.

\vspace{2mm} 
\noindent \textbf{Note.} This paper focuses on bilevel algorithms based on vanilla decentralized SGD~\citep{sayed2014adaptive,lian2017can}. Since decentralized SGD is inherently influenced by data heterogeneity, studying decentralized bilevel algorithms within this framework enables a quantitative analysis of how data heterogeneity impacts decentralized bilevel optimization. This establishes a foundation for understanding the quantitative benefits of reducing data heterogeneity in decentralized bilevel optimization. By contrast, skipping the analysis of decentralized bilevel SGD and directly delving into decentralized bilevel gradient tracking, as in references~\citep{zhang2023communication,dong2023single,wang2024fully}, has resulted in an inability to quantify the improvement in convergence rates achieved by removing data heterogeneity.

\subsection{Related Works}
\label{section:related}
\noindent \textbf{Bilevel optimization.} Bilevel optimization \citep{bracken1973mathematical} 
has shown extensive applications in operations research, signal processing, and machine learning \citep{colson2007overview, vicente1994bilevel, zhang2023introduction}. A central challenge in bilevel optimization is the estimation of the hypergradient $\nabla \Phi(x)$. {To address this, various algorithms have emerged, leveraging approaches such as approximate implicit differentiation \citep{domke2012generic,ghadimi2018approximation,grazzi2020iteration,ji2021bilevel}, iterative differentiation \citep{domke2012generic,franceschi2018bilevel,grazzi2020iteration,ji2021bilevel,maclaurin2015gradient}, and Neumann series \citep{chen2021closing,hong2023two,xiao2022alternating,li2022fully}.} However, these approaches incur additional inner-loops to estimate the lower-level solution $y^\star(x)$ and the Hessian inverse of $g(x,y)$,  resulting in extra computational overhead and worsening the iteration complexity. A recent work \citep{dagreou2022framework} develops a novel single-loop framework for stochastic bilevel optimization, in which the Hessian inverse is removed from the algorithmic structure and the lower and upper variables are updated simultaneously. 
Moreover, first-order algorithms for bilevel optimization have gained an increasing focus to solve large-scale optimization tasks. An effective approach is to view the lower-level loss as a penalty and add it to the upper-level loss \citep{liu2022bome,kwon2023fully,lu2024first}. Moreover, \cite{yang2023achieving} also uses finite-difference of stochastic gradients to evaluate Hessian/Jacobian, reducing computation and storage cost especially in large-scaled scenarios.

\vspace{2mm}
\noindent \textbf{Decentralized optimization.} Decentralized optimization is useful in situations where the centralized control of all nodes by a single server is either practically infeasible or prohibitively expensive.
Early well-known algorithms include decentralized gradient descent \citep{nedic2009distributed,yuan2016convergence}, diffusion strategies \citep{chen2012diffusion}, dual averaging \citep{duchi2011dual}, EXTRA \citep{shi2015extra}, Exact-Diffusion \citep{li2017decentralized,yuan2017exact1}, gradient tracking \citep{di2016next,nedic2017achieving,xu2015augmented}, and decentralized ADMM methods \citep{chang2014multi,shi2014linear}. In the stochastic context, decentralized SGD is established in \citep{lian2017can} to achieve the same asymptotic linear speedup as centralized SGD.  Since then, many efforts have extended decentralized SGD to directed topologies \citep{nedic2014distributed}, time-varying topologies \citep{koloskova2020unified,nedic2014distributed}, and data-heterogeneous scenarios \citep{lin2021quasi,lu2019gnsd,tang2018d,xin2020improved,yuan2023removing,yuan2021decentlam}. Lower bounds and optimal complexities are also recently established for stochastic decentralized optimization \citep{lu2021optimal,yuan2022revisiting}. 

\vspace{2mm}
\noindent \textbf{Decentralized stochastic bilevel optimization (SBO).} 
Decentralized SBO algorithms are studied in \citep{chen2022decentralized,chen2023decentralized,gao2023convergence,lu2022stochastic,yang2022decentralized,wang2024fully} with solid theoretical guarantees and strong empirical performance. Nevertheless, these algorithms entail computationally expensive inner-loop updates and often exhibit a limited focus on the non-asymptotic convergence stage in their analyses, as discussed in Sec.~\ref{sec:intro-limitation}. Table \ref{table:comparison} presents a comparison of our results with existing works in terms of both asymptotic and transient iteration complexities. Specifically, \cite{wang2024fully} achieve the hyper-gradient evaluation only by first-order gradients.  Meanwhile, \citep{lu2022decentralized,niu2023distributed} are also proposed to solve the personalized bilevel problem:
\begin{equation}\label{prob:personalized}
\begin{aligned}
    \min_{x \in \RR^d} \quad  \frac{1}{N}\sum_{i=1}^N f_i\left(x, y_i^\star(x)\right), \quad\quad
     \vspace{-1mm} \mathrm{s.t.} \quad y_i^\star(x) := \argmin_{y \in \RR^p}   g_i(x,y) , \quad  {\forall\,1\leq  i \leq N}.
\end{aligned}
\end{equation}
where each node has a personalized lower-level cost function. This differs from problem \eqref{prob:general}, where both the upper- and lower-level cost functions are globally averaged.

\vspace{2mm}
\noindent \textbf{Concurrent works.} Recently, a single-loop decentralized stochastic bilevel optimization method, SPARKLE, was introduced in \citep{anonymous2024sparkle}. Building on our findings on the impact of data heterogeneity\footnote{The main results of this paper were first presented in arXiv reports \citep{kong2024decentralized}.}, SPARKLE quantifies the performance gains achieved by mitigating data heterogeneity through gradient tracking. Another recent work \citep{wen2024communication} proposed a fully first-order algorithm for decentralized \emph{deterministic} bilevel optimization, successfully avoiding Hessian/Jacobian-vector product computations. However, this algorithm relies on a stronger assumption of the Lipschitz continuity of the Hessian matrices \(\nabla^2 f_i\) and \(\nabla^2 g_i\) and does not clarify performance in stochastic settings.

\section{Preliminaries}
\label{subsection:Notations and Assumptions}
\subsection{Notations}
For a second-order differentiable function \( f:\mathbb{R}^{d}\times\mathbb{R}^{p}\to \mathbb{R} \), we denote its partial gradients as \( \nabla_1 f(x, y) \in \mathbb{R}^d \) and \( \nabla_2 f(x, y) \in \mathbb{R}^p \) with respect to \( x \) and \( y \), respectively. The corresponding partial Hessians are \( \nabla_{12}^2 f(x, y) \in \mathbb{R}^{d \times p} \) and \( \nabla_{22}^2 f(x, y)\in \mathbb{R}^{p\times p} \). Additionally, we use \( \nabla_x f(x,y^{\star}(x)) \) to denote the gradient of \( f \) with respect to \( x \), treating \( y \) as a function of \( x \). We let \( \|\cdot\| \) denote the \(\ell_2\) norm for both vectors and matrices, and \( \|\cdot\|_F \) the Frobenius norm for matrices. The vector \( \one_N\in\mathbb{R}^N \) represents an all-ones vector. For local variables \( \{x_i^{(t)}\}_{i=1}^N \), the subscript \( i \) denotes the node index, while the superscript \( t \) indicates the iteration. Their average, \( \sum_{i=1}^N x_{i}^{(t)}/N \), is denoted as \( \bar{x}^{(t)} \). Finally, we write \( a\lesssim b \) if there exists a constant \( C > 0 \) such that \( a\leq Cb \).  

\subsection{Assumptions}
With the notations introduced above, we next state the assumptions used in the paper.
\begin{assumption}[\sc Smoothness]
\label{assumption:smooth}
There exist positive constants $L_{\nabla f}$, $L_{\nabla g}$, $L_{\nabla^2 g}$, $L_f$ and $\mu_g$ such that  for any $1\le i \le N$,

1. $\nabla f_i$,$\nabla g_i$,$\nabla^2g_i$ are $L_{\nabla f}$, $L_{\nabla g}$, $L_{\nabla^2 g}$ Lipschitz continuous,  respectively;

2. $\left\Vert\nabla_2f_i(x,y^\star(x))\right\Vert\leq L_f<\infty$ for all $x\in\mathbb{R}^d$ in which $y^\star(x)$ is defined in problem \eqref{eq:lower};

3. $g_i(x,\cdot)$ is $\mu_g$-strongly convex for any given $x\in \mathbb{R}^d$. 

\end{assumption}
\noindent It is noteworthy that {the second condition} of Assumption \ref{assumption:smooth} relaxes the restrictive assumptions of Lipschitz continuity of $f$ or, equivalently, the boundedness of $\nabla_2 f$ used {in~\citep{gao2023convergence, lu2022stochastic}}.

Due to the heterogeneity of local data distributions, the local functions $\{(f_i, g_i)\}_{i=1}^N$ are non-identical across different nodes. As this paper focuses on bilevel optimization algorithms built upon vanilla decentralized SGD, we bound gradient dissimilarity using the following standard assumption~\citep{koloskova2020unified,lian2017can}:
\begin{assumption}[\sc Gradient dissimilarity]
\label{assumption: data heterogeneity}
There exist constants  ${b_1, b_2}\geq0$ such that for all $(x,y)\in\mathbb{R}^{d}\times\mathbb{R}^{p}$, it holds that:
\begin{subequations}
\begin{align*}
    &\dfrac{1}{N}\sum_{i=1}^N\left\Vert \nabla_1f_i(x,y)-\nabla_1f(x,y)\right\Vert^2\leq {b_1^2},\quad\hspace{2pt}\dfrac{1}{N}\sum_{i=1}^N\left\Vert \nabla_2f_i(x,y)-\nabla_2f(x,y)\right\Vert^2\leq {b_1^2};\\
    &\dfrac{1}{N}\sum_{i=1}^N\left\Vert \nabla_2g_i(x,y)-\nabla_2g(x,y)\right\Vert^2\leq {b_2^2},\quad\hspace{2pt}\dfrac{1}{N}\sum_{i=1}^N\left\Vert \nabla_{12}^2g_i(x,y)-\nabla_{12}^2g(x,y)\right\Vert^2\leq {b_2^2},\\
    &\dfrac{1}{N}\sum_{i=1}^N\left\Vert \nabla_{22}^2g_i(x,y)-\nabla_{22}^2g(x,y)\right\Vert^2\leq {b_2^2}.
\end{align*}
\end{subequations}
\end{assumption}
\noindent 
We also make the following standard assumption for stochastic gradients and Hessians.

\begin{assumption}[\sc Stochasticity]
\label{assumption:unbiased}
There exists a constant $\sigma\geq 0$ such that  for any given $(x,y)\in\RR^{d}\times \RR^p$ and $1\leq i\leq N$:

\vspace{1mm}
\noindent 1. the gradient oracles satisfy:      
\begin{align*}
&\EE_{\xi_i\sim\cD_{f_i}}[\nabla_1 F(x,y;\xi_i)]=\nabla_1 f_i(x,y),\quad\EE_{\xi_i\sim\cD_{f_i}}\left[\left\|\nabla_1 F(x,y;\xi_i)-\nabla_1 f_i(x,y)\right\|^2\right]\le\sigma^2,\\
&\EE_{\xi_i\sim\cD_{f_i}}[\nabla_2 F(x,y;\xi_i)]=\nabla_2 f_i(x,y),\quad\EE_{\xi_i\sim\cD_{f_i}}\left[\left\|\nabla_2 F(x,y;\xi_i)-\nabla_2 f_i(x,y)\right\|^2\right]\le\sigma^2,\\
&\EE_{\zeta_i\sim\cD_{g_i}}[\nabla_2 G(x,y;\zeta_i)]=\nabla_2 g_i(x,y),\quad
\EE_{\zeta_i\sim\cD_{g_i}}\left[\left\|\nabla_2 G(x,y;\zeta_i)-\nabla_2 g_i(x,y)\right\|^2\right]\le\sigma^2;
\end{align*}
2. the Hessian/Jacobian oracles satisfy:
\begin{align*}
\hspace{-5mm}
&\EE_{\zeta_i\sim\cD_{g_i}}\left[\nabla_{12}^2G(x,y;\zeta_i)\right]=\nabla_{12}^2g_i(x,y),\quad
\EE_{\zeta_i\sim\cD_{g_i}}\left[\left\|\nabla_{12}^2G(x,y;\zeta_i)-\nabla_{12}^2g_i(x,y)\right\|^2\right]\le\sigma^2,\\
&\EE_{\zeta_i\sim\cD_{g_i}}\left[\nabla_{22}^2G(x,y;\zeta_i)\right]=\nabla_{22}^2g_i(x,y),\quad
\EE_{\zeta_i\sim\cD_{g_i}}\left[\left\|\nabla_{22}^2G(x,y;\zeta_i)-\nabla_{22}^2g_i(x,y)\right\|^2\right]\le\sigma^2.
\end{align*}
\end{assumption}

This work studies decentralized algorithms over networks of $N$ nodes interconnected by a graph $\mathcal{G}$ with a set of edges $\mathcal{E}$. Node $i$ is connected to node $j$ if $(i,j)\in \mathcal{E}$. To facilitate decentralized communication, we introduce the mixing matrix $W=[w_{i,j}]_{i,j=1}^N\in\RR^{N \times N}$ in which each weight $w_{i,j}$ gauges information flowing from node $j$ to node $i$. Furthermore, we set $w_{i,j}=0$ for $(j,i)\notin \mathcal{E}$. The following assumption on the mixing matrix is widely used for decentralized algorithms~\citep{koloskova2020unified,lian2017can,yuan2016convergence}.
\begin{assumption}[\sc Mixing matrix]
\label{assumption: gossip communication}
The mixing matrix $W$ is doubly stochastic, \ie, $$\one_{N}^\top W=\one_{N}^\top , \quad W \one_{N}=\one_{N}.$$
Moreover, we assume  $\rho:=\left\Vert W-\one_N \one_N^\top/ N\right\Vert_2\in[0,1)$.
\end{assumption}
\begin{remark}[\sc spectral gap]
In decentralized algorithms, the quantity $1-\rho$ is commonly known as the {\it spectral gap}~\citep{lu2021optimal, yuan2023removing} of $W$, which serves as a metric for measuring the connectivity of the network topology. Notably, as $1-\rho \to 1$, it indicates that the topology is well-connected (\eg, for a fully connected graph, the mixing matrix is $W=\one_N \one_{N}^\top/N$ with $\rho=0)$. Conversely, as $1-\rho\to 0$, it suggests that the topology is potentially sparse \citep{lian2017can, ying2021exponential}.
\end{remark}

\begin{remark}[\sc Non-symmetric mixing matrix]
While Assumption~\ref{assumption: gossip communication} requires the mixing matrix \( W \) to be doubly stochastic, it does not need to be symmetric. This contrasts with the recent work \citep{anonymous2024sparkle}, where the SPARKLE algorithm is proposed and analyzed under the assumption of a symmetric and doubly-stochastic mixing matrix.
\end{remark}

\section{Decentralized Stochastic One-loop Bilevel Algorithm (D-SOBA)}
\label{section:ours}
\noindent \textbf{Major challenges of decentralized SBO.} In this section, we present our \ours algorithm for decentralized SBO. The primary challenge in SBO stems from the estimation of the hypergradient $\nabla \Phi(x)$, \ie, $\nabla_x f(x, y^\star(x))$, due to the implicit dependence of $y^\star(x)$ on $x$. Under Assumption~\ref{assumption:smooth} and the principles of implicit function theory~\citep{griewank2008evaluating}, $\nabla \Phi(x)$ can be expressed as:

\begin{equation}
\label{eq: grad phix}
 \nabla \Phi(x)=\   \nabla_1 f(x,y^\star(x)) - \Big( \nabla^2_{12} g(x,y^\star(x))\ \cdot \nonumber
   \left[\nabla^2_{22} g(x,y^\star(x))\right]^{-1} \cdot \nabla_2 f(x,y^\star(x)) \Big) 
\end{equation}
which is computationally expensive due to the need for inverting the partial Hessian. Moreover, the Hessian-inversion 
$
\left[\nabla_{22}^2 g(x,y^\star(x))\right]^{-1} = \left[\frac{1}{N}\sum_{i=1}^N \nabla_{22}^2~ g_i(x,y^\star(x))\right]^{-1}
$
cannot be easily accessed through decentralized communication even in the absence of stochastic noise affecting the estimate of $\nabla_{22}^2~ g_i(x,y^\star(x))$. This challenge can be partially mitigated by incorporating auxiliary inner loops to approximate the lower-level solution $y^\star(x)$ and evaluate the Hessian inversion using the Neumann series, as demonstrated in works such as \citep{chen2022decentralized, chen2023decentralized, lu2022stochastic, yang2022decentralized}. However, introducing these auxiliary inner loops results in sub-optimal convergence, as shown in Table~\ref{table:comparison}. Moreover, it complicates algorithmic implementation and may impose a substantial burden on communication and computation.

\vspace{2mm}
\noindent
\textbf{Centralized SOBA.} The challenge of Hessian inversion can be effectively addressed by a novel framework known as SOBA.
SOBA, initially proposed in \citep{dagreou2022framework}, is a single-node algorithm. We now extend it to solve the distributed problem \eqref{prob:general} in the centralized setup. To begin with, SOBA introduces
\begin{align}
\label{eq:z-star} 
z^\star(x) 
=  \left[\nabla^2_{22} g(x,y^\star(x))\right]^{-1} \nabla_2 f(x,y^\star(x)) 
\end{align}
which can be regarded as the solution to minimizing the following optimization problem 
\begin{align}
\label{eq:soba}
\!\frac{1}{2} z^\top \nabla_{22}^2 g \left(x, y^\star\right)z - z^{\top} \nabla_{2}f \left(x, y^\star\right) \overset{\eqref{prob:general}}{=} \frac{1}{N} \sum_{i=1}^N \left\{ \frac{1}{2} z^\top \nabla_{22}^2 g_i \left(x, y^\star\right)z - z^{\top} \nabla_{2}f_i \left(x, y^\star\right) \right\},\!
\end{align}
where $y^\star$ denotes $y^\star(x)$. It is worth noting that while $z^\star(x)$ in \eqref{eq:z-star} cannot be written as a finite sum across nodes, problem \eqref{eq:soba} involves only simple sums. 

To save computations, we can approximately solve \eqref{eq:soba} using one-step (stochastic) gradient descent. This, combined with one-step (stochastic) gradient descent to update the upper- and lower-level variables $(x,y)$, forms the centralized single-loop framework for solving \eqref{prob:general}:
\begin{align}
\label{eq:c-soba}
\underbrace{x^{(t+1)} = x^{(t)} -  \frac{\alpha_t}{N}\sum_{i=1}^N D_{x,i}^{(t)}}_{\rm minimize \;\eqref{eq:upper}},  \quad
\underbrace{y^{(t+1)} = y^{(t)} -  \frac{\beta_t}{N}\sum_{i=1}^N D_{y,i}^{(t)}}_{\rm minimize \;\eqref{eq:lower}},  \quad
\underbrace{z^{(t+1)} = z^{(t)} -  \frac{\gamma_t}{N}\sum_{i=1}^N D_{z,i}^{(t)}}_{\rm minimize \;\eqref{eq:soba}}, 
\end{align}
where $D_{x,i}^{(t)}, D_{y,i}^{(t)}$ and $D_{z,i}^{(t)}$ are defined as follows: 
\begin{subequations}
\label{eq-defi-D}
\begin{align}
D^{(t)}_{x,i}(x^{(t)},y^{(t)},z^{(t)}) &= \nabla_1 f_i(x^{(t)},y^{(t)}) - \nabla^2_{12}g_i(x^{(t)},y^{(t)})z^{(t)},  \\
D^{(t)}_{y,i}(x^{(t)},y^{(t)},z^{(t)}) &= \nabla_{2}g_i(x^{(t)},y^{(t)}), \\
D^{(t)}_{z,i}(x^{(t)},y^{(t)},z^{(t)}) &= 
\nabla^2_{22}g_i(x^{(t)},y^{(t)})z^{(t)} - \nabla_2 f_i(x^{(t)},y^{(t)}),
\end{align} 
\end{subequations}
and $\alpha_t$, $\beta_t$ and $\gamma_t$ are learning rates. We denote recursion \eqref{eq:c-soba} as centralized SOBA since a central server is required to collect $D_{x,i}^{(t)}, D_{y,i}^{(t)}$, and $D_{z,i}^{(t)}$ across the entire network, as well as update variables $x^{(t)}, y^{(t)}$ and $z^{(t)}$. It is noteworthy that \eqref{eq:c-soba} does not require any inner loop to approximate the lower-level solution $y^\star(x)$ or evaluate the Hessian inversion of $g(x,y)$.

\vspace{2mm}
\noindent \textbf{Decentralized SOBA (D-SOBA) Framework.} 
\label{sec-sub:d-soba}
Inspired by \citep{chen2012diffusion,nedic2009distributed}, we  extend centralized SOBA \eqref{eq:c-soba} to decentralized setup: 
\begin{subequations}
\label{eq:d-soba}
\begin{align}
x_i^{(t+1)} &= \sum_{j\in \cN_i} w_{ij}(x_j^{(t)} - \alpha_t D_{x,j}^{(t)}  ), \\
y_i^{(t+1)} &= \sum_{j\in \cN_i} w_{ij}(y_j^{(t)} - \beta_t D_{y,j}^{(t)}  ),\\
z_i^{(t+1)} &= \sum_{j\in \cN_i} w_{ij}(z_j^{(t)} - \gamma_t D_{z,j}^{(t)}  ),  
\end{align} 
\end{subequations}
where \( x_i^{(t)} \), \( y_i^{(t)} \), and \( z_i^{(t)} \) are local variables maintained by node \( i \) at iteration \( t \), and \( D_{x,i}^{(t)} \), \( D_{y,i}^{(t)} \), and \( D_{z,i}^{(t)} \) are simplified notations for \( D_{x,i}^{(t)}(x^{(t)},y^{(t)},z^{(t)}) \), \( D_{y,i}^{(t)}(x^{(t)},y^{(t)},z^{(t)}) \), and \( D_{z,i}^{(t)}(x^{(t)},y^{(t)},z^{(t)}) \), respectively, as defined in \eqref{eq-defi-D}. The set $\mathcal{N}_i$ includes node $i$ and all its immediate neighbors. We refer to recursions \eqref{eq:d-soba} as the D-SOBA framework, which forms the foundation of the D-SOBA-SO and D-SOBA-FO algorithms.

\vspace{2mm}
\noindent \textbf{D-SOBA-SO Algorithm.} The quantities \( D_{x,i}^{(t)} \), \( D_{y,i}^{(t)} \), and \( D_{z,i}^{(t)} \) in the decentralized SOBA framework \eqref{eq:d-soba} involve second-order matrices such as \( \nabla^2_{12} g_i(x^{(t)},y^{(t)}) \) and \( \nabla^2_{22} g_i(x^{(t)},y^{(t)}) \). When these second-order matrices are explicitly maintained in the algorithm, the resulting method is referred to as D-SOBA-SO. In stochastic settings, D-SOBA-SO samples stochastic estimates of the gradients, Jacobians, and Hessians in \eqref{eq-defi-D}. In particular, suppose D-SOBA-SO independently samples a minibatch of data, \( \xi^{(t)}_i \sim \mathcal{D}_{f_i} \) and \( \zeta^{(t)}_i \sim \mathcal{D}_{g_i} \), it approximates 
\begin{subequations}
\label{eq-defi-D-hat}
\begin{align}
{D}^{(t)}_{x,i}(x_i^{(t)},y_i^{(t)},z_i^{(t)}) &\approx \nabla_1 F(x_{i}^{(t)},y_{i}^{(t)};\xi^{(t)}_i) - \nabla^2_{12} G(x_{i}^{(t)},y_{i}^{(t)};\zeta^{(t)}_i) z_i^{(t)},  \\
D^{(t)}_{y,i}(x_i^{(t)},y_i^{(t)},z_i^{(t)}) &\approx \nabla_2 G(x_{i}^{(t)},y_{i}^{(t)};\zeta^{(t)}_i), \\
D^{(t)}_{z,i}(x_i^{(t)},y_i^{(t)},z_i^{(t)}) &\approx 
\nabla^2_{22} G(x_{i}^{(t)},y_{i}^{(t)};\zeta^{(t)}_i)z_i^{(t)} - \nabla_2 F(x_{i}^{(t)},y_{i}^{(t)};\xi^{(t)}_i).
\end{align} 
\end{subequations}
The detailed implementation of \colorbox{green!30}{D-SOBA-SO} is listed in Algorithm~\ref{alg:DeMA-SOBA}.  Furthermore, we impose a moving average on the update of $x_i$ in Algorithm \ref{alg:DeMA-SOBA}. As shown in \citep{chen2023optimal}, the moving average step enables a more stable and finer-grained direction for estimating hypergradients. This strategy plays a crucial role in reducing the order of bias from sample noise in the convergence analysis, as well as relaxing the technical assumptions. For notational simplicity in the proof, we introduce a scaling parameter \( \tau \), such that the upper-level step size is given by \( \tau\alpha \).

\begin{algorithm}[t!]
    \caption{\colorbox{green!30}{D-SOBA-SO} and \colorbox{orange!30}{D-SOBA-FO}}
    \label{alg:DeMA-SOBA}
    \begin{algorithmic}
    \REQUIRE{{Initialize $x_i^{(0)}\in\RR^d$, $y_i^{(0)}\in\RR^p$, $z_i^{(0)}\in\RR^p$, and $h_i^{(0)}\in\RR^d$ for any $i=1,2,\cdots,N$}, the step size $\{\alpha_t\}_{t=0}^{T-1},\{\beta_t\}_{t=0}^{T-1},\{\gamma_t\}_{t=0}^{T-1},\{\theta_t\}_{t=0}^{T-1}$, $\tau$, the mixing matrix $W$, \colorbox{orange!20}{$\{\delta_t\}_{t=0}^{T-1}$}}.
    \FOR{$t=0,1,\cdots,T-1$}  
    \FOR{each node $i=1,2,\cdots,N$ in parallel} 
    \STATE Get stochastic sample $\xi_i^{(t)}\sim\mathcal{D}_{f_i}$ and $\zeta_i^{(t)}\sim\mathcal{D}_{g_i}$.
    \STATE \colorbox{green!30}{$p_{H,i}^{(t)}:=\nabla^2_{22}G(x_{i}^{(t)},y_{i}^{(t)};\zeta^{(t)}_i)z_{i}^{(t)}$;} 
    \STATE \colorbox{green!30}{$p_{J,i}^{(t)}:=\nabla^2_{12}G(x_{i}^{(t)},y_{i}^{(t)};\zeta^{(t)}_i)z_{i}^{(t)}$.} 
    \STATE  \colorbox{orange!30}{$p_{H,i}^{(t)}:=\dfrac{1}{2\delta_t}\left(\nabla_2G(x_i^{(t)},y_i^{(t)}+\delta_tz_i^{(t)},\zeta_i^{(t)})-\nabla_2G(x_i^{(t)},y_i^{(t)}-\delta_tz_i^{(t)},\zeta_i^{(t)})\right)$;}
    \STATE \colorbox{orange!30}{$p_{J,i}^{(t)}:=\dfrac{1}{2\delta_t}\left(\nabla_1G(x_i^{(t)},y_i^{(t)}+\delta_tz_i^{(t)},\zeta_i^{(t)})-\nabla_1G(x_i^{(t)},y_i^{(t)}-\delta_tz_i^{(t)},\zeta_i^{(t)})\right)$.} 
    \STATE $x_{i}^{(t+1)}:=\sum_{j\in \cN_i} w_{ij}(x_{j}^{(t)}-\tau\alpha_th_{j}^{(t)})$;
    \STATE $y_{i}^{(t+1)}:=\sum_{j\in \cN_i}w_{ij}\left(y_{j}^{(t)}-\beta_t\nabla_{2}G(x_{j}^{(t)},y_{j}^{(t)};\zeta^{(t)}_j)\right)$;
    \STATE $z_{i}^{(t+1)}:=\sum_{j\in \cN_i}w_{ij}\left(z_{j}^{(t)}-\gamma_t\left(p_{H,j}^{(t)}-\nabla_2F(x_{j}^{(t)},y_{j}^{(t)};\xi^{(t)}_j)\right)\right)$;
    \STATE $\omega_{i}^{(t+1)}:=\nabla_1F(x_{i}^{(t)},y_{i}^{(t)};\xi^{(t)}_i)-p_{J,i}^{(t)}$;
    \STATE $h_{i}^{(t+1)}:=(1-\theta_t)h_{i}^{(t)}+\theta_t\omega_{i}^{(t+1)}$.
    \ENDFOR 
    \ENDFOR 
    \end{algorithmic}
\end{algorithm}

\vspace{2mm}
\noindent \textbf{D-SOBA-FO Algorithm.} In high-dimensional bilevel optimization, storing and computing the second-order matrices \( \nabla^2_{12} g_i(x^{(t)},y^{(t)}) \) and \( \nabla^2_{22} g_i(x^{(t)},y^{(t)}) \) is highly expensive, particularly in large-scale machine learning applications, such as data reweighting in large language models \citep{pan2024scalebio}. Finite difference approximation is an efficient approach for estimating higher-order gradients using lower-order information \citep{allen2018neon2,carmon2018accelerated,yang2023achieving}. Here, we introduce finite difference Hessian/Jacobian-vector product approximations to D-SOBA. Given the Lipschitz continuity of \( \nabla_{12}^2 g_i \) under Assumption \ref{assumption:smooth}, we obtain the following approximation for some perturbation \( \delta_t > 0 \):
\begin{subequations}
\label{eq-first-order}
\begin{align}
\nabla_{1}g_i(x_i^{(t)},y_i^{(t)}+\delta_t z_i^{(t)})=\nabla_{1}g_i(x_i^{(t)},y_i^{(t)})+\delta_t\nabla_{12}^2g_i(x_i^{(t)},y_i^{(t)})z_i^{(t)}+\mathcal{O}\left(\delta_t^2\Vert z_i^{(t)}\Vert^2\right),\\
\nabla_{1}g_i(x_i^{(t)},y_i^{(t)}-\delta_t z_i^{(t)})=\nabla_{1}g_i(x_i^{(t)},y_i^{(t)})-\delta_t\nabla_{12}^2g_i(x_i^{(t)},y_i^{(t)})z_i^{(t)}+\mathcal{O}\left(\delta_t^2\Vert z_i^{(t)}\Vert^2\right).
\end{align}
\end{subequations}
Thus, the term \( (2\delta_t)^{-1} (\nabla_{1} g_i(x_i^{(t)},y_i^{(t)}+\delta_t z_i^{(t)}) - \nabla_{1} g_i(x_i^{(t)},y_i^{(t)}-\delta_t z_i^{(t)})) \) provides an approximation of the Jacobian-vector product \( \nabla_{12}^2 g_i(x_i^{(t)},y_i^{(t)}) z_i^{(t)} \), particularly when \( \delta_t \) is small:
\begin{align*}
&\dfrac{1}{2\delta_t}(\nabla_{1}g_i(x_i^{(t)},y_i^{(t)}+\delta_t z_i^{(t)})-\nabla_{1}g_i(x_i^{(t)},y_i^{(t)}-\delta_t z_i^{(t)}))\\
=&\nabla_{12}^2g_i(x_i^{(t)},y_i^{(t)})z_i^{(t)}+\mathcal{O}\left(\delta_t\Vert z_i^{(t)}\Vert^2\right)\approx\nabla_{12}^2g_i(x_i^{(t)},y_i^{(t)})z_i^{(t)}.
\end{align*}
Similarly, $\nabla_{22}^2g_i(x_i^{(t)},y_i^{(t)})z_i^{(t)}$ can also be evaluated by first-order gradients as follows:
\begin{align*}
&\dfrac{1}{2\delta_t}(\nabla_{2}g_i(x_i^{(t)},y_i^{(t)}+\delta_t z_i^{(t)})-\nabla_{2}g_i(x_i^{(t)},y_i^{(t)}-\delta_t z_i^{(t)}))\\
=&\nabla_{22}^2g_i(x_i^{(t)},y_i^{(t)})z_i^{(t)}+\mathcal{O}\left(\delta_t\Vert z_i^{(t)}\Vert^2\right)\approx\nabla_{22}^2g_i(x_i^{(t)},y_i^{(t)})z_i^{(t)}.
\end{align*}
By replacing the true gradients in the above expressions with stochastic approximations, we introduce \colorbox{orange!30}{D-SOBA-FO}, which enables Hessian/Jacobian evaluation while requiring only first-order gradients. See Algorithm \ref{alg:DeMA-SOBA} for details.

\section{Theoretical Analysis}\label{section:Convergence Analysis} 
Here we establishes the convergence guarantees for {D-SOBA-SO} and {D-SOBA-FO}. We define the notation \(\mathcal{F}^{(t)} = \sigma\left[\bigcup_{\tau=0}^t \left(\bigcup_{i=1}^N \{x^{(\tau)}_i, y^{(\tau)}_i, z^{(\tau)}_i, h^{(\tau)}_i\} \cup \{\alpha_{\tau}, \beta_{\tau}, \gamma_{\tau}, \theta_{\tau}, \delta_{\tau}\} \right)\right]\) to represent the filtration of the $t$-th iteration, i.e., the \(\sigma\)-field generated by all elements with superscripts within the first \(t\) iterations. Additionally, we denote the conditional expectation with respect to \(\mathcal{F}^{(t)}\) as \(\mathbb{E}_t[\cdot] := \mathbb{E}[\cdot | \mathcal{F}^{(t)}]\). The following lemma quantifies the error introduced by the finite-difference approximation in {D-SOBA-FO}:
\begin{proposition}
\label{error of finite} 
Suppose Assumptions \ref{assumption:smooth} and \ref{assumption:unbiased} hold. Then, the terms \( p_{H,i}^{(t)} \) and \( p_{J,i}^{(t)} \) obtained from D-SOBA-FO satisfy (see proof in Lemma \ref{desjifen}):
\begin{subequations}
    \begin{align}
        \label{chafen_H}
        \mathbb{E}_t\left\Vert p_{H,i}^{(t)}-\mathbb{E}_t[p_{H,i}^{(t)}]\right\Vert^2\leq\sigma^2\left\Vert z_i^{(t)}\right\Vert^2,\quad\left\Vert\mathbb{E}_t[p_{H,i}^{(t)}]-\nabla_{22}g_i(x_i^{(t)},y_i^{(t)})z_i^{(t)}\right\Vert^2\leq\iota^2;\\
        \label{chafen_J}
        \mathbb{E}_t\left\Vert p_{J,i}^{(t)}-\mathbb{E}_t[p_{J,i}^{(t)}]\right\Vert^2\leq\sigma^2\left\Vert z_i^{(t)}\right\Vert^2,\quad\left\Vert\mathbb{E}_t[p_{J,i}^{(t)}]-\nabla_{12}g_i(x_i^{(t)},y_i^{(t)})z_i^{(t)}\right\Vert^2\leq\iota^2,
    \end{align}
\end{subequations} 
where $\iota^2=\dfrac{1}{3}L^2_{\nabla^2g}\delta_t^2\left\Vert z_i^{(t)}\right\Vert^4$.
\end{proposition}
\noindent Proposition \ref{error of finite} demonstrates that the error in the finite-difference Hessian/Jacobian-vector approximation is primarily dominated by sample noise when \(\delta_t\) is sufficiently small.  

We then present the unified convergence results for both {D-SOBA-SO} and {D-SOBA-FO}, stated in the following theorem:
\begin{theorem1}
    \label{thm:convergence of DeMA-SOBA}
    Suppose Assumptions~\ref{assumption:smooth}, \ref{assumption: data heterogeneity}, \ref{assumption:unbiased}, and \ref{assumption: gossip communication} hold. If \( \alpha = \Theta(\sqrt{N/T}) \) and \( \tau = \Theta(\kappa^{-4}) \), then there exist constants \( c_1, c_2, c_3 > 0 \) such that, for all \( 0 \leq t < T \), setting \( \alpha_t \equiv \alpha \), \( \beta_t \equiv c_1 \alpha_t \), \( \gamma_t \equiv c_2 \alpha_t \), and \( \theta_t \equiv c_3 \alpha_t \), the iterates \( \{\bar{x}^{(t)}\}_{t\ge0} \) generated by Algorithm~\ref{alg:DeMA-SOBA} (including both D-SOBA-SO and D-SOBA-FO) satisfy ({see proof in Appendix \ref{proof of convergence of D-SOBA}}):

\begin{align}
\label{eq:convergence_MADeSOBA}
    &\ \dfrac{1}{T}\sum_{t=0}^{T-1}\EE\left[\left\|\nabla\Phi(\bar{x}^{(t)})\right\|^2\right] \nonumber \\
    \lesssim&\ \underbrace{\dfrac{\kappa^5}{\sqrt{NT}}}_{\mathrm{A. rate}} + \quad \underbrace{\dfrac{\rho^{\frac{2}{3}}\kappa^{6}}{(1-\rho)^{\frac{1}{3}}T^{\frac{2}{3}}}+{\dfrac{\rho^{\frac{2}{3}}(b_1^{\frac{2}{3}}\kappa^{\frac{16}{3}}+b_2^{\frac{2}{3}}\kappa^{6})}{(1-\rho)^{\frac{2}{3}}T^{\frac{2}{3}}}}+\dfrac{\rho\kappa^{6}}{(1-\rho)T}+{\dfrac{\kappa^4}{(1-\rho)NT}}+\dfrac{\kappa^{13}}{T}+\dfrac{\kappa^7}{NT}}_{\mathrm{extra\  overhead}},
\end{align}
where $\kappa:=\max\{L_f,L_{\nabla f},L_{\nabla g},L_{\nabla^2 g}\}/\mu_g$ denotes the condition number.
\end{theorem1}

\vspace{2mm}
\noindent 
\textbf{D-SOBA-SO and D-SOBA-FO achieve the same convergence rate.} According to Theorem \ref{thm:convergence of DeMA-SOBA}, {D-SOBA-SO} and {D-SOBA-FO} achieve the same convergence rate. Moreover, Proposition \ref{error of finite} shows that the error introduced by the finite difference approximation is negligible in the overall error, implying that first-order hyper-gradient evaluation does not substantially impact convergence. To the best of our knowledge, this is the first result demonstrating that fully first-order decentralized bilevel algorithms can achieve a convergence rate of the same order as their second-order counterparts.

\vspace{2mm}
\noindent 
{\textbf{Two-Stage Convergence of D-SOBA. }The convergence rate established in Theorem \ref{thm:convergence of DeMA-SOBA} reveals a two-stage behavior for both D-SOBA-SO and D-SOBA-FO, which can be summarized as follows:
\begin{itemize}[leftmargin=*]
\item \textbf{Stage 1: } When $T$ is small, the $1/\sqrt{NT}$ term does not dominate the convergence rate; instead, higher-order terms prevail, leading to slower convergence compared to centralized algorithms. This stage corresponds to the transient iteration period of the decentralized optimization process.
\item \textbf{Stage 2: } When $T$ is sufficiently large, the $1/\sqrt{NT}$ term becomes dominant. The decentralized algorithm achieves linear speedup and convergence performance comparable to that of centralized approaches.
\end{itemize}}

\vspace{2mm}
\noindent 
\textbf{Asymptotic linear speedup.} An algorithm is said to achieve {\em linear speedup} if the term \( 1/\sqrt{NT} \) dominates the convergence rate \citep{koloskova2020unified,lian2017can} as \( T \) becomes sufficiently large. In this regime, the number of iterations required to reach an \( \varepsilon \)-stationary solution is \( 1/(N\varepsilon^2) \), decreasing linearly with the number of computing nodes. As shown in \eqref{eq:convergence_MADeSOBA}, \ours eventually attains linear speedup, whereas the algorithms in \citep{chen2022decentralized,chen2023decentralized} achieve only a significantly slower asymptotic rate of \( 1/\sqrt{T} \). Furthermore, the algorithm that rely solely on first-order gradients proposed in \citep{wang2024fully} achieve a complexity of \( 1/(N\epsilon^7) \), which is worse than that of the proposed D-SOBA-FO.

\vspace{2mm}
\noindent 
\textbf{Transient iteration complexity.} While D-SOBA attains the asymptotic rate of \( 1/(NT) \) as \( T \to \infty \), it undergoes additional transient iterations, quantified by the extra overhead terms in \eqref{eq:convergence_MADeSOBA}, before reaching this rate. Transient iteration complexity \citep{pu2021sharp} refers to the number of iterations an algorithm has to experience before reaching its asymptotic linear-speedup stage, that is, iterations $1,\cdots, T$ where $T$ is relatively small so that non-$NT$ terms still dominate the rate. Here, the smoothness  coefficients as well as the conditional number $\kappa$ are viewed as constants in the transient analysis since we mainly consider the influence of $N,\rho,b_1,b_2$ in transient complexity analysis. Transient iteration complexity measures the non-asymptotic stage in decentralized stochastic algorithms, see Fig.~\ref{fig: tran_iter}. With the explicit asymptotic rate and extra overhead characterized in \eqref{eq:convergence_MADeSOBA}, we establish the transient iteration complexity of D-SOBA algorithms.
\begin{corollary}[\sc transient iteration complexity]
    \label{thm:transient time of DeMA-SOBA}
    Under the same assumptions as in Theorem \ref{thm:convergence of DeMA-SOBA}, the transient iteration complexity of \ours is  {\it (Proof in Appendix \ref{proof of transient time of D-SOBA})}
    \begin{align} \label{eq:transient-expression}
    \mathcal{O}\left(\max\left\{ \frac{N^3}{(1-\rho)^2}, \frac{N^3b^2}{(1-\rho)^4}\right\}\right),
    \end{align}
    {where $b:=\max\{b_1,b_2\}$.}
\end{corollary}

\noindent Many existing works, such as \citep{chen2022decentralized,chen2023decentralized,lu2022decentralized,wang2024fully}, fail to establish the transient iteration complexity as their analysis ignores non-dominant convergence terms. Transient iteration complexity is crucial for differentiating between various decentralized bilevel algorithms. While \citep{yang2022decentralized} and \citep{gao2023convergence} achieve nearly the same asymptotic complexity as D-SOBA (differing by a factor of \(\log(1/\epsilon)\)), their transient complexity is worse than that of D-SOBA, especially when the data heterogeneity \(b^2\) is small and the network topology is sparse (i.e., \(1 - \rho \to 0\)). This addresses \textbf{Open Question Q1} listed in Section~\ref{sec:intro-limitation}.

\vspace{2mm}
\noindent \textbf{Joint influence of graph and heterogeneity}. To our knowledge, Corollary~\ref{thm:transient time of DeMA-SOBA} presents the first result that 
quantifies how network topology and data heterogeneity
jointly affect the non-asymptotic convergence in decentralized SBO. First, expression \eqref{eq:transient-expression} reveals that a sparse topology with $1-\rho \to 0$ can significantly amplify the influence of data heterogeneity $b^2$. Second, a large data heterogeneity $b^2$ also exacerbates the adverse impact of sparse topologies from $\cO((1-\rho)^{-2})$ to $\cO((1-\rho)^{-4})$. Prior to our work, no results have clarified the interplay between network topology and data heterogeneity. Furthermore, expression \eqref{eq:transient-expression} quantifies the benefits that can be achieved by enhancing network connectivity and reducing data heterogeneity. For instance, if the influence of data heterogeneity \(b^2\) can be eliminated, the transient complexity can be improved from $\cO((1-\rho)^{-4})$ to $\cO((1-\rho)^{-2})$. This addresses \textbf{Open Question Q2} and motivates recent efforts \citep{anonymous2024sparkle} to develop more effective decentralized bilevel algorithms that mitigate the influence of data heterogeneity by leveraging advanced methods such as EXTRA \citep{shi2015extra}, Exact-Diffusion \citep{yuan2017exact1}, and Gradient-Tracking \citep{xu2015augmented, nedic2017achieving}.

\vspace{2mm}
\noindent {\textbf{Influence of the data heterogeneity across different levels}. Theorem \ref{thm:convergence of DeMA-SOBA} demonstrates that the heterogeneity at the lower level exerts a more pronounced influence on the convergence rate, particularly when the lower-level problem possesses a large condition number. This effect arises because the corresponding $\kappa$ parameter associated with $b_2$ exhibits a higher order of magnitude compared to that of $b_1$, as derived from Equation \eqref{eq:convergence_MADeSOBA}. Specifically, the increased sensitivity to lower-level heterogeneity can be attributed to the amplification of errors in optimization processes under ill-conditioned scenarios.
However, when $\kappa$ is treated as a constant and the analysis focuses on transient iterations, both $b_1$ and $b_2$ contribute equally to the transient complexity. This uniformity occurs because, under such assumptions, the dominant terms in the complexity bound become symmetric with respect to the heterogeneity parameters, leading to comparable impacts during the transient regime.}

\vspace{2mm}
\noindent 
\textbf{Comparison with single-level DSGD}. Our comprehensive convergence analysis enables the first detailed comparison between bilevel and single-level optimization. We find that both the asymptotic rate and transient iteration complexity of \ours are identical to those of single-level decentralized SGD \citep{chen2021accelerating} with respect to network size \(N\), data heterogeneity \(b^2\), and spectral gap \(1 - \rho\). This suggests that the nested lower- and upper-level structure does not introduce substantial challenges to decentralized stochastic optimization in terms of both asymptotic rate and transient iteration complexity, as illustrated in the comparison in Table~\ref{table:comparison}. However, upon examining the convergence rate in \eqref{eq:convergence_MADeSOBA}, we observe that the presence of the lower-level optimization exacerbates the influence of the condition number. While single-level DSGD achieves an asymptotic rate of \(1/\sqrt{NT}\), D-SOBA results in \(\kappa^5/\sqrt{NT}\). Prior to our work, all previous works listed in Table~\ref{table:comparison} fail to explicitly capture the influence of \(\kappa\) on both asymptotic and non-asymptotic terms. With all these results, \textbf{Open Question Q3} has been addressed.

\vspace{2mm}
\noindent 
{
\textbf{Non-asymptotic concensus error. }With the theoretical analysis of the convergence rate, we can also present an upper-bound for the concensus error of \ours as follows ({see proof in Appendix \ref{Proof of the non-asymptotic concensus error}}):
\begin{corollary}[\sc concensus error]
    \label{thm:concensus error}
    Suppose Assumptions~\ref{assumption:smooth}, \ref{assumption: data heterogeneity}, \ref{assumption:unbiased}, and \ref{assumption: gossip communication} hold. If we take the hyperparameters as Theorem \ref{thm:concensus error}, then the concensus error of \ours satisfies that:
\begin{align}
\label{eq:concensus_MADeSOBA}
    &\dfrac{1}{T}\sum_{t=0}^{T-1}\EE\left[\dfrac{\left\Vert \mathbf{x}^{(t)}-\bar{\mathbf{x}}^{(t)} \right\Vert_F^2}{n}+\dfrac{\left\Vert \mathbf{y}^{(t)}-\bar{\mathbf{y}}^{(t)} \right\Vert_F^2}{n}+\dfrac{\left\Vert \mathbf{z}^{(t)}-\bar{\mathbf{z}}^{(t)} \right\Vert_F^2}{n}\right] \nonumber \\
    \lesssim&\dfrac{N}{T}\left(\dfrac{b^2}{(1-\rho)^2}+\dfrac{1}{1-\rho}\right)+\dfrac{N^{\frac{1}{2}}}{T^{\frac{3}{2}}(1-\rho)^2}+\dfrac{1}{(1-\rho)NT}.
\end{align}
\end{corollary}
}

\vspace{2mm}
\noindent 
\textbf{Convergence in deterministic scenario.} 
When gradient or Hessian information can be accessed noiselessly, \ie, $\sigma_{f1}$, $\sigma_{g1}$, and $\sigma_{g2}$ are all zero, \ours has enhanced convergence. Following the argument in Theorem \ref{thm:convergence of DeMA-SOBA}, we can readily derive the following result:
\begin{corollary}[\sc deterministic convergence]
    \label{thm:convergence of deterministic}
    Suppose Assumptions ~\ref{assumption:smooth}, \ref{assumption: data heterogeneity}, and \ref{assumption: gossip communication} hold, and variances $\sigma_{f1}=\sigma_{g1}=\sigma_{g2}=0$. \oursSO and \oursFO achieve the following conergence rate ({see proof in Appendix \ref{proof of deterministic}}):
    \begin{align}
        \dfrac{1}{T}\sum_{t=0}^{T-1} \left\Vert\nabla\Phi(\bar{x}^{(t)})\right\Vert^2  
        \lesssim&\ {\dfrac{\rho^{\frac{2}{3}}(b_1^{\frac{2}{3}}\kappa^{\frac{16}{3}}+b_2^{\frac{2}{3}}\kappa^{6})}{T^{\frac{2}{3}}(1-\rho)^{\frac{2}{3}}}}+\dfrac{\rho\kappa^6}{(1-\rho)T}+{\dfrac{\kappa^4}{(1-\rho)NT}}+\dfrac{\kappa^{13}}{T}. 
    \end{align}
\end{corollary}
\noindent Corollary \ref{thm:convergence of deterministic} implies that \ours achieves an asymptotic rate of \(\mathcal{O}(T^{-2/3}(1-\rho)^{-2/3})\). In the same deterministic setting, the concurrent single-loop decentralized bilevel algorithm \citep{dong2023single} achieves a rate of at best \(\mathcal{O}(T^{-1}(1-\rho)^{-3/2})\), which is inferior to ours when the communication topology is sufficiently sparse (i.e., as \((1-\rho) \to 0\)). Additionally, the single-loop algorithm employs a gradient tracking technique, which introduces extra communication overhead and does not address performance in stochastic settings.

\section{Experiments}
\label{section: experiments}
In this section, we present experiments to validate the influence of convergence caused by communication topologies, data heterogeneity, moving average parameter $\theta$, and finite-difference parameter $\delta$. We also compare \ours to other decentralized SBO algorithms and show the benefit of finite-difference Hessian/Jacobian-vector approximation of D-SOBA-FO in high-dimensional settings.

\subsection{Linear regression with a synthetic dataset.}
\label{section: experiments_linear}

Here, we consider problem \eqref{prob:general} with the upper- and lower-level loss functions defined as 
\begin{subequations}
\label{exp:toymodel_detail}
\begin{align*}
f_i(x,y)&=\mathbb{E}_{\xi_i,\zeta_i}\left[\left(\xi_i^\top y-\zeta_i\right)^2\right], \\
g_i(x,y)&=\mathbb{E}_{\xi_i,\zeta_i}\left[\left(\xi_i^\top y-\zeta_i\right)^2+|x| \Vert y \Vert^2\right] ,
\end{align*}
\end{subequations}
where $y\in\mathbb{R}^p$ denotes the regression parameter, $|x|\geq0$ is a ridge regularization parameter. This setting exemplifies tuning the regularization parameter with the upper-/lower-level objectives associated with validation and training datasets, respectively. Each node $i$ observes data samples $(\xi_i,\zeta_i)$ in a streaming manner in which $\xi_i$ is a $p$-dimensional vector with all elements drawn independently and randomly form the uniform distribution $U\left(-2\times 1.5^{{1}/{3}},2\times 1.5^{{1}/{3}}\right)$. Then $\zeta_i$ is generated by $\zeta_i=\xi_i^\top\omega_i+\delta_i$, where $\delta_i{\sim}\mathcal{N}({0},1)$.
Moreover, $\omega_i$ is generated by $\omega_i=\omega+\varepsilon_i$, where $\omega$ is a given vector whose elements are generated independently and randomly form the uniform distribution $U\left(0,10\right)$. Here we adjust $\varepsilon\overset{iid}{\sim}\mathcal{N}({0},2I_p)$ to represent severe heterogeneity across nodes while $\varepsilon_i\overset{iid}{\sim}\mathcal{N}({0},0.5I_p)$ yields mild heterogeneity.

We set $p=10$ and the number of nodes was set to $9$ and $20$. The communication topologies are set to Ring, 2D-Torus~\citep{ying2021exponential} topologies ($3\times3$ when $N=9$ and $4\times5$ when $N=20$) as well as a fully connected network (i.e. centralized SOBA). The step-sizes $\alpha_t,\beta_t,\gamma_t$ are initialized as $0.1$ and multiplied by $0.8$ per $1,000$ iterations while $\theta_t$ is set to $0.2$. We repeat all the cases $50$ times and plot the average of all trials. 

\begin{figure}[t!]
\centering
\subfigure{
		\includegraphics[width=0.45\textwidth]{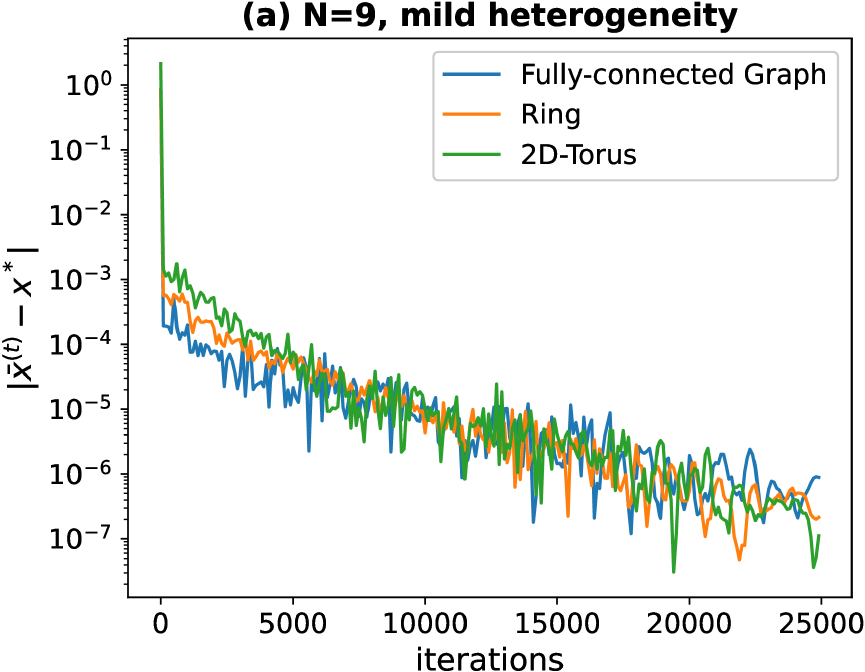}}
\hspace{7pt}
\subfigure{
		\includegraphics[width=0.45\textwidth]{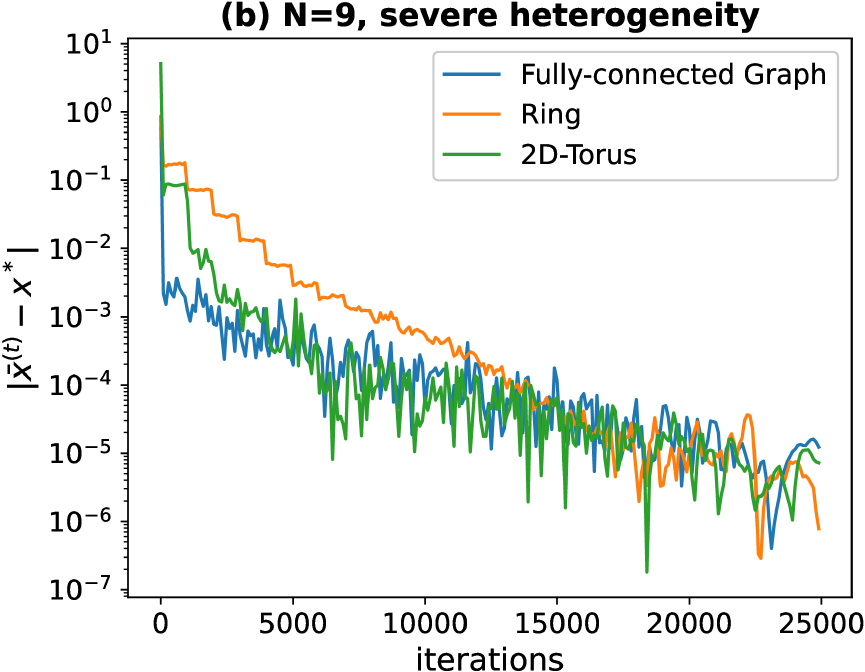}}\\
\subfigure{
		\includegraphics[width=0.45\textwidth]{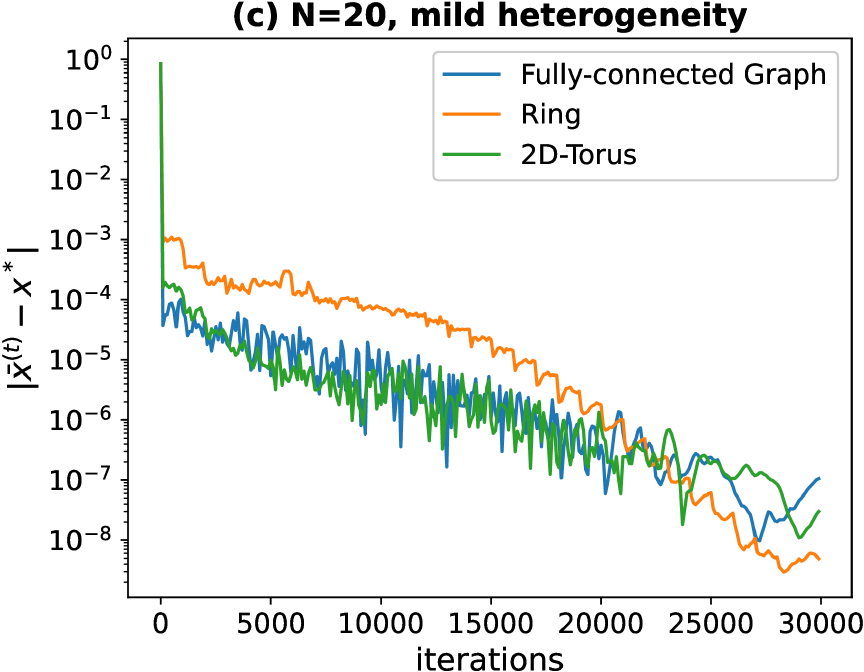}}
\hspace{7pt}
\subfigure{
		\includegraphics[width=0.45\textwidth]{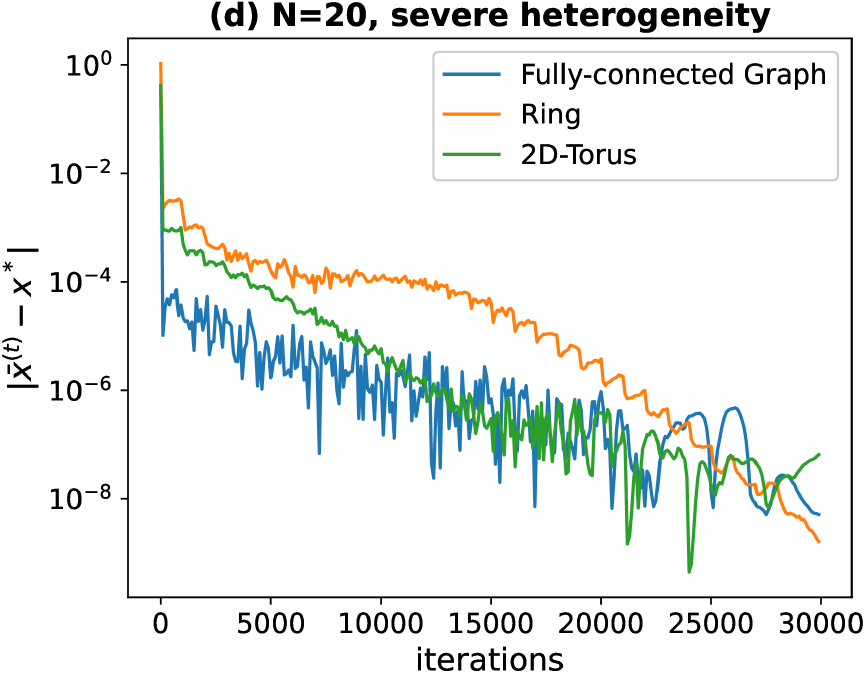}}
\caption{\small Convergence performance of D-SOBA-SO over various networks under different data heterogeneity.}
\label{fig: toy}
\end{figure}

Fig.\ref{fig: toy} illustrates the performance of D-SOBA-SO. {For Fig. \ref{fig: toy}, we can observe that D-SOBA with a decentralized communication topology achieves linear speedup if the corresponding convergence curve matches with the convergence curve with a fully-connected graph. Thus, Fig.\ref{fig: toy} illustrates the two-stage convergence of the D-SOBA algorithm.}

It can be also observed that for both $N=9$ and $N=20$, the transient stage of D-SOBA-SO over decentralized networks becomes longer as data heterogeneity increases, highlighting the crucial impact of data heterogeneity. Additionally, D-SOBA-SO over the 2D-torus graph exhibits a shorter transient stage than that with the ring graph, indicating that networks with a smaller spectral gap (i.e., worse connectivity) can significantly slow down convergence. Moreover, the transient iterations of both the 2D-torus and ring networks increase as the network size $N$ grows.  All these results are consistent with our transient iteration complexity derived in Corollary \ref{thm:transient time of DeMA-SOBA}.

\subsection{Data hyper-cleaning on FashionMNIST dataset.}
\label{section: experiments_minst}
Here we consider a data hyper-cleaning problem \citep{shaban2019truncated} for a Fashion MNIST dataset \citep{xiao2017fashion}, which consists of 60000 images for training and 10000 images for testing. Then we split the 60000 training images into the training set and validation set, which contains 50000 images and 10000 images separately.  

In the data hyper-cleaning problem, we aim to train a classifier in a corrupted setting in which the label of each training data is replaced by a random class number with a probability $p$ (i.e. the corruption rate) on a decentralized network with $N=10$ clients. It can be viewed as a bilevel optimization problem \eqref{prob:general} whose loss functions on the upper and lower level of the $i$-th node denote:
\begin{subequations}
\label{exp:hyperclean_detail_app}
\begin{align}
f_i(x,y)&=\dfrac{1}{\left|\mathcal{D}_{val}^{(i)}\right|}\sum_{(\xi_e,\zeta_e)\in D_{val}^{(i)}}L(\phi(\xi_e;y),\zeta_e),\\
g_i(x,y)&=\dfrac{1}{\left|\mathcal{D}_{tr}^{(i)}\right|}\sum_{(\xi_e,\zeta_e)\in \mathcal{D}_{tr}^{(i)}}\sigma(x_e)L(\phi(\xi_e;y),\zeta_e)+C\left\Vert y\right\Vert^2,
\end{align}
\end{subequations}
where $\phi$ denotes the parameters of a two-layer MLP network with a 300-dim hidden layer and ReLU activation while $y$ denotes its parameters. $L$ denotes cross-entropy loss, $\sigma(u)=(1+e^{-u})^{-1}$ denotes the sigmoid function. $\mathcal{D}_{val}^{(i)}$ and $\mathcal{D}_{tr}^{(i)}$ denote the validation and training sets of the $i$-th client, which are sampled randomly by Dirichlet distribution with parameters $\alpha=0.1$ in non-i.i.d. cases \citep{lin2021quasi}. Following with the setting in \citep{shaban2019truncated}, we set $C=0.001$. The step-sizes in all the gradient updates are all set to 0.1. And all the experiments are repeated 10 times.

\textbf{Comparison between D-SOBA and other existing algorithms. }We firstly compare D-SOBA-SO and D-SOBA-FO with MA-DSBO \citep{chen2023decentralized} and Gossip DSBO~\citep{yang2022decentralized} by solving \eqref{exp:hyperclean_detail_app} over an exponential graph \citep{ying2021exponential} with $p=0.1,0.4$, respectively. For all algorithms, the step-sizes are set to 0.1 and the batch size are set to 200. Meanwhile, the moving average parameters $\theta_t$ of D-SOBA-SO, D-SOBA-FO and MA-DSBO are set to 0.8. For MA-DSBO, we set the number of inner-loop iterations $T=5$ and the number of outer-loop iterations $K=5$. For Gossip DSBO, we set the number of Hessian-inverse estimation iterations $T=5$. At the end of the update of outer parameters, we use the average of $y$ among all clients to do the classification of the test set. 

Fig.~\ref{fig: FashionMINST_algs11} presents the upper-level loss and test accuracy of D-SOBA-SO, D-SOBA-FO, MA-DSBO, and Gossip DSBO when the corruption rate $p$ equals 0.1 and 0.4 over an exponential graph \citep{ying2021exponential} with $N=10$ nodes, respectively. We observe that both D-SOBA-SO and D-SOBA-FO achieves higher test accuracy compared to the other two algorithms. 

{Moreover, although the first-order variant in D-SOBA-FO reduces computational overhead, it may still lead to suboptimal performance than D-SOBA-SO in Figure 3 as the finite difference of first-order gradients in D-SOBA-FO exactly approximates a Hessian/Jacobian–vector product. Then D-SOBA-FO may introduce additional error into the hyper-gradient evaluation. In contrast, D-SOBA-SO computes exact Hessian-/Jacobian–vector products directly from stochastic samples, resulting in fewer iterations to reach a given accuracy compared to first-order methods. On the other hand, the Hessian/Jacobian–vector product is computed via PyTorch's forward-mode automatic differentiation in D-SOBA-SO, which is significantly faster than first computing the full Hessian/Jacobian matrix and then multiplying it by the vector. This computational advantage prevents its per-iteration time from becoming prohibitively slow compared to D-SOBA-FO. }

\begin{figure}[t!]
\centering
	\subfigure{
		\includegraphics[width=0.48\textwidth]{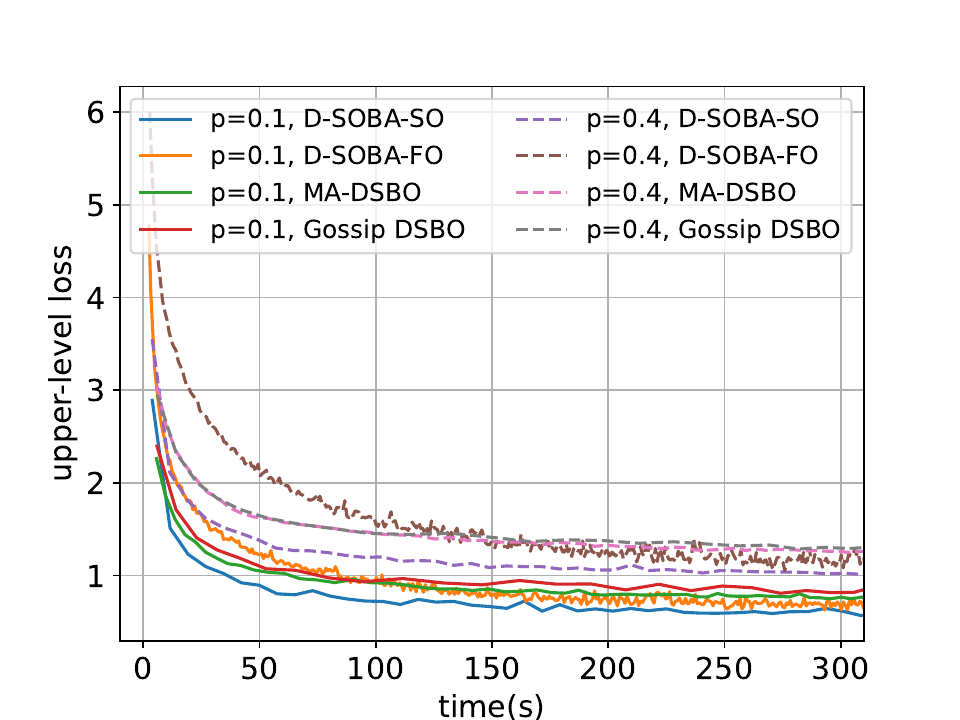}}
	\subfigure{
		\includegraphics[width=0.48\textwidth]{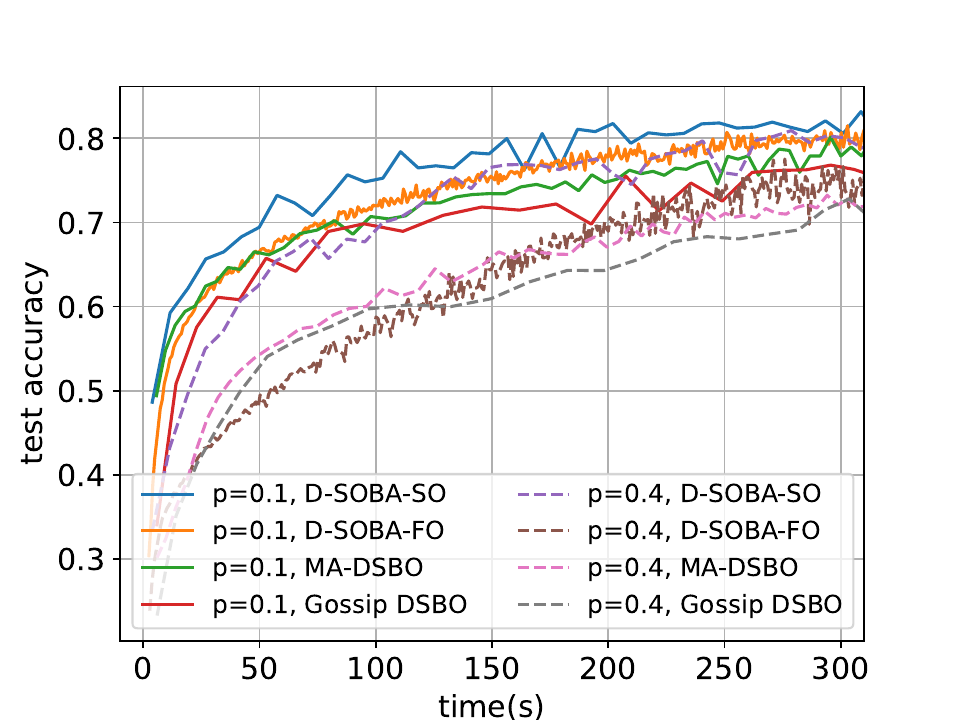}
        \label{fig: FashionMINST_topo_testacc_app}}
\caption{The upper-level loss (left) and test accuracy (right) of different decentralized stochastic bilevel optimization algorithms.}
\label{fig: FashionMINST_algs11}
\end{figure}

\begin{figure}
\centering
	\subfigure{
		\includegraphics[width=0.48\textwidth]{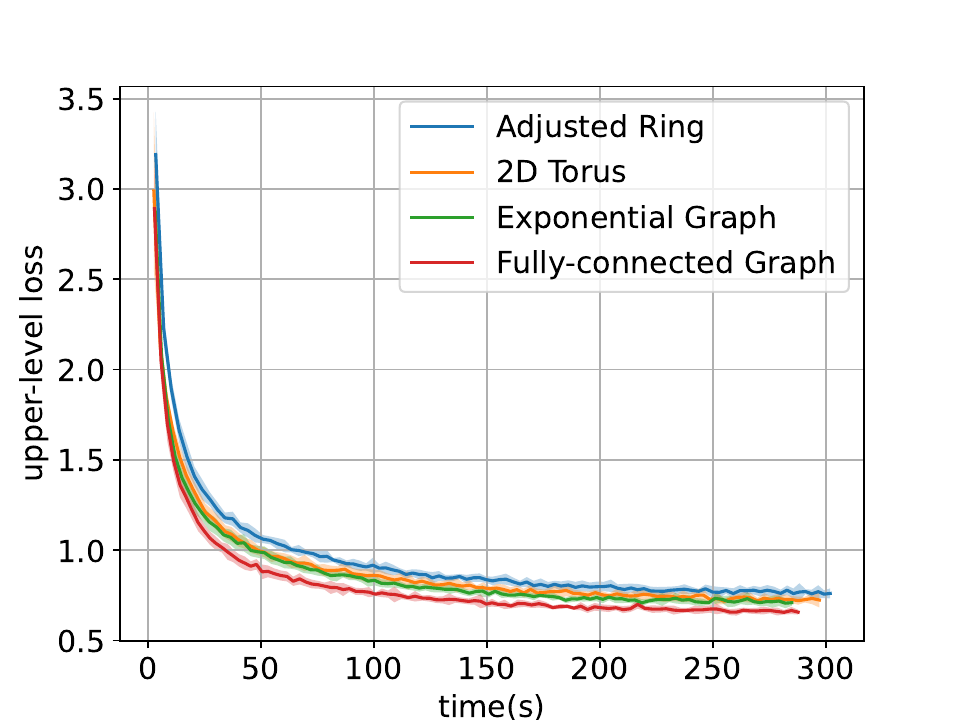}}
	\subfigure{
		\includegraphics[width=0.48\textwidth]{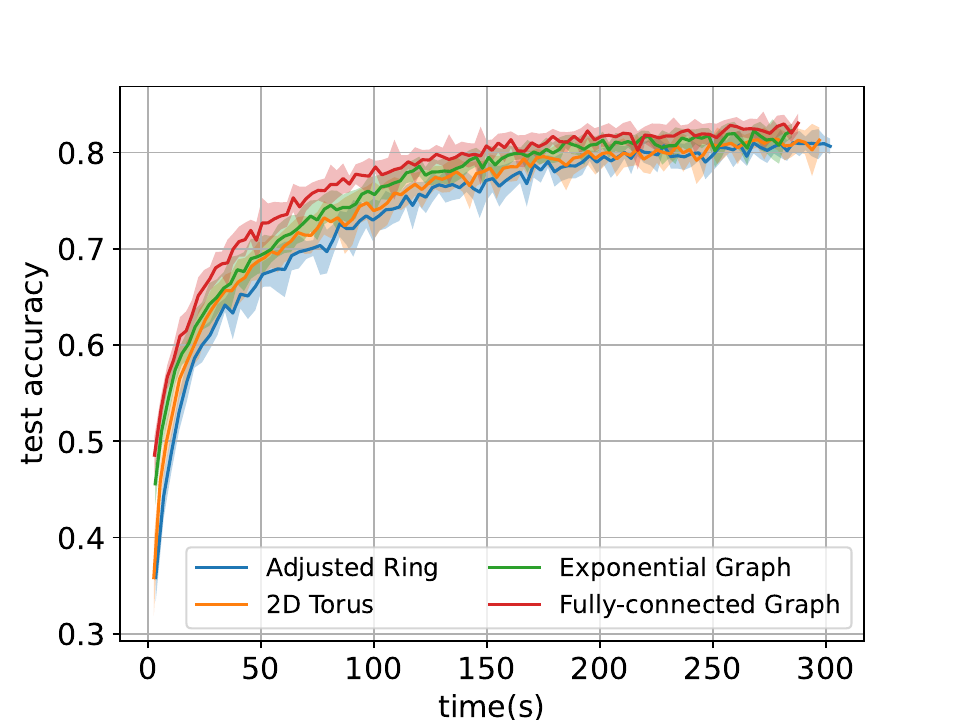}
        \label{fig: FashionMINST_topo_testacc_app}}
\caption{The upper-level loss (left) and test accuracy (right) of D-SOBA with different communication topologies.}
\label{fig: FashionMINST_topo_app}
\end{figure}

\begin{figure}[t!]
\hspace{-15pt}
	\centering
	\begin{minipage}[c]{0.48\textwidth}
    \hspace{-17pt}
    \centering
    	\subfigure{
    		\includegraphics[width=1\textwidth]{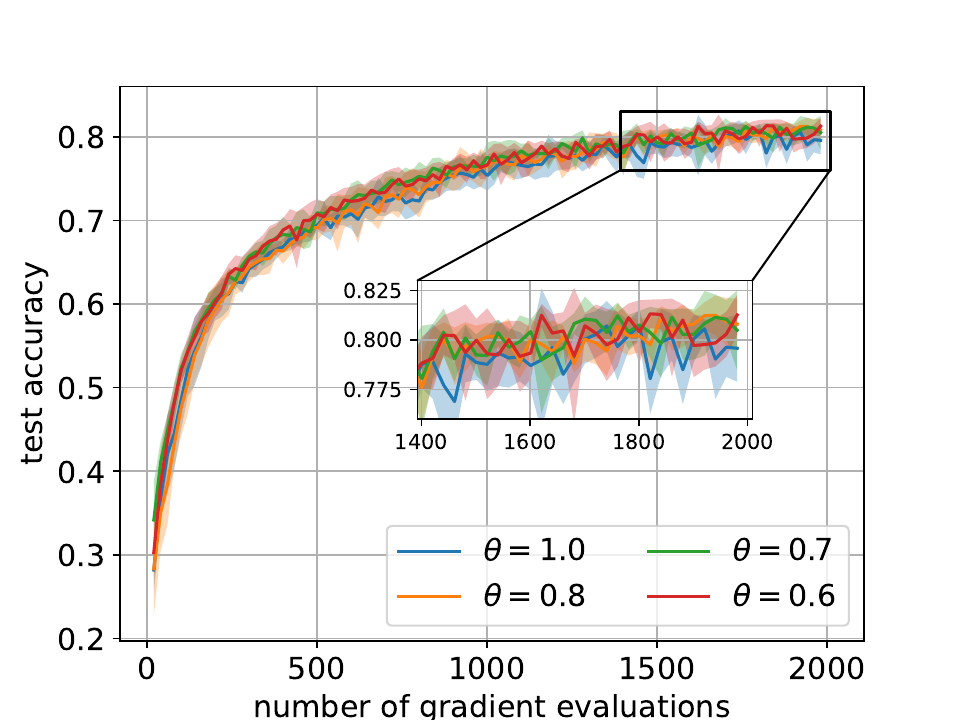}}
      \vspace{-10pt}
    \hspace{-16pt}
    \caption{\small The test accuracy of D-SOBA-SO with different moving-average parameter $\theta_t$.}
    \label{fig: FashionMINST_theta}
	\end{minipage} 
    \hspace{4pt}
	\begin{minipage}[c]{0.48\textwidth}
    \centering
      \hspace{-20pt}
    	\subfigure{
    		\includegraphics[width=\textwidth]{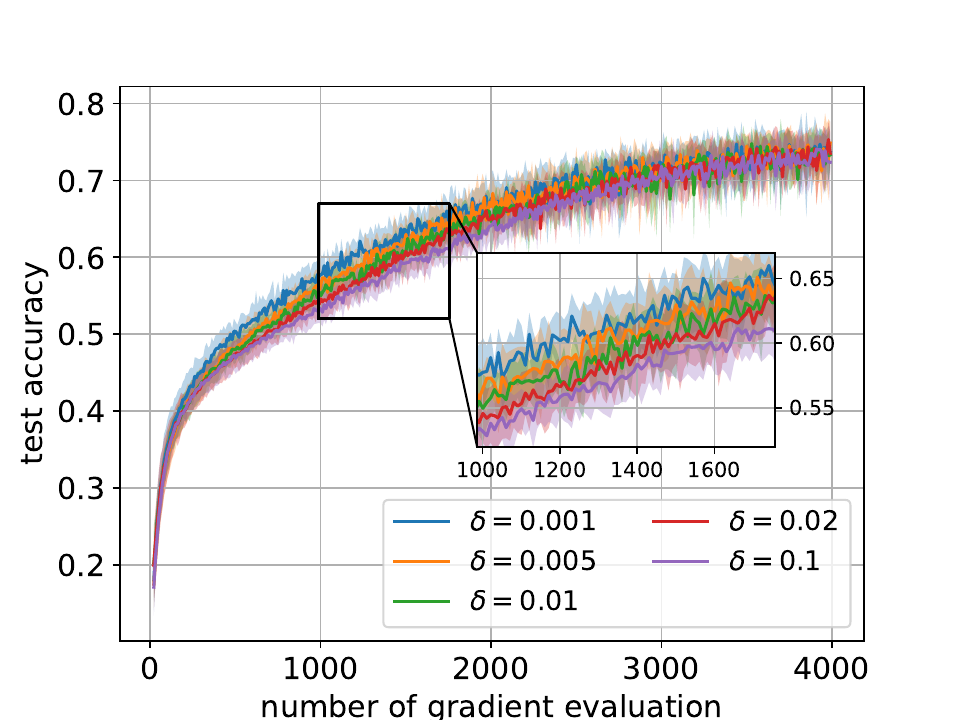}}
      \vspace{-10pt}
    \hspace{-20pt}
    \caption{\small The test accuracy of D-SOBA-FO with different moving-average parameter $\delta_t$.}
    \label{fig: FashionMINST_delta}
	\end{minipage} 
\hspace{-15pt}
\end{figure}

\textbf{Convergence of D-SOBA with different communication topologies. }Then we conduct the D-SOBA-SO with different topologies including Adjust Ring, 2D-Torus, Exponential graph as well as the centralized case in the non-i.i.d setting, while the weight matrix of Adjust Ring $W=[w_{ij}]_{N\times N}$ satisfies:
\begin{align*}
w_{ij}=\begin{cases} 0.4,\quad\text{if }(j-i)\%{N}\in\{\pm1\} ,
\\ 0.2,\quad\text{if }j=i,
\\ 0,\hspace{1.5em}\text{  else}.
    \end{cases}
\end{align*}
We set $p=0.2$, batch size equal to 200, and repeat all the cases 10 times and illustrate the mean of all trials. Figure \ref{fig: FashionMINST_topo_app} illustrates the upper-level training loss and test accuracy of different cases, {from which we can observe how the two-stage convergence of decentralized communication affect the practical optimization tasks. Specifically, D-SOBA with sparser topologies exhibits slower convergence in the upper-level loss and then comes with suboptimal test accuracy.}

\textbf{Influence of moving-average iterations on convergence. }Here we aim to find the influence of moving-average parameter $\theta$ on the convergence. Let the batch size equal 100 and $p=0.2$, we run D-SOBA-SO on an Adjusted Ring with $\theta\in\{1.0,0.8,0.7,0.6\}$. We repeat each case 10 times and obtain the test accuracy. As the Figure \ref{fig: FashionMINST_theta} demonstrates, the cases when $\theta=0.7,0.8$ achieve a higher average test accuracy than when $\theta=0.6,1.0$, which implies the trade-off of the selection $\theta$ in the convergence.

\textbf{Influence of finite-difference parameter $\delta$ in D-SOBA-FO. } To find how the hyperparameter $\delta$ in D-SOBA-FO affect the convergence, we run D-SOBA-FO on an Adjusted Ring with the batch size equal 100 and $p=0.3$. And the term $\theta=0.8$. Moreover, the term $\delta=0.001,0.005,0.01,0.02,0.1$, respectively. As Fig. \ref{fig: FashionMINST_delta} shows, a smaller $\delta_t$ leads to a faster convergence at the beginning of the iteration. However, the influence of $\delta$ on the convergence is insignificant as all the cases finally achieve similar test accuracy. This result illustrates how $\delta_t$ affect the convergence, which has been discussed in Section \ref{section:Convergence Analysis}.

{
\begin{figure}[t!]
\centering
	\subfigure{
		\includegraphics[width=0.48\textwidth]{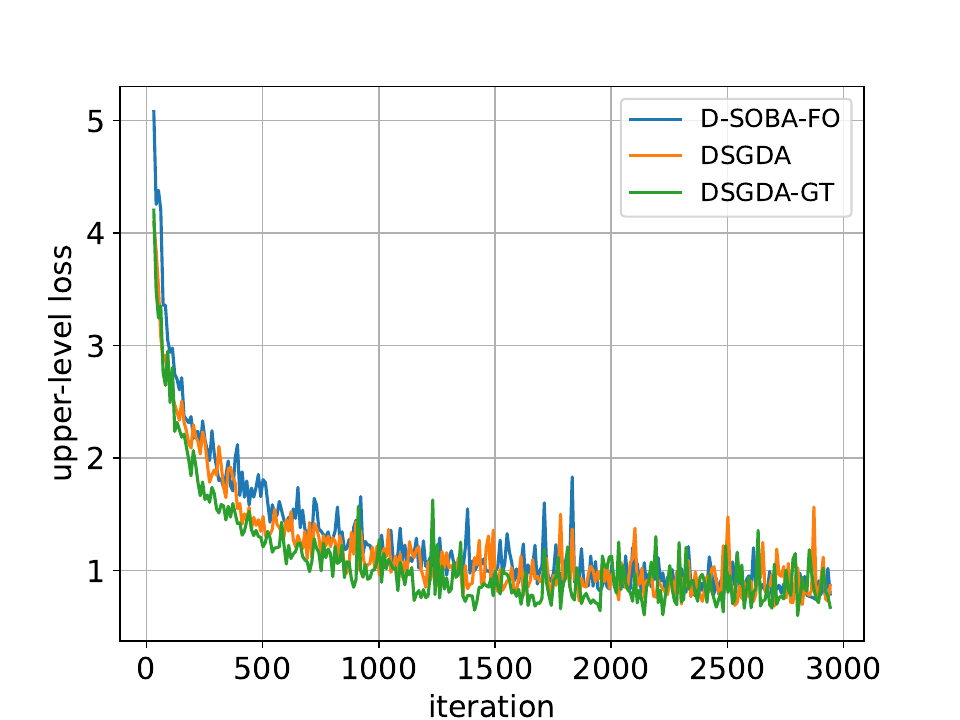}}
	\subfigure{
		\includegraphics[width=0.48\textwidth]{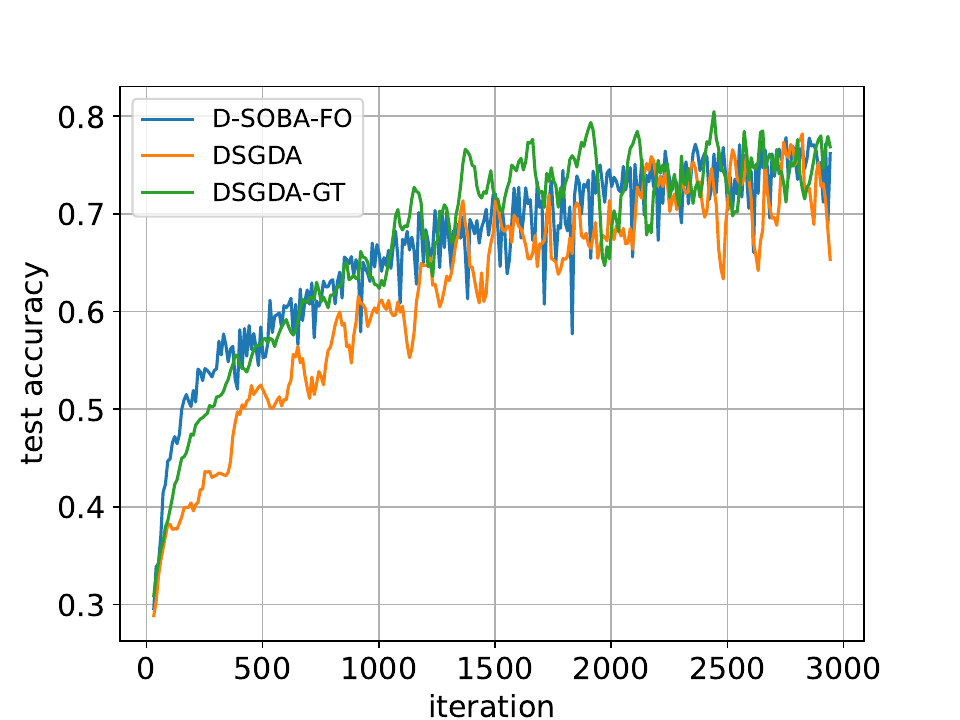}
        \label{fig: FashionMINST_topo_testacc_app}}
\caption{{The upper-level loss (left) and test accuracy (right) of different first-order decentralized stochastic bilevel optimization algorithms.}}
\label{fig: FashionMINST_fo_dsgda}
\end{figure}

\begin{figure}[t!]
\centering
	\subfigure{
		\includegraphics[width=0.48\textwidth]{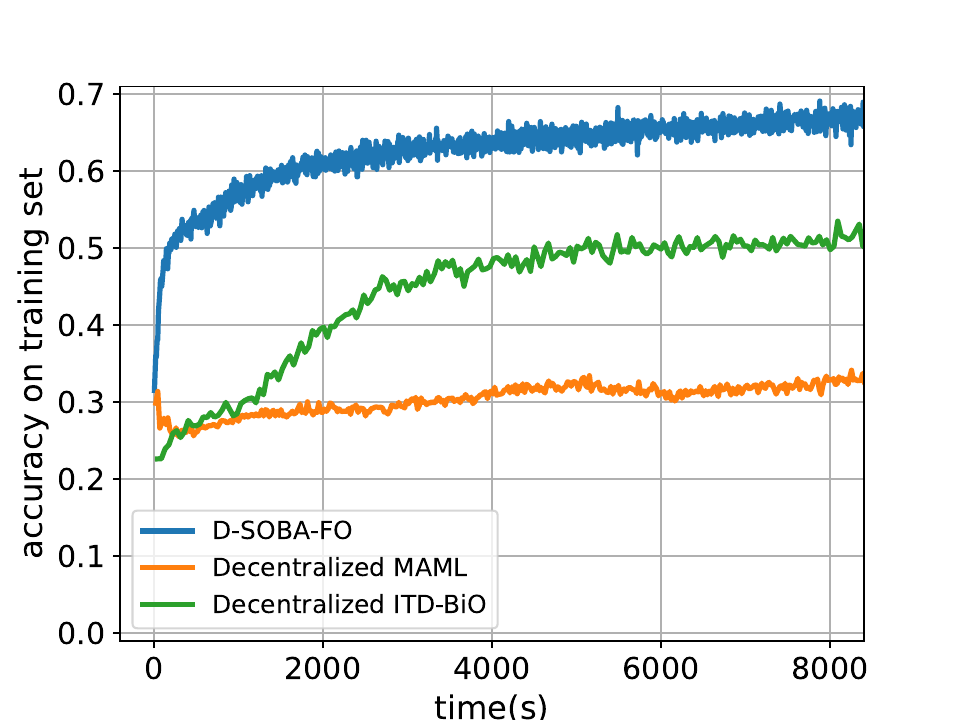}}
	\subfigure{
		\includegraphics[width=0.48\textwidth]{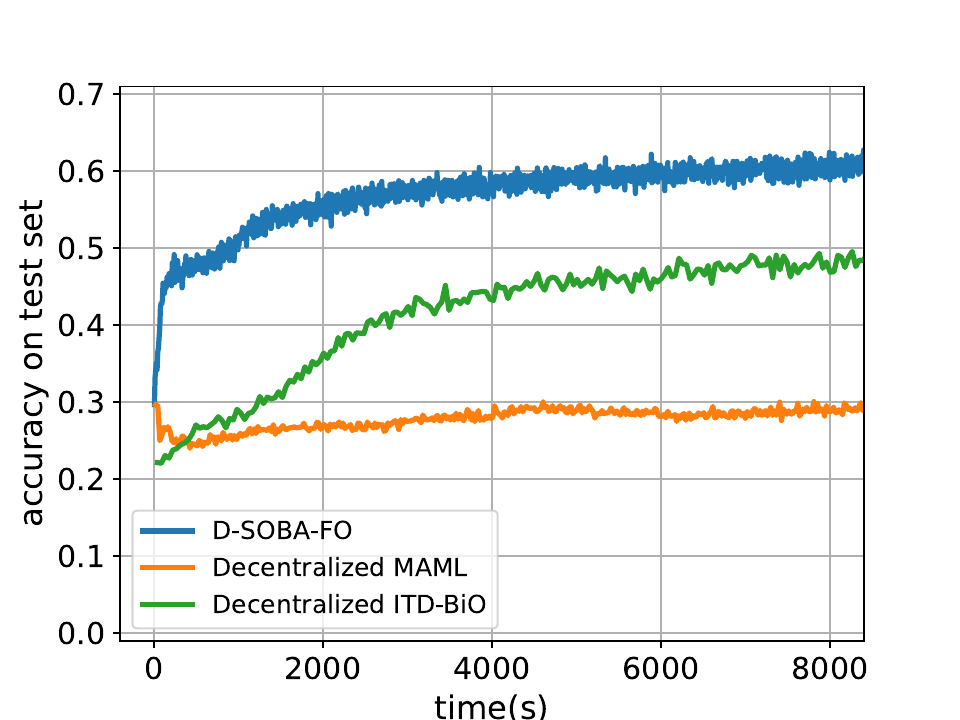}
        \label{fig: Fmeta}}
\caption{The accuracy on training set (left) and testing set (right) of different algorithms for the meta-learning problem.}
\label{fig: FashionMINST_algs1}
\end{figure}

\textbf{Comparison between D-SOBA-FO and first-order baselines.} Additionally, we compare D-SOBA-FO with DSGDA\footnote{{DSGDA is the non-tracking version of DSGDA-GT that proposed in \citep{wang2024fully}.}} as well as DSGDA-GT \citep{wang2024fully} on an Adjusted Ring topology, with batch size 50 and $p=0.2$. Key hyperparameters: $\delta=0.0001$ (D-SOBA-FO), inner iterations $T=1$ per gradient step (DSGDA), $\theta_t=0.8$ (D-SOBA-FO moving average), and $\alpha=1$ (DSGDA). All methods use step size 0.03. Figure \ref{fig: FashionMINST_fo_dsgda} demonstrates comparable performance between D-SOBA-FO and DSGDA in both upper-level loss minimization and test accuracy.
}

\subsection{Decentralized meta-learning}
Consider a meta-learning problem \citep{finn2017model}. We consider $R$ tasks $\{\mathcal{T}_r,r=1,\cdots,R\}$. Each task $\mathcal{T}_r$ has its own loss function $l_r(x,y_r)=\mathbb{E}_{\xi\in\mathcal{D}_r}L(x,y_r,\xi)$, where $y_r$ denotes the task-specific parameters and $x$ denotes global parameters shared by all the tasks. Meta-learning problem aims to find good global parameter $x^{\star}$ to bulid the embedded features. Then each task can adapt its task-specific parameters $y_i$.

Here, we train a meta-learning framework in a decentralized scenario \citep{kayaalp2022dif} over $N$ nodes. For $i=1,2,\cdots,N$, let $\mathcal{D}_{r,i}$ denote the data set received by the $i$-th node of the $r$-th task $\mathcal{T}_r$. And we split $\mathcal{D}_{r,i}$ to $\mathcal{D}^{\text{train}}_{r,i}$ and $\mathcal{D}^{\text{val}}_{r,i}$ for training and validation, respectively. Then we find the optimal parameter by solving the decentralized SBO problem \eqref{prob:general} with the upper- and lower-level loss functions defined as:
\begin{equation}
    \begin{aligned}
        f_i(x,y)&=\dfrac{1}{R}\sum_{r=1}^R\left[\dfrac{1}{|\mathcal{D}^{\text{val}}_{r,i}|}\sum_{\xi\in\mathcal{D}^{\text{val}}_{r,i}}L(x,y_r,\xi)\right], \\
        g_i(x,y)&=\dfrac{1}{R}\sum_{r=1}^R\left[\dfrac{1}{|\mathcal{D}^{\text{train}}_{r,i}|}\sum_{\xi\in\mathcal{D}^{\text{train}}_{r,i}}L(x,y_r,\xi)+\mathcal{R}(y_r)\right],
    \end{aligned}
\end{equation}
where $\mathcal{R}(y_r)$ denotes a strongly-convex regularization term of $y_r$.

In this experiment, we compared D-SOBA-FO with ITD-BiO \citep{ji2021bilevel} and MAML \citep{finn2017model} with decentralized communication over a 5-way 10-shot task over a network of $N=8$ nodes connected by Ring graph. The dataset is miniImageNet \citep{vinyals2016matching}, which is generated from ImageNet \citep{russakovsky2015imagenet} and consists of 100 classes with each class containing 600 images of size $84\times84$. Following the setting of, we split these classes into 64 for training, 16 for validation, and 20 for testing. Each training and validation class splits its data by Dirichlet distribution with parameters $\alpha=0.1$ \citep{lin2021quasi}. We use a four-layer CNN with four convolution blocks in which each block sequentially consists of a $3\times3$ convolution with 32 filters, batch normalizationm, ReLU activation, and $2\times2$ max pooling. $L$ denotes cross-entropy loss and $\mathcal{R}(y)=0.001||y||^2$. The number of tasks $R=1000$. The batch size is set to 32. The parameters of the last linear layer are set to be task-specific and the other parameters are global. We choose the stepsize 0.002 for upper-level iterations and 0.01 for lower-level iterations for all algorithms. The stepsize of auxiliary variables in D-SOBA-FO is 0.01 and the term $\delta_t$ in D-SOBA-FO is 0.001. 

Figure \ref{fig: FashionMINST_algs1} the average accuracy on training set for all nodes and the test accuracy of the three algorithms. We can observe that D-SOBA-FO outperforms than the other algorithms in this large-scaled problems. Thus, this experiment indicates the advantage of D-SOBA-FO in large-scaled scenarios.

\section{Conclusion and Limitations}
\vspace{-5pt}
This paper introduces D-SOBA-SO and D-SOBA-FO, two single-loop decentralized SBO algorithms that achieve state-of-the-art asymptotic convergence rate and gradient/Hessian complexity under more relaxed assumptions compared to existing methods. Especially, D-SOBA-FO evaluates the Hessian/Jacobian-vector only by first-order gradients. We also clarify, for the first time, the joint influence of network topology and data heterogeneity on bilevel algorithms through an analysis of transient iteration complexity. Nevertheless, \ours relies on the assumption of bounded gradient dissimilarity, which might be violated when local distribution $\cD_{f_i}$ or $\cD_{g_i}$ differs drastically across nodes. 

\section*{Acknowledgements}
The work of Kun Yuan is supported by National Key Research and Development Program of China (No. 2024YFA1012902), and National Natural Science Foundation of China (No. 92370121, 12301392, W2441021). Songtao Lu is supported in part by the Chinese University of Hong Kong (CUHK) Direct Grant (Project No. 4055259).

\newpage

\appendix
\section{Proof of the Convergence Rate for Algorithm \ref{alg:DeMA-SOBA}.}
\label{proof of convergence of D-SOBA}
Here we present the proof of Theorem \ref{thm:convergence of DeMA-SOBA} as well as two corollaries including the transient time and the convergence rate in the deterministic case of Algorithm \ref{alg:DeMA-SOBA}, which provides the theoretical analysis on the convergence of Algorithm \ref{alg:DeMA-SOBA}. 
\subsection{Notations}
We first introduce some notations which will be used in the proof as follows. Firstly, for  $0\leq t\leq T$, let $\bar{y}^{(t)}_{\star}=\mathop{\arg\min}_{y}g(\bar{x}^{(t)},y)$ denote the minimizer of \eqref{eq:lower} with respect to  $\bar{x}^{(t)}$. Then, let $\bar{z}^{(t)}_{\star}=(\nabla_{22}^2g(\bar{x}^{(t)},\bar{y}^{(t)}_{\star}))^{-1}\nabla_2f(\bar{x}^{(t)},\bar{y}^{(t)}_{\star})$ denote the minimizer of \eqref{eq:soba} with respect to  $\bar{x}^{(t)}, \bar{y}^{(t)}_{\star}$. Denote $q_i^{(t)} =\nabla_1f_i(x_i^{(t)},y_i^{(t)})-\nabla^2_{12}g_i(x_i^{(t)},y_i^{(t)})z_i^{(t)} $, and $\bar{q}^{(t)}=\frac{1}{N}\sum_{i=1}^nq_i^{(t)}$. 

For iteration $t$, we use the following notations to represent the stochastic gradients in Algorithm \ref{alg:DeMA-SOBA}: 
\begin{subequations}
\label{alg-variables}
\begin{align*}
u_{1,i}^{(t)} & :=\nabla_1F(x_{i}^{(t)},y_{i}^{(t)};\xi^{(t)}_i), \quad\hspace{4pt}
u_{2,i}^{(t)}:=\nabla_2F(x_{i}^{(t)},y_{i}^{(t)};\xi^{(t)}_i),\quad\hspace{4pt}
v_{2,i}^{(t)} := \nabla_{2}G(x_{i}^{(t)},y_{i}^{(t)};\zeta^{(t)}_i).
\end{align*} 
\end{subequations}

\textbf{Remark.} From the definition of $\mathcal{F}^{(t)}$, we know that $h^{(t+1)}_i$ and $x^{(t+1)}_i$ ($1\leq i\leq N$) are both $\mathcal{F}^{(t)}$-measurable. Then from the update of $x^{(t+1)}_i$ in Algorithm \ref{alg:DeMA-SOBA}, we know that $x^{(t+1)}_{i}$ is actually  $\mathcal{F}^{(t)}$-measurable. 

Then we present the matrix form of iterators on different clients as follows:
\begin{align*}
    \mathbf{x}^{(t)}&=\left(x^{(t)}_{1},x^{(t)}_{2} ,\cdots,x^{(t)}_{N}\right)^\top _{N\times{d}},\quad\mathbf{\bar{x}}^{(t)}=\left(\bar{x}^{(t)} ,\bar{x}^{(t)} ,\cdots,\bar{x}^{(t)}\right)^\top _{N\times{d}},\\
    \mathbf{y}^{(t)}&=\left(y^{(t)}_{1},y^{(t)}_{2} ,\cdots,y^{(t)}_{N}\right)^\top _{N\times{p}},\quad\mathbf{\bar{y}}^{(t)}=\left(\bar{y}^{(t)} ,\bar{y}^{(t)} ,\cdots,\bar{y}^{(t)} \right)^\top _{N\times{p}},\\
    \mathbf{z}^{(t)}&=\left(z^{(t)}_{1}, z^{(t)}_{2} ,\cdots,z^{(t)}_{N}\right)^\top _{N\times{p}},\quad\mathbf{\bar{z}}^{(t)}=\left(\bar{z}^{(t)} ,\bar{z}^{(t)},\cdots,\bar{z}^{(t)} \right)^\top _{N\times{p}},\\
    \mathbf{h}^{(t)}&=\left(h^{(t)}_{1}, h^{(t)}_{2} ,\cdots,h^{(t)}_{N}\right)^\top _{N\times{d}},\\
    \mathbb{E}_{t-1}[\mathbf{\omega}^{(t)}]&=\left(\mathbb{E}_{t-1}[{\omega}^{(t)}_{1}],\mathbb{E}_{t-1}[{\omega}^{(t)}_{2}],\cdots,\mathbb{E}_{t-1}[{\omega}^{(t)}_{N}]\right)^\top _{N\times{d}},\\
    \mathbb{E}_{t-1}[\overline{\mathbf{\omega}^{(t)}}]&=\left(\mathbb{E}_{t-1}[\bar{\omega}^{(t)}],\mathbb{E}_{t-1}[\bar{\omega}^{(t)}],\cdots,\mathbb{E}_{t-1}[\bar{\omega}^{(t)}]\right)^\top _{N\times{d}},\\
    \mathbf{v}^{(t)}&=\left(v^{(t)}_{1},v^{(t)}_{2} ,\cdots,v^{(t)}_{N}\right)^\top _{N\times{p}},\\
    \mathbf{r}^{(t)}&=\left(p_{H,1}^{(t)}-u_{2,1}^{(t)},p_{H,2}^{(t)}-u_{2,2}^{(t)},\cdots,p_{H,N}^{(t)}-u_{2,N}^{(t)}\right)^\top _{N\times{d}}.
\end{align*}

The gradient of lower-level loss functions of different nodes can also be presented as the following matrix forms:
\begin{equation*}
    \begin{aligned}
        \nabla_2\mathbf{g}^{(t+1)}&=\left(\nabla_2g_1(x_{1}^{(t)},y_{1}^{(t)}),\nabla_2g_2(x_{2}^{(t)},y_{2}^{(t)}) ,\cdots,\nabla_2g_N(x_{N}^{(t)},y_{N}^{(t)})\right)^\top _{N\times{p}}.
    \end{aligned}
\end{equation*}

Besides, we denote $\Delta_{t}^2$ as the consensus error of the $t$-th round of iteration, \ie,
\begin{align*}
    \Delta_{t}^2=\kappa^2\left[\left\Vert \mathbf{x}^{(t)}-\bar{\mathbf{x}}^{(t)} \right\Vert_F^2+\left\Vert \mathbf{y}^{(t)}-\bar{\mathbf{y}}^{(t)} \right\Vert_F^2\right]+\left\Vert \mathbf{z}^{(t)}-\bar{\mathbf{z}}^{(t)} \right\Vert_F^2.
\end{align*}

We also assume that the step-sizes are fixed over different rounds, which means that there exist constants $c_1,c_2,c_3>0$ that
\begin{equation*}
    \begin{aligned}
        &\alpha_1=\alpha_2=\cdots=\alpha_T=\alpha,\hspace{34pt} \beta_1=\beta_2=\cdots=\beta_T=\beta=c_1\alpha,\\
        &\gamma_1=\gamma_2=\cdots=\gamma_T=\gamma=c_2\alpha,\quad \theta_1=\theta_2=\cdots=\theta_T=\theta=c_3\alpha.
    \end{aligned}
\end{equation*}

\subsection{Technical Lemmas}
\begin{lemma}\label{des} Suppose $f(x)$ is $\mu$-strongly convex and $L$- smooth. For any $x$ and $\eta<\frac{2}{\mu+L}$, define $x^{+}=x-\eta \nabla f(x), x^\star=$ $\arg \min f(x)$. Then we have
$$
\left\|x^{+}-x^\star\right\| \leq(1-\eta \mu)\left\|x-x^\star\right\|.
$$
\end{lemma}
\begin{proof}
    See Lemma 10 in \cite{qu2018harnessing}.
\end{proof}

\begin{lemma}
Suppose Assumption \ref{assumption:smooth} holds, then $\nabla\Phi(x)$ is $L_{\nabla\Phi}$-Lipschitz continuous, where
\begin{align*}
    L_{\nabla\Phi}=&L_{\nabla f}+\dfrac{2L_{\nabla f}L_{\nabla g}+L_f^2L_{\nabla^2g}}{\mu_g}+\dfrac{2L_fL_{\nabla g}L_{\nabla^2g}+L_{\nabla f}L_{\nabla g}^2}{\mu_g^2}+\dfrac{L_fL_{\nabla^2g}L_{\nabla g}^2}{\mu_g^3}=\mathcal{O}(\kappa^3).
\end{align*}

Also, if we define $y^{\star}(x)=\mathop{\arg\min}_{y}g(x,y)$, $z^{\star}(x)=(\nabla_{22}^2g(x,y^{\star}(x)))^{-1}\nabla_2f(x,y^{\star}(x))$, then $y^{\star}(x)$ are $L_{y^{\star}}$-Lipschitz continuous, and $z^{\star}(x)$ are $L_{z^{\star}}$-Lipschitz continuous, where 
\begin{align*}
    L_{y^{\star}}=&\dfrac{L_{\nabla g}}{\mu_g}=\mathcal{O}(\kappa),\quad L_{z^{\star}}=\sqrt{1+L_{y^{\star}}^2}\left(\dfrac{L_{\nabla f}}{\mu_g}+\dfrac{L_fL_{\nabla^2g}}{\mu_g^2}\right)=\mathcal{O}(\kappa^3).
\end{align*}

And we also have:
\begin{align}
\label{est:z*}
    \left\Vert z^{\star}(x)\right\Vert\leq\dfrac{L_f}{\mu_g},\quad\forall x\in\mathbb{R}^d.
\end{align}
\end{lemma}
\begin{proof}
    Due to Assumption \ref{assumption:smooth} and its remark, the proof of this lemma is the same as the single agent case which is shown in \citep[Lemma B.2.]{chen2023optimal}.
\end{proof}

{
\begin{remark}
If we define $L:=\max\left\{L_f,L_{\nabla f},L_{\nabla g}, L_{\nabla^2 g}\right\}$, then the order of $L_{\nabla \Phi}$ is actually $\mathcal{O}(L\kappa^3)$. We follow the existing works \cite{anonymous2024sparkle,chen2023decentralized} to view $\kappa = L/\mu$ and the smooth term $L:=\max\{L_f,L_{\nabla f},L_{\nabla g}, L_{\nabla^2 g}\}$ as a constant.
\end{remark}
}

\subsection{Error Analysis of the Hessian/Jacobian vector product evaluation}
Here, we analysis the error of the Hessian/Jacobian vector product evaluation. The following lemma presents the error of the Hessian/Jacobian vector product evaluation in both D-SOBA-SO and D-SOBA-FO. 

\begin{lemma}
\label{desjifen} 
Suppose Assumption \ref{assumption:smooth} and \ref{assumption:unbiased} hold, then the terms $p_{H,i}^{(t)}$, $p_{J,i}^{(t)}$ obtained from D-SOBA-SO satisfy:
\begin{subequations}
    \begin{align}
        \label{chafen_H}
        \mathbb{E}_t\left\Vert p_{H,i}^{(t)}-\mathbb{E}_t[p_{H,i}^{(t)}]\right\Vert^2\leq\sigma^2\left\Vert z_i^{(t)}\right\Vert^2,\quad\left\Vert\mathbb{E}_t[p_{H,i}^{(t)}]-\nabla_{22}g_i(x_i^{(t)},y_i^{(t)})z_i^{(t)}\right\Vert^2\leq\iota^2;\\
        \label{chafen_J}
        \mathbb{E}_t\left\Vert p_{J,i}^{(t)}-\mathbb{E}_t[p_{J,i}^{(t)}]\right\Vert^2\leq\sigma^2\left\Vert z_i^{(t)}\right\Vert^2,\quad\left\Vert\mathbb{E}_t[p_{J,i}^{(t)}]-\nabla_{12}g_i(x_i^{(t)},y_i^{(t)})z_i^{(t)}\right\Vert^2\leq\iota^2,
    \end{align}
\end{subequations}
for all $\iota>0$. If we choose $\delta_t\leq\sqrt{3}\iota{L_{\nabla^2g}^{-1}\Vert z_i^{(t)}\Vert^{-2}}$ in D-SOBA-FO, \eqref{chafen_H} and \eqref{chafen_J} also hold.
\end{lemma}
\begin{proof}
    Firstly, we consider D-SOBA-FO. From assumption \ref{assumption:unbiased}, we know that the Hessian of $g_i$ with respect to a stochastic sample $\zeta_i^{(t)}\in\mathcal{D}_{g_i}$ exists for all $x\in\mathbb{R}^d$ and $y\in\mathbb{R}^p$. Then, for the term $p_{H,i}^{(t)}$ obtained from D-SOBA-FO, we have:
    \begin{equation}
    \begin{aligned}
    &\mathbb{E}_t\left\Vert p_{H,i}^{(t)}-\mathbb{E}_t[p_{H,i}^{(t)}]\right\Vert^2\\
    =&\mathbb{E}_t\left\Vert\dfrac{1}{2\delta_t}\int_{-1}^1\left(\nabla_{22}^2G(x_i^{(t)},y_i^{(t)}+t(\delta_t z_i^{(t)});\zeta_i^{(t)})-\nabla_{22}^2g_i(x_i^{(t)},y_i^{(t)}+t(\delta_t z_i^{(t)}))\right)(\delta_t z_i^{(t)})dt\right\Vert^2\\
    \leq&\dfrac{1}{4\delta_t^2}\int_{-1}^12\mathbb{E}_t\left\Vert\nabla_{22}^2G(x_i^{(t)},y_i^{(t)}+t(\delta_t z_i^{(t)});\zeta_i^{(t)})-\nabla_{22}^2g_i(x_i^{(t)},y_i^{(t)}+t(\delta_t z_i^{(t)}))\right\Vert^2\left\Vert\delta_t z_i^{(t)}\right\Vert^2dt\\
    \leq&\dfrac{1}{\delta_t^2}\cdot\sigma^2\left\Vert\delta_t z_i^{(t)}\right\Vert^2\leq\sigma^2\left\Vert z_i^{(t)}\right\Vert^2.
    \end{aligned}    
    \end{equation}

    Then, we consider the error between $\mathbb{E}_t[p_{H,i}^{(t)}]$ and $\nabla_{22}g_i(x_i^{(t)},y_i^{(t)})z_i^{(t)}$. We can obtain:
    \begin{equation}
    \begin{aligned}
    &\left\Vert\mathbb{E}_t[p_{H,i}^{(t)}]-\nabla_{22}g_i(x_i^{(t)},y_i^{(t)})z_i^{(t)}\right\Vert^2\\
    =&\left\Vert\dfrac{1}{2\delta_t}\int_{-1}^1\left(\nabla_{22}^2g_i(x_i^{(t)},y_i^{(t)}+t(\delta_t z_i^{(t)}))-\nabla_{22}^2g_i(x_i^{(t)},y_i^{(t)})\right)(\delta_t z_i^{(t)})dt\right\Vert^2\\
    \leq&\dfrac{1}{4\delta_t^2}\int_{-1}^12\left\Vert\nabla_{22}^2g_i(x_i^{(t)},y_i^{(t)}+t(\delta_t z_i^{(t)}))-\nabla_{22}^2g_i(x_i^{(t)},y_i^{(t)})\right\Vert^2\left\Vert\delta_t z_i^{(t)}\right\Vert^2dt\leq\dfrac{1}{3}L^2_{\nabla^2g}\delta_t^2\left\Vert z_i^{(t)}\right\Vert^4.
    \end{aligned}    
    \end{equation}
    
    Similarly, for $p_{J,i}^{(t)}$, we have:
    \begin{equation}
    \begin{aligned}
    &\mathbb{E}_t\left\Vert p_{J,i}^{(t)}-\mathbb{E}_t[p_{J,i}^{(t)}]\right\Vert^2\\
    =&\mathbb{E}_t\left\Vert\dfrac{1}{2\delta_t}\int_{-1}^1\left(\nabla_{12}^2G(x_i^{(t)},y_i^{(t)}+t(\delta_t z_i^{(t)});\zeta_i^{(t)})-\nabla_{12}^2g(x_i^{(t)},y_i^{(t)}+t(\delta_t z_i^{(t)}))\right)(\delta_t z_i^{(t)})dt\right\Vert^2\\
    \leq&\dfrac{1}{4\delta_t^2}\int_{-1}^12\mathbb{E}_t\left\Vert\nabla_{12}^2G(x_i^{(t)},y_i^{(t)}+t(\delta_t z_i^{(t)});\zeta_i^{(t)})-\nabla_{12}^2g(x_i^{(t)},y_i^{(t)}+t(\delta_t z_i^{(t)}))\right\Vert^2\left\Vert\delta_t z_i^{(t)}\right\Vert^2dt\\
    \leq&\dfrac{1}{\delta_t^2}\cdot\sigma^2\left\Vert\delta_t z_i^{(t)}\right\Vert^2\leq\sigma^2\left\Vert z_i^{(t)}\right\Vert^2.
    \end{aligned}    
    \end{equation}

    And the error between $\mathbb{E}_t[p_{J,i}^{(t)}]$ and $\nabla^2_{12}g_i(x_i^{(t)},y_i^{(t)})z_i^{(t)}$ can be bounded by:
    \begin{equation}
    \begin{aligned}
    &\left\Vert\mathbb{E}_t[p_{J,i}^{(t)}]-\nabla_{12}g_i(x_i^{(t)},y_i^{(t)})z_i^{(t)}\right\Vert^2\\
    =&\left\Vert\dfrac{1}{2\delta_t}\int_{-1}^1\left(\nabla_{12}^2g_i(x_i^{(t)},y_i^{(t)}+t(\delta_t z_i^{(t)}))-\nabla_{12}^2g_i(x_i^{(t)},y_i^{(t)})\right)(\delta_t z_i^{(t)})dt\right\Vert^2\\
    \leq&\dfrac{1}{4\delta_t^2}\int_{-1}^1\left\Vert2\nabla_{12}^2g_i(x_i^{(t)},y_i^{(t)}+t(\delta_t z_i^{(t)}))-\nabla_{12}^2g_i(x_i^{(t)},y_i^{(t)})\right\Vert^2\left\Vert\delta_t z_i^{(t)}\right\Vert^2dt\leq\dfrac{1}{3}L^2_{\nabla^2g}\delta_t^2\left\Vert z_i^{(t)}\right\Vert^4.
    \end{aligned}    
    \end{equation}
    Let $\delta_t\leq\sqrt{3}\iota{L_{\nabla^2g}^{-1}\Vert z_i^{(t)}\Vert^{-2}}$, then \eqref{chafen_H} and \eqref{chafen_J} hold in D-SOBA-FO.

    Then we consider the scenario of D-SOBA-SO. From Assumption \ref{assumption:unbiased}, we have:
\begin{subequations}
    \begin{align}
        \label{chafen_H1}
        \mathbb{E}_t\left\Vert p_{H,i}^{(t)}-\mathbb{E}_t[p_{H,i}^{(t)}]\right\Vert^2\leq\mathbb{E}_t\left[\left\Vert \nabla_{22}^2G(x_i^{(t)},y_i^{(t)};\zeta_i^{(t)})-\nabla_{22}^2g_i(x_i^{(t)},y_i^{(t)})\right\Vert^2\right]\left\Vert z_i^{(t)}\right\Vert^2\leq\sigma^2\left\Vert z_i^{(t)}\right\Vert^2;\\
        \label{chafen_J1}
        \mathbb{E}_t\left\Vert p_{J,i}^{(t)}-\mathbb{E}_t[p_{J,i}^{(t)}]\right\Vert^2\leq\mathbb{E}_t\left[\left\Vert \nabla_{12}^2G(x_i^{(t)},y_i^{(t)};\zeta_i^{(t)})-\nabla_{12}^2g_i(x_i^{(t)},y_i^{(t)})\right\Vert^2\right]\left\Vert z_i^{(t)}\right\Vert^2\leq\sigma^2\left\Vert z_i^{(t)}\right\Vert^2.
    \end{align}
\end{subequations}
{We also have $\left\Vert\mathbb{E}_t[p_{H,i}^{(t)}]-\nabla_{22}g_i(x_i^{(t)},y_i^{(t)})z_i^{(t)}\right\Vert^2=\left\Vert\mathbb{E}_t[p_{J,i}^{(t)}]-\nabla_{12}g_i(x_i^{(t)},y_i^{(t)})z_i^{(t)}\right\Vert^2=0$ from Assumption \ref{assumption:unbiased} for D-SOBA-SO.} So \eqref{chafen_H} and \eqref{chafen_J} also hold in D-SOBA-SO.
\end{proof}

With Lemma \ref{desjifen}, we can directly obtain that Proposition \ref{error of finite} is also satisfied.

\subsection{Descent Lemma}
In this subsection, we firstly present the descent lemma.
\begin{lemma}[Descent Lemma]
    Suppose Assumption \ref{assumption:smooth} and \ref{assumption:unbiased} hold, and the step-sizes $\alpha$, $\gamma$ satisfy:
    \begin{align}
    \label{eq:assumptionhyper}
        \alpha\leq\min\left\{\dfrac{c_3}{2\tau(c_3L_{\nabla\Phi}+L_{\nabla\eta})},\dfrac{\tau^2}{20c_3}\right\},\quad 0<\tau<\sqrt{20c_3}\leq1,\quad\gamma\mu_g\leq1.
    \end{align}
    Then in Algorithm \ref{alg:DeMA-SOBA}, we have:
    \begin{equation}
    \label{eq:final_avgh11}
        \begin{aligned}
            &\alpha\sum_{t=0}^{T-1} \mathbb{E}\left[\left\Vert\bar{h}^{(t)}\right\Vert^2\right]+\alpha\sum_{t=0}^{T} \mathbb{E}\left[\left\Vert\bar{h}^{(t)}-\nabla\Phi(\bar{x}^{(t)})\right\Vert^2\right]\\
            \leq&\dfrac{2}{\tau}\mathbb{E}[E_0]+{\dfrac{1}{c_3}\mathbb{E}\left[\left\Vert\bar{h}^{(0)}-\nabla\Phi(\bar{x}^{(0)})\right\Vert^2\right]}+\dfrac{4L_{\nabla\Phi}^2\tau^2}{c_3^2}\alpha\sum_{t=0}^{T-1} \mathbb{E}\left[\left\Vert\bar{h}^{(t)}\right\Vert^2\right]\\
            &+\dfrac{1}{2}\cdot\alpha\sum_{t=0}^{T-1} \mathbb{E}\left[\left\Vert\bar{h}^{(t)}-\nabla\Phi(\bar{x}^{(t)})\right\Vert^2\right]+\left(60L^2+\dfrac{6(1+c_3)\alpha\sigma^2}{N}\right)\dfrac{1}{N}\cdot\alpha\sum_{t=0}^{T-1}\mathbb{E}[\Delta_t^2]\\&+\dfrac{6(1+c_3)T\alpha^2}{N}\left(1+\dfrac{L_f^2}{\mu_g^2}\right)\sigma^2+10T\alpha\iota^2            +60\left(\dfrac{L_f^2L_{\nabla^2g}^2}{\mu_g^2}+L_{\nabla f}^2\right)\sum_{t=0}^{T-1}\alpha\mathbb{E}\left[\left\Vert \bar{y}^{(t)}-\bar{y}^{(t)}_{\star}\right\Vert^2\right]\\
            &+\left(60L_{\nabla^2g}^2+\dfrac{6(1+c_3)\alpha\sigma^2}{N}\right)\sum_{t=0}^{T-1}\alpha\mathbb{E}\left[\left\Vert \bar{z}^{(t)}-\bar{z}^{(t)}_{\star}\right\Vert^2\right].
        \end{aligned}
    \end{equation}
\end{lemma}
\begin{proof}
    Define
    \begin{align}
    \label{define_eta}
        \eta(x,h,\tau)=\min_{y\in\mathbb{R}^p}\left[\langle h,y-x\rangle+\dfrac{1}{2\tau}\left\Vert y-x\right\Vert^2\right].
    \end{align}
    From \citep[Lemma 3.2]{ghadimi2020single}, the minimum problem in Eq. \eqref{define_eta} has the solution $y^{\star}=x-\tau h$, and  $\eta(x,h,\tau)$ is $L_{\nabla\eta}$-smooth, where
    \begin{align}
        L_{\nabla\eta}=2\sqrt{\left(1+\dfrac{1}{\tau}\right)^2+\left(1+\dfrac{\tau}{2}\right)^2}.
    \end{align}
    As $\Phi(x)$ is $L_{\nabla\Phi}$-smooth, we have:
    \begin{align}
    \label{eq:lsmoothPhi}
        \Phi(\bar{x}^{(t+1)})-\Phi(\bar{x}^{(t)})\leq\tau\alpha\langle\nabla\Phi(\bar{x}^{(t)}),-\bar{h}^{(t)}\rangle+\dfrac{L_{\nabla\Phi}}{2}\left\Vert\bar{x}^{(t+1)}-\bar{x}^{(t)}\right\Vert^2.
    \end{align}
    Also, according to \citep[Lemma B.8.]{chen2023optimal}, we have:
    \begin{equation}
    \begin{aligned}
    \label{eq:lsmootheta}
        &\eta(\bar{x}^{(t)},\bar{h}^{(t)},\tau)-\eta(\bar{x}^{(t+1)},\bar{h}^{(t+1)},\tau)\\
        \leq&-\theta\tau\left\Vert\bar{h}^{(t)}\right\Vert^2+\theta\tau\langle\bar{\omega}^{(t+1)},\bar{h}^{(t)}\rangle+\dfrac{L_{\nabla\eta}}{2}\left(\left\Vert\bar{x}^{(t+1)}-\bar{x}^{(t)}\right\Vert^2+\left\Vert\bar{h}^{(t+1)}-\bar{h}^{(t)}\right\Vert^2\right).
    \end{aligned}
    \end{equation}
    Define $E_t=\Phi(\bar{x}^{(t)})-\inf_{x}\Phi(x)-c_3^{-1}\eta(\bar{x}^{(t)},\bar{h}^{(t)},\tau)$. 
    Substituting \eqref{eq:lsmoothPhi} into \eqref{eq:lsmootheta} and taking the expectation with respect to $\mathcal{F}^{(t)}$, we have:
    \begin{equation}
    \label{eq:midd_ht}
    \begin{aligned}
        &\alpha\left\Vert\bar{h}^{(t)}\right\Vert^2\\
        \leq&\dfrac{1}{c_3\tau}\left(\mathbb{E}_t\left[\eta(\bar{x}^{(t+1)},\bar{h}^{(t+1)},\tau)\right]-\eta(\bar{x}^{(t)},\bar{h}^{(t)},\tau)\right)+\dfrac{1}{\tau}\left(\Phi(\bar{x}^{(t)})-\mathbb{E}_t\left[\Phi(\bar{x}^{(t+1)})\right]\right)\\
        &-\alpha\mathbb{E}_t\left[\left\langle\nabla\Phi(\bar{x}^{(t)})-\bar{\omega}^{(t+1)},\bar{h}^{(t)}\right\rangle\right]+\dfrac{c_3 L_{\nabla\Phi}+L_{\nabla\eta}}{2c_3\tau}\left\Vert\bar{x}^{(t+1)}-\bar{x}^{(t)}\right\Vert^2\\
        &+\dfrac{L_{\nabla\eta}}{2c_3\tau}\mathbb{E}_t\left[\left\Vert\bar{h}^{(t+1)}-\bar{h}^{(t)}\right\Vert^2\right]\\
        \leq&\dfrac{1}{\tau}\left(E_t-\mathbb{E}_t\left[E_{t+1}\right]\right)+\alpha\left(\left\Vert\nabla\Phi(\bar{x}^{(t)})-\mathbb{E}_t[\bar{\omega}^{(t+1)}]\right\Vert^2+\dfrac{1}{4}\left\Vert\bar{h}^{(t)}\right\Vert^2\right)+\dfrac{\alpha}{4}\left\Vert\bar{h}^{(t)}\right\Vert^2 \\
        &+\dfrac{\alpha^2}{4}\mathbb{E}_t\left[\left\Vert\nabla\Phi(\bar{x}^{(t)})-\bar{\omega}^{(t+1)}\right\Vert^2\right]+\dfrac{\alpha^2}{4}\mathbb{E}_t\left[\left\Vert\bar{h}^{(t)}-\nabla\Phi(\bar{x}^{(t)})\right\Vert^2\right]\\
        \leq&\dfrac{1}{\tau}\left(E_t-\mathbb{E}_t\left[E_{t+1}\right]\right)+\alpha\left(\left\Vert\nabla\Phi(\bar{x}^{(t)})-\mathbb{E}_t[\bar{\omega}^{(t+1)}]\right\Vert^2+\dfrac{1}{4}\left\Vert\bar{h}^{(t)}\right\Vert^2\right)+\dfrac{\alpha}{4}\left\Vert\bar{h}^{(t)}\right\Vert^2 \\
        &+\dfrac{\alpha}{4}\mathbb{E}_t\left[\left\Vert\nabla\Phi(\bar{x}^{(t)})-\bar{\omega}^{(t+1)}\right\Vert^2\right]+\dfrac{\alpha^2}{4}\left\Vert\bar{\omega}^{(t+1)}-\mathbb{E}_t[\bar{\omega}^{(t+1)}]\right\Vert^2\\
        &+\dfrac{\alpha}{4}\mathbb{E}_t\left[\left\Vert\bar{h}^{(t)}-\nabla\Phi(\bar{x}^{(t)})\right\Vert^2\right],
    \end{aligned}
    \end{equation}
    where the first inequality uses the fact that $\bar{x}^{(t)}$ and $\bar{h}^{(t)}$ are measurable with respect to $\mathcal{F}_t$. The second inequality uses the assumption  $\alpha\leq\dfrac{c_3}{2\tau(c_3L_{\nabla\Phi}+L_{\nabla\eta})}$, 
    $\dfrac{c_3}{\tau^2}\leq\dfrac{1}{20}$, and the fact that $L_{\nabla\eta}\leq\dfrac{5}{\tau}$. The third inequality uses the assumption that $\alpha\leq1$.

    Then, taking the expectation and summation on both sides of \eqref{eq:midd_ht}, we have:
    \begin{equation}
    \label{eq:final_avgh}
        \begin{aligned}
            &\alpha\sum_{t=0}^{T-1} \mathbb{E}\left[\left\Vert\bar{h}^{(t)}\right\Vert^2\right]\\
            \leq&\dfrac{2}{\tau}\mathbb{E}[E_0]+\alpha^2\sum_{t=0}^{T-1}\mathbb{E}\left[\left\Vert\bar{\omega}^{(t+1)}-\mathbb{E}_t[\bar{\omega}^{(t+1)}]\right\Vert^2\right]+3\alpha\sum_{t=0}^{T-1} \mathbb{E}\left[\left\Vert\mathbb{E}_t[\bar{\omega}^{(t+1)}]-\nabla\Phi(\bar{x}^{(t)})\right\Vert^2\right]\\
            &+\dfrac{1}{2}\cdot\alpha\sum_{t=0}^{T-1} \mathbb{E}\left[\left\Vert\bar{h}^{(t)}-\nabla\Phi(\bar{x}^{(t)})\right\Vert^2\right]\\
            \leq&\dfrac{2}{\tau}\mathbb{E}[E_0]+\alpha^2\sum_{t=0}^{T-1}\mathbb{E}\left[\left\Vert\bar{\omega}^{(t+1)}-\mathbb{E}_t[\bar{\omega}^{(t+1)}]\right\Vert^2\right]+6\alpha\sum_{t=0}^{T-1} \mathbb{E}\left[\left\Vert\mathbb{E}_t[\bar{\omega}^{(t+1)}]-\bar{q}^{(t)}\right\Vert^2\right]\\
            &+6\alpha\sum_{t=0}^{T-1} \mathbb{E}\left[\left\Vert\bar{q}^{(t)}-\nabla\Phi(\bar{x}^{(t)})\right\Vert^2\right]
            +\dfrac{1}{2}\cdot\alpha\sum_{t=0}^{T-1} \mathbb{E}\left[\left\Vert\bar{h}^{(t)}-\nabla\Phi(\bar{x}^{(t)})\right\Vert^2\right].
        \end{aligned}
    \end{equation}

    Then we consider the term $\left\Vert\bar{h}^{(t)}-\nabla\Phi(\bar{x}^{(t)})\right\Vert^2$. Note that:
    \begin{equation}
    \begin{aligned}
        \bar{h}^{(t+1)}-\nabla\Phi(\bar{x}^{(t+1)})=&(1-\theta)\left(\bar{h}^{(t+1)}-\nabla\Phi(\bar{x}^{(t)})\right)+\theta\left(\bar{q}^{(t)}-\nabla\Phi(\bar{x}^{(t)})\right)\\
        &+\nabla\Phi(\bar{x}^{(t)})-\nabla\Phi(\bar{x}^{(t+1)})+\theta\left(\bar{\omega}^{(t+1)}-\bar{q}^{(t)}\right).
    \end{aligned}
    \end{equation}
    Taking the conditional expectation with respect to $\mathcal{F}^{(t)}$ and noting that $\bar{x}^{(t+1)}$ is measurable with respect to $\mathcal{F}_t$, we have:
    \begin{equation}
        \begin{aligned}
            &\mathbb{E}_t\left[\left\Vert\bar{h}^{(t+1)}-\nabla\Phi(\bar{x}^{(t+1)})\right\Vert^2\right]\\
            \leq&(1-\theta)\mathbb{E}_t\left[\left\Vert\bar{h}^{(t)}-\nabla\Phi(\bar{x}^{(t)})\right\Vert^2\right]+2\theta\mathbb{E}_t\left[\left\Vert\mathbb{E}_t[\bar{\omega}^{(t+1)}]-\bar{q}^{(t)}\right\Vert^2\right]\\
            &+2\theta\mathbb{E}_t\left[\left\Vert\left(\bar{q}^{(t)}-\nabla\Phi(\bar{x}^{(t)})\right)+\dfrac{1}{\theta}\left(\nabla\Phi(\bar{x}^{(t)})-\nabla\Phi(\bar{x}^{(t+1)})\right)\right\Vert^2\right]+\theta^2\left\Vert\bar{\omega}^{(t+1)}-\mathbb{E}_t[\bar{\omega}^{(t+1)}]\right\Vert^2\\
            \leq&(1-\theta)\mathbb{E}_t\left[\left\Vert\bar{h}^{(t)}-\nabla\Phi(\bar{x}^{(t)})\right\Vert^2\right]+2\theta\mathbb{E}_t\left[\left\Vert\mathbb{E}_t[\bar{\omega}^{(t+1)}]-\bar{q}^{(t)}\right\Vert^2\right]+4\theta\mathbb{E}_t\left[\left\Vert\bar{q}^{(t)}-\nabla\Phi(\bar{x}^{(t)})\right\Vert^2\right]\\
            &+\dfrac{4L_{\nabla\Phi}^2\tau^2\alpha^2}{\theta}\mathbb{E}_t\left[\left\Vert\bar{h}^{(t)}\right\Vert^2\right]+\theta^2\left\Vert\bar{\omega}^{(t+1)}-\mathbb{E}_t[\bar{\omega}^{(t+1)}]\right\Vert^2
        \end{aligned}
    \end{equation}
    where the first inequality is due to the convex of 2-norm $\left\Vert\cdot\right\Vert^2$, and the last equality is from the Lipchitz continuity of $\nabla\Phi$. 
    
    Taking expectation on both sides and then taking summation from $t=0$ to $T-1$ and, we have:
    \begin{equation}
    \label{eq:final_h-PHI}
        \begin{aligned}
            &\alpha\sum_{t=0}^{T} \mathbb{E}\left[\left\Vert\bar{h}^{(t)}-\nabla\Phi(\bar{x}^{(t)})\right\Vert^2\right]\\
            \leq&{\dfrac{1}{c_3}\mathbb{E}\left[\left\Vert\bar{h}^{(0)}-\nabla\Phi(\bar{x}^{(0)})\right\Vert^2\right]}+2\alpha\sum_{t=0}^{T-1} \mathbb{E}\left[\left\Vert\mathbb{E}_t[\bar{\omega}^{(t+1)}]-\bar{q}^{(t)}\right\Vert^2\right]\\
            &+4\alpha\sum_{t=0}^{T-1} \mathbb{E}\left[\left\Vert\bar{q}^{(t)}-\nabla\Phi(\bar{x}^{(t)})\right\Vert^2\right]+\dfrac{4L_{\nabla\Phi}^2\tau^2}{c_3^2}\alpha\sum_{t=0}^{T-1} \mathbb{E}\left[\left\Vert\bar{h}^{(t)}\right\Vert^2\right]\\
            &+c_3\alpha^2\sum_{t=0}^{T-1}\mathbb{E}\left[\left\Vert\bar{\omega}^{(t+1)}-\mathbb{E}_t[\bar{\omega}^{(t+1)}]\right\Vert^2\right].
        \end{aligned}
    \end{equation}

Then, we consider the term $\left\Vert \bar{q}^{(t)}-\nabla\Phi(\bar{x}^{(t)})\right\Vert^2$. Note that:
    \begin{equation}
        \begin{aligned}
            \bar{q}^{(t)}-\nabla\Phi(\bar{x}^{(t)})
            =&\dfrac{1}{N}\sum_{i=1}^N\left[\left(\nabla_1f_i(x^{(t)}_{i},y^{(t)}_{i})-\nabla_1f_i(\bar{x}^{(t)},\bar{y}^{(t)}_{\star})\right)-\nabla_{22}g_i(x^{(t)}_{i},y^{(t)}_{i})(z^{(t)}_{i}-\bar{z}^{(t)}_{\star})\right]\\
            &-\dfrac{1}{N}\sum_{i=1}^N\left(\nabla_{22}g_i(x^{(t)}_{i},y^{(t)}_{i})-\nabla_{22}g_n(\bar{x}^{(t)},\bar{y}^{(t)}_{\star})\right)\bar{z}^{(t)}_{\star}.
        \end{aligned}
    \end{equation}
    Then, taking the norm on both sides and using Assumption \ref{assumption:smooth}, we have:
    \begin{equation}
    \label{eq:final_omega-phi}
        \begin{aligned}
            &\left\Vert\bar{q}^{(t)}-\nabla\Phi(\bar{x}^{(t)})\right\Vert^2\\
            \leq&3\cdot\dfrac{1}{N}\sum_{i=1}^N\left\Vert\nabla_1f_i(x^{(t)}_{i},y^{(t)}_{i})-\nabla_1f_i(\bar{x}^{(t)},\bar{y}^{(t)}_{\star})\right\Vert^2+3\cdot\dfrac{1}{N}\sum_{i=1}^N\left\Vert\nabla_{22}g_i(x^{(t)}_{i},y^{(t)}_{i})\right\Vert^2\left\Vert z^{(t)}_{i}-\bar{z}^{(t)}_{\star}\right\Vert^2\\
            &+3\cdot\dfrac{1}{N}\sum_{i=1}^N\left\Vert\nabla_{22}g_i(x^{(t)}_{i},y^{(t)}_{i})-\nabla_{22}g_i(\bar{x}^{(t)},\bar{y}^{(t)}_{\star})\right\Vert^2\left\Vert\bar{z}^{(t)}_{\star}\right\Vert^2\\
            \leq&6\left(\dfrac{L_f^2L_{\nabla^2g}^2}{\mu_g^2}+L_{\nabla f}^2\right)\dfrac{1}{N}\left(\left\Vert\mathbf{x}^{(t)}-\mathbf{\bar{x}}^{(t)}\right\Vert_F^2+\left\Vert\mathbf{y}^{(t)}-\mathbf{\bar{y}}^{(t)}\right\Vert_F^2+N\left\Vert \bar{y}^{(t)}-\bar{y}^{(t)}_{\star}\right\Vert^2\right)\\
            &+6L_{\nabla^2g}^2\cdot\dfrac{1}{N}\sum_{i=1}^N\left(\left\Vert z^{(t)}_{i}-\bar{z}^{(t)}\right\Vert^2+\left\Vert \bar{z}^{(t)}-\bar{z}^{(t)}_{\star}\right\Vert^2\right)\\
            \leq&6L^2\dfrac{1}{N}\Delta_t^2+6\left(\dfrac{L_f^2L_{\nabla^2g}^2}{\mu_g^2}+L_{\nabla f}^2\right)\left\Vert \bar{y}^{(t)}-\bar{y}^{(t)}_{\star}\right\Vert^2+6L_{\nabla^2g}^2\left\Vert  \bar{z}^{(t)}-\bar{z}^{(t)}_{\star}\right\Vert^2.
        \end{aligned}
    \end{equation} 

    Next, considering the term $\left\Vert\mathbb{E}_t[\bar{\omega}^{(t+1)}]-\bar{q}^{(t)}\right\Vert^2$, we have:
    \begin{equation}
        \begin{aligned}
            \left\Vert\mathbb{E}_t[\bar{\omega}^{(t+1)}]-\bar{q}^{(t)}\right\Vert^2\leq\dfrac{1}{N}\sum_{i=1}^N\left\Vert \mathbb{E}_t[p_{J,i}^{(t)}]-\nabla_{12}^2g_i(x_i^{(t)},y_i^{(t)})z_i^{(t)}\right\Vert^2
            \leq\iota^2.
        \end{aligned}
    \end{equation}

    Finally, we consider the term$\left\Vert\bar{\omega}^{(t+1)}-\mathbb{E}_t[\bar{\omega}^{(t+1)}]\right\Vert^2$. From Assumption \ref{assumption:unbiased} and Eq. \eqref{chafen_J}, we have:
    \begin{equation}
        \begin{aligned}
            \left\Vert\bar{\omega}^{(t+1)}-\mathbb{E}_t[\bar{\omega}^{(t+1)}]\right\Vert^2\leq&\dfrac{2}{N^2}\sum_{i=1}^N\left[\left\Vert p_{J,i}^{(t)}-\mathbb{E}_t[p_{J,i}^{(t)}]\right\Vert^2+\left\Vert u_{1,i}^t-\mathbb{E}_t[u_{1,i}^t]\right\Vert^2\right]\\
            \leq&\dfrac{2}{N^2}\sum_{i=1}^N\left(1+\left\Vert z_i^{(t)}\right\Vert^2\right)\sigma^2\\
            \leq&\dfrac{6}{N}\left(1+\dfrac{L_f^2}{\mu_g^2}\right)\sigma^2+\dfrac{6\sigma^2}{N}\left\Vert\bar{z}^{(t)}-\bar{z}_{\star}^{(t)}\right\Vert^2+\dfrac{6\sigma^2}{N}\dfrac{1}{N}\sum_{i=1}^N\left\Vert {z}_i^{(t)}-\bar{z}^{(t)}\right\Vert^2.
        \end{aligned}
    \end{equation}
    
    Adding \eqref{eq:final_avgh} and \eqref{eq:final_h-PHI} together, and then plugging \eqref{eq:final_omega-phi} into it, we have:
    \begin{equation}\label{H1_new}
        \begin{aligned}
            &\alpha\sum_{t=0}^{T-1} \mathbb{E}\left[\left\Vert\bar{h}^{(t)}\right\Vert^2\right]+\alpha\sum_{t=0}^{T} \mathbb{E}\left[\left\Vert\bar{h}^{(t)}-\nabla\Phi(\bar{x}^{(t)})\right\Vert^2\right]\\
            \leq&\dfrac{2}{\tau}\mathbb{E}[E_0]+{\dfrac{1}{c_3}\mathbb{E}\left[\left\Vert\bar{h}^{(0)}-\nabla\Phi(\bar{x}^{(0)})\right\Vert^2\right]}+\dfrac{4L_{\nabla\Phi}^2\tau^2}{c_3^2}\alpha\sum_{t=0}^{T-1} \mathbb{E}\left[\left\Vert\bar{h}^{(t)}\right\Vert^2\right]\\
            &+10\alpha\sum_{t=0}^{T-1} \mathbb{E}\left[\left\Vert\mathbb{E}_t[\bar{\omega}^{(t+1)}]-\bar{q}^{(t)}\right\Vert^2\right]+\dfrac{1}{2}\cdot\alpha\sum_{t=0}^{T-1} \mathbb{E}\left[\left\Vert\bar{h}^{(t)}-\nabla\Phi(\bar{x}^{(t)})\right\Vert^2\right]\\
            &+10\alpha\sum_{t=0}^{T-1} \mathbb{E}\left[\left\Vert\bar{q}^{(t)}-\nabla\Phi(\bar{x}^{(t)})\right\Vert^2\right]+(1+c_3)\alpha^2\sum_{t=0}^{T-1}\mathbb{E}\left[\left\Vert\bar{\omega}^{(t+1)}-\mathbb{E}_t[\bar{\omega}^{(t+1)}]\right\Vert^2\right]\\
            \leq&\dfrac{2}{\tau}\mathbb{E}[E_0]+{\dfrac{1}{c_3}\mathbb{E}\left[\left\Vert\bar{h}^{(0)}-\nabla\Phi(\bar{x}^{(0)})\right\Vert^2\right]}+\dfrac{4L_{\nabla\Phi}^2\tau^2}{c_3^2}\alpha\sum_{t=0}^{T-1} \mathbb{E}\left[\left\Vert\bar{h}^{(t)}\right\Vert^2\right]\\
            &+\dfrac{1}{2}\cdot\alpha\sum_{t=0}^{T-1} \mathbb{E}\left[\left\Vert\bar{h}^{(t)}-\nabla\Phi(\bar{x}^{(t)})\right\Vert^2\right]+\left(60L^2+\dfrac{6(1+c_3)\alpha\sigma^2}{N}\right)\dfrac{1}{N}\cdot\alpha\sum_{t=0}^{T-1}\mathbb{E}[\Delta_t^2]\\&+\dfrac{6(1+c_3)T\alpha^2}{N}\left(1+\dfrac{L_f^2}{\mu_g^2}\right)\sigma^2+10T\alpha\iota^2            +60\left(\dfrac{L_f^2L_{\nabla^2g}^2}{\mu_g^2}+L_{\nabla f}^2\right)\sum_{t=0}^{T-1}\alpha\mathbb{E}\left[\left\Vert \bar{y}^{(t)}-\bar{y}^{(t)}_{\star}\right\Vert^2\right]\\
            &+\left(60L_{\nabla^2g}^2+\dfrac{6(1+c_3)\alpha\sigma^2}{N}\right)\sum_{t=0}^{T-1}\alpha\mathbb{E}\left[\left\Vert \bar{z}^{(t)}-\bar{z}^{(t)}_{\star}\right\Vert^2\right].
        \end{aligned}
    \end{equation}
    Then we complete the proof of this lemma.
\end{proof}
 
\subsection{Estimation bias of $y,z$}
In this section, we first estimate the convergence error induced by $y$ and $z$, and then estimate the error between $\nabla\Phi(\bar{x}^{(t+1)})$ and $\bar{h}^{(t+1)}$.
\begin{lemma}
\label{lemma:est yt+1-yt+1*}
Suppose Assumption \ref{assumption:smooth} and \ref{assumption:unbiased} hold, and the step-size $\beta$ satisfies:
\begin{align}
\label{eq:assumptiony}
    \beta(\mu_g+L_{\nabla g})\leq1.
\end{align}
Then in Algorithm \ref{alg:DeMA-SOBA}, we have:
\begin{equation}
    \begin{aligned}
    \label{eq:lemma:yt+1-yt+1*}
        &\sum_{t=0}^{T-1} \alpha\mathbb{E}\left[\left\Vert\bar{y}^{(t)}-\bar{y}^{(t)}_{\star}\right\Vert^2\right]\\
        \leq&C_{yx}\sum_{t=0}^{T-1} \tau^2\alpha\mathbb{E}\left[\left\Vert\bar{h}^{(t)}\right\Vert^2\right]+C_{yD}\sum_{t=0}^{T-1} \alpha\mathbb{E}\left[\dfrac{1}{N}\Delta_t^2\right]+C_{y1}\dfrac{T\alpha^2}{N}+C_{y2},
    \end{aligned}
\end{equation}
where the constants are defined as:
\begin{align*}
C_{yx}&=\dfrac{3L_{y^{\star}}^2}{c_1^2\mu_g^2}=\mathcal{O}\left(\dfrac{\kappa^4}{c_1^2}\right),\quad C_{yD}=\dfrac{6L_{\nabla g}^2}{\mu_g^2\kappa^2}=\mathcal{O}\left(1\right),\\
C_{y1}&=\frac{2c_1\sigma_{g1}^2}{\mu_g}=\mathcal{O}\left(c_1\kappa\right),\quad C_{y2}={\dfrac{1}{c_1\mu_g}\mathbb{E}\left[\left\Vert\bar{y}^{(0)}-\bar{y}^{(0)}_{\star}\right\Vert^2\right]}=\mathcal{O}\left(\frac{\kappa}{c_1}\right).
\end{align*}
\end{lemma}

\begin{proof}
    Firstly, from Jensen's Inequality and the Lipchitz continuity of $y^{\star}(x)$, we have:
    \begin{equation}
    \begin{aligned}
    \label{first:yt+1-yt+1*}
        \left\Vert\bar{y}^{(t+1)}-\bar{y}^{(t+1)}_{\star}\right\Vert^2&\leq\left(1+\dfrac{\beta\mu_g}{2}\right)\left\Vert\bar{y}^{(t+1)}-\bar{y}^{(t)}_{\star}\right\Vert^2+\left(1+\frac{2}{\beta\mu_g}\right)\left\Vert\bar{y}^{(t)}_{\star}-\bar{y}^{(t+1)}_{\star}\right\Vert^2\\
        &\leq\left(1+\dfrac{\beta\mu_g}{2}\right)\left\Vert\bar{y}^{(t+1)}-\bar{y}^{(t)}_{\star}\right\Vert^2+\left(\alpha^2+\frac{2\alpha^2}{\beta\mu_g}\right)L_{y^{\star}}^2\tau^2\left\Vert\bar{h}^{(t)}\right\Vert^2.
    \end{aligned}
    \end{equation}
    Then, as $g$ is $\mu_g$-strongly convex with respect to $y$, we have:
    \begin{equation}
    \label{final:yt+1-y*}
        \begin{aligned}
        &\mathbb{E}_t\left[\left\Vert\bar{y}^{(t+1)}-\bar{y}^{(t)}_{\star}\right\Vert^2\right]\\
        =&\left\Vert\bar{y}^{(t)}-\beta\cdot\dfrac{1}{N}\sum_{i=1}^N\nabla_2g_i({x}^{(t)}_i,{y}^{(t)}_i)-\bar{y}^{(t)}_{\star}\right\Vert^2+\beta^2\dfrac{1}{N^2}\sum_{i=1}^N\mathbb{E}_t\left[\left\Vert\nabla_2g_i({x}^{(t)}_i,{y}^{(t)}_i)-v^{(t)}_i\right\Vert^2\right]\\
        \leq&\left(1+\dfrac{\beta\mu_g}{2}\right)\left\Vert\bar{y}^{(t)}-\beta\nabla_2g(\bar{x}^{(t)},\bar{y}^{(t)})-\bar{y}^{(t)}_{\star}\right\Vert^2\\
        &+\left(\beta^2+\dfrac{2\beta}{\mu_g}\right)\dfrac{1}{N}\sum_{i=1}^N\left\Vert\nabla_2g_n({x}^{(t)}_i,{y}^{(t)}_i)-\nabla_2g_n(\bar{x}^{(t)},\bar{y}^{(t)})\right\Vert^2+\dfrac{\beta^2\sigma_{g1}^2}{N}\\
        \leq&\left(1+\dfrac{\beta\mu_g}{2}\right)\left(1-{\beta\mu_g}\right)^2\left\Vert\bar{y}^{(t)}-\bar{y}^{(t)}_{\star}\right\Vert^2+\dfrac{3\beta L_{\nabla g}^2}{N\mu_g\kappa^2}\Delta_t^2+\dfrac{\beta^2\sigma^2}{N},
        \end{aligned}
    \end{equation}
    where the first equation is due to Assumption \ref{assumption:unbiased}, the first inequality is from Jensen's inequality and Assumption \ref{assumption:unbiased}, and the last inequality is from $\beta\mu_g<1$ and Lemma \ref{des}.
    
    Taking the conditional expectation of (\ref{first:yt+1-yt+1*}) and using (\ref{final:yt+1-y*}), we have:
    \begin{equation}
    \label{final:yt+1-yt+1*}
    \begin{aligned}
        &\mathbb{E}_t\left[\left\Vert\bar{y}^{(t+1)}-\bar{y}^{(t+1)}_{\star}\right\Vert^2\right]\\
        \leq&\left(1+\dfrac{\beta\mu_g}{2}\right)^2\left(1-{\beta\mu_g}\right)^2\left\Vert\bar{y}^{(t)}-\bar{y}^{(t)}_{\star}\right\Vert^2+\left(1+\dfrac{\beta\mu_g}{2}\right)\left(\dfrac{3\beta L_{\nabla g}^2}{N\mu_g\kappa^2}\Delta_t^2+\dfrac{\beta^2\sigma^2}{N}\right)\\
        &+\left(\alpha^2+\frac{2\alpha^2}{\beta\mu_g}\right)L_{y^{\star}}^2\tau^2\left\Vert\bar{h}^{(t)}\right\Vert^2\\
        \leq&\left(1-{\beta\mu_g}\right)\left\Vert\bar{y}^{(t)}-\bar{y}^{(t)}_{\star}\right\Vert^2+\dfrac{6\beta L_{\nabla g}^2}{N\mu_g\kappa^2}\Delta_t^2+\dfrac{2\beta^2\sigma^2}{N}+\dfrac{3\tau^2\alpha^2L_{y^{\star}}^2}{\beta\mu_g}\left\Vert\bar{h}^{(t)}\right\Vert^2,
    \end{aligned}
    \end{equation}
    where the second inequality holds because $-\dfrac{3}{4}(\beta\mu_g)^2+\dfrac{1}{2}(\beta\mu_g)^3+\dfrac{1}{4}(\beta\mu_g)^4\leq0$ when $0\leq\beta\mu_g\leq1$.
    Taking summation and expectation on both sides, we have:
    \begin{equation}
    \label{sum:yt+1-yt+1*}
    \begin{aligned}
        &\beta\mu_g\sum_{t=0}^{T-1} \mathbb{E}\left[\left\Vert\bar{y}^{(t)}-\bar{y}^{(t)}_{\star}\right\Vert^2\right]\\
        \leq&\dfrac{3\tau^2\alpha^2L_{y^{\star}}^2}{\beta\mu_g}\sum_{t=0}^{T-1} \mathbb{E}\left[\left\Vert\bar{h}^{(t)}\right\Vert^2\right]+\dfrac{6\beta L_{\nabla g}^2}{\mu_g\kappa^2}\sum_{t=0}^{T-1} \mathbb{E}\left[\dfrac{1}{N}\Delta_t^2\right]+\dfrac{2T\beta^2\sigma^2}{N}+{\mathbb{E}\left[\left\Vert\bar{y}^{(0)}-\bar{y}^{(0)}_{\star}\right\Vert^2\right]}.
    \end{aligned}
    \end{equation}

    Finally, multiplying $\dfrac{1}{c_1\mu_g}$ on both sides of \eqref{sum:yt+1-yt+1*}, we have:
    \begin{equation}
    \begin{aligned}
    \label{final:yt+1-yt+1**}
        &\sum_{t=0}^{T-1} \alpha\mathbb{E}\left[\left\Vert\bar{y}^{(t)}-\bar{y}^{(t)}_{\star}\right\Vert^2\right]\\
        \leq&\dfrac{3\tau^2L_{y^{\star}}^2}{c_1^2\mu_g^2}\sum_{t=0}^{T-1} \alpha\mathbb{E}\left[\left\Vert\bar{h}^{(t)}\right\Vert^2\right]+\dfrac{6L_{\nabla g}^2}{\mu_g^2}\sum_{t=0}^{T-1} \alpha\mathbb{E}\left[\dfrac{1}{N}\Delta_t^2\right]+\frac{2c_1\sigma^2}{\mu_g}\dfrac{T\alpha^2}{N}+{\dfrac{1}{c_1\mu_g}\mathbb{E}\left[\left\Vert\bar{y}^{(0)}-\bar{y}^{(0)}_{\star}\right\Vert^2\right]}\\
        =&C_{yx}\sum_{t=0}^{T-1} \tau^2\alpha\mathbb{E}\left[\left\Vert\bar{h}^{(t)}\right\Vert^2\right]+C_{yD}\sum_{t=0}^{T-1} \alpha\mathbb{E}\left[\dfrac{1}{N}\Delta_t^2\right]+C_{y1}\dfrac{T\alpha^2}{N}+C_{y2},
    \end{aligned}
    \end{equation}
    where $C_{yx}=\dfrac{3L_{y^{\star}}^2}{c_1^2\mu_g^2}$, $C_{yD}=\dfrac{6L_{\nabla g}^2}{\mu_g^2\kappa^2}$, $C_{y1}=\dfrac{2c_1\sigma^2}{\mu_g}$, $C_{y2}={\dfrac{1}{c_1\mu_g}\mathbb{E}\left[\left\Vert\bar{y}^{(0)}-\bar{y}^{(0)}_{\star}\right\Vert^2\right]}$.
\end{proof}

\begin{lemma}
\label{lemma:est zt+1-zt+1*}
Suppose Assumption \ref{assumption:smooth} and \ref{assumption:unbiased} hold. If the step-size $\gamma$ satisfies:
\begin{align}
\label{eq:assumptionz}
    -\dfrac{1}{6}\gamma\mu_g+ {\dfrac{6}{N}}\gamma^2\sigma^2+\dfrac{1}{2}\gamma^3\mu_g^3\leq0,\quad \gamma\mu_g\leq\min\left\{1,\dfrac{NL^2}{\mu_g\sigma^2}\right\},
\end{align}
in Algorithm \ref{alg:DeMA-SOBA}, we have:
    \begin{equation}
    \label{eq:lemma:zt+1-zt+1*}
        \begin{aligned}
           &\sum_{t=0}^{T-1} \alpha\mathbb{E}\left[\left\Vert\bar{z}^{(t+1)}-\bar{z}^{(t+1)}_{\star}\right\Vert^2\right]\\
           \leq& C_{zx}\sum_{t=0}^{T-1} \tau^2\alpha\mathbb{E}\left[\left\Vert\bar{h}^{(t)}\right\Vert^2\right]+C_{zD}\sum_{t=0}^{T-1} \alpha\mathbb{E}\left[\dfrac{1}{N}\Delta_t^2\right]+C_{z1}\dfrac{T\alpha^2}{N}+C_{z2}+C_{z3}T\alpha\iota^2,
        \end{aligned}
    \end{equation}
where the constants are defined as:

\begin{align*}
    C_{zx}&=\dfrac{32}{\mu_g^2}\left(\dfrac{L_{\nabla^2g}^2L_f^2}{\mu_g^2}+L_{\nabla f}^2\right)C_{yx}+\dfrac{8L_{z^{\star}}^2}{c_2^2\mu_g^2}=\mathcal{O}\left(\dfrac{\kappa^8}{c_1^2}+\dfrac{\kappa^8}{c_2^2}\right),\\
     C_{zD}&=\dfrac{32}{\mu_g^2}\left[\left(\dfrac{L_{\nabla^2g}^2L_f^2}{\mu_g^2}+L_{\nabla f}^2\right)C_{yD}+2L^2\right]=\mathcal{O}(\kappa^4),\\
    C_{z1}&=\dfrac{32}{\mu_g^2}\left(\dfrac{L_{\nabla^2g}^2L_f^2}{\mu_g^2}+L_{\nabla f}^2\right)C_{y1}+\dfrac{8c_2}{\mu_g}\left(1+\dfrac{L_f^2}{\mu_g^2}{}\right)\sigma^2=\mathcal{O}(c_1\kappa^5+c_2\kappa^3),\\
    C_{z2}&=\dfrac{32}{\mu_g^2}\left(\dfrac{L_{\nabla^2g}^2L_f^2}{\mu_g^2}+L_{\nabla f}^2\right)C_{y2}+{\dfrac{1}{c_2\mu_g}\mathbb{E}\left[\left\Vert \bar{z}^{(0)}-\bar{z}^{(0)}_{\star}\right\Vert^2\right]}=\mathcal{O}\left(\dfrac{\kappa^5}{c_1}+\dfrac{\kappa}{c_2}\right),\\
    C_{z3}&=\dfrac{16}{\mu_g^2}=\mathcal{O}(\kappa^2).
\end{align*}
\end{lemma}

\begin{proof}
    For the term $\bar{z}^{(t+1)}-\bar{z}^{(t+1)}_{\star}$, we have:
    \begin{equation}
    \begin{aligned}
    \label{z_chubu}
        \left\Vert \bar{z}^{(t+1)}-\bar{z}^{(t+1)}_{\star}\right\Vert^2\leq&\left(1+\dfrac{\gamma\mu_g}{3}\right)\left\Vert \bar{z}^{(t+1)}-\bar{z}^{(t)}_{\star}\right\Vert^2+\left(1+\dfrac{3}{\gamma\mu_g}\right)\left\Vert \bar{z}^{(t)}_{\star}-\bar{z}^{(t+1)}_{\star}\right\Vert^2\\
        \leq&\left(1+\dfrac{\gamma\mu_g}{3}\right)\left\Vert \bar{z}^{(t+1)}-\bar{z}^{(t)}_{\star}\right\Vert^2+\left(\alpha^2+\dfrac{3\alpha^2}{\gamma\mu_g}\right)\tau^2L_{z^{\star}}^2\left\Vert \bar{h}^{(t)}\right\Vert^2,
    \end{aligned}
    \end{equation}
    where the first inequality is due to Jensen's Inequality and the second is due to the Lipchitz continuity of $z^{\star}(x)$. Then, from the definition of $\bar{z}^{(t)}_{\star}$, we have:
    \begin{equation}
        \begin{aligned}
            &\bar{z}^{(t+1)}-\bar{z}^{(t)}_{\star}\\
           =&\dfrac{1}{N}\sum_{i=1}^N(I-\gamma\nabla_{22}g_i(x^{(t)}_{i},y^{(t)}_{i}))(\bar{z}^{(t)}-\bar{z}^{(t)}_{\star})-\dfrac{\gamma}{N}\sum_{i=1}^N\nabla_{22}g_i(x^{(t)}_{i},y^{(t)}_{i})(z^{(t)}_{i}-\bar{z}^{(t)})\\
            &-\dfrac{\gamma}{N}\sum_{i=1}^N\left[\left(\nabla_{22}g_i(x^{(t)}_{i},y^{(t)}_{i})-\nabla_{22}g(\bar{x}^{(t)},\bar{y}^{(t)}_{\star})\right)\bar{z}^{(t)}_{\star}-\left(\nabla_{1}f_i(x^{(t)}_{i},y^{(t)}_{i})-\nabla_{1}f(\bar{x}^{(t)},\bar{y}^{(t)}_{\star})\right)\right]\\
            &-\dfrac{\gamma}{N}\sum_{i=1}^N\left[(p_{H,i}-\mathbb{E}_t[p_{H,i}])-(u^{(t+1)}_{1,i}-\nabla_1f_i(x^{(t)}_{i},y^{(t)}_{i}))-(\mathbb{E}_t[p_{H,i}]-\nabla_{22}^2g_i(x^{(t)}_{i},y^{(t)}_{i})z^{(t)}_{i})\right],
        \end{aligned}
    \end{equation}
    Taking the conditional expectation with respect to $\mathcal{F}^{(t)}$ on both sides, with Jensen's Inequality, and Assumption \ref{assumption:unbiased}, we have:
    \begin{equation}
        \label{eq:z^{t+1}-i-{z}^{(t)}-i^*-start}
        \begin{aligned}
            &\mathbb{E}_t\left[\left\Vert \bar{z}^{(t+1)}-\bar{z}^{(t)}_{\star}\right\Vert^2\right]\\
            \leq&\left(1+\dfrac{\gamma\mu_g}{2}\right)\dfrac{1}{N}\sum_{i=1}^N\left\Vert (I-\gamma\nabla_{22}g_i(x^{(t)}_{i},y^{(t)}_{i}))(\bar{z}^{(t)}-\bar{z}^{(t)}_{\star})\right\Vert^2\\
            &+4\gamma^2L_{\nabla g}^2\left(1+\dfrac{2}{\gamma\mu_g}\right)\cdot\dfrac{1}{N}\left\Vert \mathbf{z}^{(t)}-\mathbf{\bar{z}}^{(t)}\right\Vert_F^2\\
            &+4\gamma^2\left(1+\dfrac{2}{\gamma\mu_g}\right)\left\Vert\dfrac{1}{N}\sum_{i=1}^N\left(\nabla_{22}g_i(x^{(t)}_{i},y^{(t)}_{i})-\nabla_{22}g(\bar{x}^{(t)},\bar{y}^{(t)}_{\star})\right)\right\Vert^2\left\Vert\bar{z}^{(t)}_{\star}\right\Vert^2\\
            &+4\gamma^2\left(1+\dfrac{2}{\gamma\mu_g}\right)\left\Vert\dfrac{1}{N}\sum_{i=1}^N\left(\nabla_{1}f_i(x^{(t)}_{i},y^{(t)}_{i})-\nabla_{1}f(\bar{x}^{(t)},\bar{y}^{(t)}_{\star})\right)\right\Vert^2\\
            &+4\gamma^2\left(1+\dfrac{2}{\gamma\mu_g}\right)\left\Vert\dfrac{1}{N}\sum_{i=1}^N\left(\mathbb{E}_t[p_{H,i}]-\nabla_{22}^2g_i(x^{(t)}_{i},y^{(t)}_{i})z^{(t)}_{i}\right)\right\Vert^2\\
            &+\dfrac{2\gamma^2}{N^2}\sum_{i=1}^N\mathbb{E}_t\left[\left\Vert p_{H,i}-\mathbb{E}_t[p_{H,i}]\right\Vert^2\right]+\dfrac{2\gamma^2}{N^2}\sum_{i=1}^N\mathbb{E}_t\left[\left\Vert u^{(t+1)}_{1,i}-\nabla_1f_i(x^{(t)}_{i},y^{(t)}_{i})\right\Vert^2\right].
        \end{aligned}
    \end{equation}
    
    From the Lipchitz continuity of $\nabla_{22}g_i$, we have:
    \begin{equation}
        \label{eq:z^{t+1}-i-{z}^{(t)}-i^*-part31}
        \begin{aligned}
            &\left\Vert\dfrac{1}{N}\sum_{i=1}^N\left(\nabla_{22}g_i(x^{(t)}_{i},y^{(t)}_{i})-\nabla_{22}g_i(\bar{x}^{(t)},\bar{y}^{(t)}_{\star})\right)\right\Vert^2\left\Vert\bar{z}^{(t)}_{\star}\right\Vert^2\\
            \leq&\dfrac{2L_f^2}{\mu_g^2}\left\Vert\dfrac{1}{N}\sum_{i=1}^N\left(\nabla_{22}g_i(x^{(t)}_{i},y^{(t)}_{i})-\nabla_{22}g_i(\bar{x}^{(t)},\bar{y}^{(t)})\right)\right\Vert^2\\
            &+\dfrac{2L_f^2}{\mu_g^2}\left\Vert\dfrac{1}{N}\sum_{i=1}^N\left(\nabla_{22}g_i(\bar{x}^{(t)},\bar{y}^{(t)})-\nabla_{22}g_i(\bar{x}^{(t)},\bar{y}^{(t)}_{\star})\right)\right\Vert^2\\
            \leq&\dfrac{2L_f^2}{N\mu_g^2}\sum_{i=1}^N\left(\left\Vert\nabla_{22}g_i(x^{(t)}_{i},y^{(t)}_{i})-\nabla_{22}g_i(\bar{x}^{(t)},\bar{y}^{(t)})\right\Vert^2+\left\Vert\nabla_{22}g_i(x^{(t)}_{i},y^{(t)}_{i})-\nabla_{22}g_i(\bar{x}^{(t)},\bar{y}^{(t)})\right\Vert^2\right)\\
            \leq&\dfrac{2L_{\nabla^2g}^2L_f^2}{N\mu_g^2}\left(\left\Vert\mathbf{x}^{(t)}-\mathbf{\bar{x}}^{(t)}\right\Vert_F^2+\left\Vert\mathbf{y}^{(t)}-\mathbf{\bar{y}}^{(t)}\right\Vert_F^2+N\left\Vert\bar{y}^{(t)}-\bar{y}^{(t)}_{\star}\right\Vert^2\right).
        \end{aligned}
    \end{equation}
    
    Similarly, from the Lipchitz continuity of $\nabla_1f_i$, we have:
    \begin{equation}
        \label{eq:z^{t+1}-i-{z}^{(t)}-i^*-part3}
        \begin{aligned}
            &\left\Vert\dfrac{1}{N}\sum_{i=1}^N\left(\nabla_{1}f_i(x^{(t)}_{i},y^{(t)}_{i})-\nabla_{1}f_i(\bar{x}^{(t)},\bar{y}^{(t)}_{\star})\right)\right\Vert^2\\
            \leq&2\left\Vert\dfrac{1}{N}\sum_{i=1}^N\left(\nabla_{1}f_i(x^{(t)}_{i},y^{(t)}_{i})-\nabla_{1}f_i(\bar{x}^{(t)},\bar{y}^{(t)})\right)\right\Vert^2\\
            &+2\left\Vert\dfrac{1}{N}\sum_{i=1}^N\left(\nabla_{1}f_i(\bar{x}^{(t)},\bar{y}^{(t)})-\nabla_{1}f_i(\bar{x}^{(t)},\bar{y}^{(t)}_{\star})\right)\right\Vert^2\\
            \leq&\dfrac{2L_{\nabla f}^2}{N}\left(\left\Vert\mathbf{x}^{(t)}-\mathbf{\bar{x}}^{(t)}\right\Vert_F^2+\left\Vert\mathbf{y}^{(t)}-\mathbf{\bar{y}}^{(t)}\right\Vert_F^2+N\left\Vert\bar{y}^{(t)}-\bar{y}^{(t)}_{\star}\right\Vert^2\right).
        \end{aligned}
    \end{equation}

    We also note that $g_i$ is $\mu_g$-strongly convex with respect to $y$ and $\gamma\mu_g\leq1$, then:
    \begin{equation}
    \label{eq:z^{t+1}-i-{z}^{(t)}-i^*-part5}
        \left\Vert (I-\gamma\nabla_{22}g_i(x^{(t)}_{i},y^{(t)}_{i}))(\bar{z}^{(t)}-\bar{z}^{(t)}_{\star})\right\Vert^2\leq(1-\gamma\mu_g)^2\left\Vert\bar{z}^{(t)}-\bar{z}^{(t)}_{\star}\right\Vert^2
    \end{equation}

    Substituting \eqref{chafen_H}, \eqref{eq:z^{t+1}-i-{z}^{(t)}-i^*-part31}, \eqref{eq:z^{t+1}-i-{z}^{(t)}-i^*-part3}, and \eqref{eq:z^{t+1}-i-{z}^{(t)}-i^*-part5} into \eqref{eq:z^{t+1}-i-{z}^{(t)}-i^*-start}, and then using  Assumption \ref{assumption:unbiased}, we get:
    \begin{equation}
        \label{eq:z^{t+1}-i-{z}^{(t)}-i^*-middle_cache}
        \begin{aligned}
            &\mathbb{E}_t\left[\left\Vert \bar{z}^{(t+1)}-\bar{z}^{(t)}_{\star}\right\Vert^2\right]\\
            \leq&\left(1+\dfrac{\gamma\mu_g}{2}\right)(1-\gamma\mu_g)^2\left\Vert\bar{z}^{(t)}-\bar{z}^{(t)}_{\star}\right\Vert^2+\dfrac{12\gamma L_{\nabla g}^2}{\mu_g}\cdot\dfrac{1}{N}\left\Vert \mathbf{z}^{(t)}-\mathbf{\bar{z}}^{(t)}\right\Vert_F^2\\
            &+\dfrac{24\gamma}{N\mu_g}\left(\dfrac{L_{\nabla^2g}^2L_f^2}{\mu_g^2}+L_{\nabla f}^2\right)\left(\left\Vert\mathbf{x}^{(t)}-\mathbf{\bar{x}}^{(t)}\right\Vert_F^2+\left\Vert\mathbf{y}^{(t)}-\mathbf{\bar{y}}^{(t)}\right\Vert_F^2+N\left\Vert\bar{y}^{(t)}-\bar{y}^{(t)}_{\star}\right\Vert^2\right)\\
            &+\dfrac{12\gamma}{\mu_g}\iota^2+\dfrac{2\gamma^2}{N}\sigma^2+\dfrac{6\gamma^2}{N}\left(\dfrac{1}{N}\sum_{i=1}^N\left\Vert z_i^{(t)}-\bar{z}^{(t)}\right\Vert^2+\left\Vert \bar{z}^{(t)}-\bar{z}^{(t)}_{\star}\right\Vert^2+\left\Vert \bar{z}^{(t)}_{\star}\right\Vert^2\right)\sigma^2\\
            \leq&\left(\left(1+\dfrac{\gamma\mu_g}{2}\right)(1-\gamma\mu_g)^2+\dfrac{6\gamma^2\sigma^2}{N}\right)\left\Vert\bar{z}^{(t)}-\bar{z}^{(t)}_{\star}\right\Vert^2+\dfrac{24\gamma}{\mu_g}\left(\dfrac{L_{\nabla^2g}^2L_f^2}{\mu_g^2}+L_{\nabla f}^2\right)\left\Vert\bar{y}^{(t)}-\bar{y}^{(t)}_{\star}\right\Vert^2\\
            &+\dfrac{24\gamma}{\mu_g}\left(\dfrac{L_{\nabla^2g}^2L_f^2}{\mu_g^2}+L_{\nabla f}^2\right)\cdot\dfrac{1}{N}\left(\left\Vert\mathbf{x}^{(t)}-\mathbf{\bar{x}}^{(t)}\right\Vert_F^2+\left\Vert\mathbf{y}^{(t)}-\mathbf{\bar{y}}^{(t)}\right\Vert_F^2\right)\\
            &+\left(\dfrac{12\gamma L_{\nabla g}^2}{\mu_g}+\dfrac{6\gamma^2\sigma^2}{N}\right)\cdot\dfrac{1}{N}\left\Vert \mathbf{z}^{(t)}-\mathbf{\bar{z}}^{(t)}\right\Vert_F^2+\dfrac{6\gamma^2}{N}\left(1+\dfrac{L_f^2}{\mu_g^2}{}\right)\sigma^2+\dfrac{12\gamma}{\mu_g}\iota^2\\
            \leq&\left(1-\dfrac{4\gamma\mu_g}{3}\right)\left\Vert\bar{z}^{(t)}-\bar{z}^{(t)}_{\star}\right\Vert^2+\dfrac{24\gamma}{\mu_g}\left(\dfrac{L_{\nabla^2g}^2L_f^2}{\mu_g^2}+L_{\nabla f}^2\right)\left\Vert\bar{y}^{(t)}-\bar{y}^{(t)}_{\star}\right\Vert^2\\
            &+\dfrac{48\gamma L^2}{\mu_g}\cdot\dfrac{1}{N}\Delta_t^2+\dfrac{6\gamma^2}{N}\left(1+\dfrac{L_f^2}{\mu_g^2}{}\right)\sigma^2+\dfrac{12\gamma}{\mu_g}\iota^2,
        \end{aligned}
    \end{equation}
    where the last equation uses the assumption $-\dfrac{1}{6}\gamma\mu_g+ {\dfrac{6}{N}}\gamma^2\sigma^2+\dfrac{1}{2}\gamma^3\mu_g^3\leq0$, $\gamma\mu_g\leq\dfrac{NL^2}{\mu_g\sigma_{g2}^2}$, and the fact that $\dfrac{L_{\nabla^2g}^2L_f^2}{\mu_g^2}+L_{\nabla f}^2\leq(1+\kappa^2)L^2\leq2\kappa^2L^2$.

    Plugging \eqref{eq:z^{t+1}-i-{z}^{(t)}-i^*-middle_cache} into \eqref{z_chubu} and noting that $1+\dfrac{\gamma\mu_g}{3}\leq\dfrac{4}{3}$, we have:
    \begin{equation}
        \begin{aligned}
            &\mathbb{E}_t\left[\left\Vert\bar{z}^{(t+1)}-\bar{z}^{(t+1)}_{\star}\right\Vert^2\right]\\
            \leq&\left(1+\dfrac{\gamma\mu_g}{3}\right)\left(1-\dfrac{4\gamma\mu_g}{3}\right)\left\Vert\bar{z}^{(t)}-\bar{z}^{(t)}_{\star}\right\Vert^2+\left(\alpha^2+\dfrac{3\alpha^2}{\gamma\mu_g}\right)\tau^2L_{z^{\star}}^2\left\Vert \bar{h}^{(t)}\right\Vert^2+\dfrac{64\gamma}{\mu_g}L^2\dfrac{1}{N}\Delta^2_t\\
            &+\dfrac{32\gamma}{\mu_g}\left(\dfrac{L_{\nabla^2g}^2L_f^2}{\mu_g^2}+L_{\nabla f}^2\right)\left\Vert\bar{y}^{(t)}-\bar{y}^{(t)}_{\star}\right\Vert^2+\dfrac{8\gamma^2}{N}\left(1+\dfrac{L_f^2}{\mu_g^2}{}\right)\sigma^2+\dfrac{16\gamma}{\mu_g}\iota^2\\
            \leq&(1-\gamma\mu_g)\left\Vert\bar{z}^{(t)}-\bar{z}^{(t)}_{\star}\right\Vert^2+\dfrac{8L_{z^{\star}}^2\tau^2\alpha^2}{\gamma\mu_g}\left\Vert \bar{h}^{(t)}\right\Vert^2+\dfrac{64\gamma}{\mu_g}L^2\dfrac{1}{N}\Delta^2_t\\
            &+\dfrac{32\gamma}{\mu_g}\left(\dfrac{L_{\nabla^2g}^2L_f^2}{\mu_g^2}+L_{\nabla f}^2\right)\left\Vert\bar{y}^{(t)}-\bar{y}^{(t)}_{\star}\right\Vert^2+\dfrac{8\gamma^2}{N}\left(1+\dfrac{L_f^2}{\mu_g^2}{}\right)\sigma^2+\dfrac{16\gamma}{\mu_g}\iota^2.
        \end{aligned}
    \end{equation}
    Taking summation and expectation on both sides, we obtain:
    \begin{equation}
        \begin{aligned}
            &\gamma\mu_g\sum_{t=0}^{T-1} \mathbb{E}\left[\left\Vert\bar{z}^{(t+1)}-\bar{z}^{(t+1)}_{\star}\right\Vert^2\right]\\
            \leq&{\mathbb{E}\left[\left\Vert \bar{z}^{(0)}-\bar{z}^{(0)}_{\star}\right\Vert^2\right]}+\dfrac{8L_{z^{\star}}^2\tau^2\alpha^2}{\gamma\mu_g}\sum_{t=0}^{T-1} \mathbb{E}\left[\left\Vert\bar{h}^{(t)}\right\Vert^2\right]+\dfrac{64\gamma}{\mu_g}L^2\sum_{t=0}^{T-1} \mathbb{E}\left[\dfrac{1}{N}\Delta_t^2\right]\\
            &+\dfrac{32\gamma}{\mu_g}\left(\dfrac{L_{\nabla^2g}^2L_f^2}{\mu_g^2}+L_{\nabla f}^2\right)\sum_{t=0}^{T-1} \mathbb{E}\left[\left\Vert\bar{y}^{(t)}-\bar{y}^{(t)}_{\star}\right\Vert^2\right]+\dfrac{8T\gamma^2}{N}\left(1+\dfrac{L_f^2}{\mu_g^2}{}\right)\sigma^2+\dfrac{16T\gamma}{\mu_g}\iota^2.
        \end{aligned}
    \end{equation}
    Finally, from \eqref{eq:lemma:yt+1-yt+1*}, we have:
    \begin{equation}
    \label{final:re_zt+1-zt+1*}
        \begin{aligned}
           &\sum_{t=0}^{T-1} \alpha\mathbb{E}\left[\left\Vert\bar{z}^{(t+1)}-\bar{z}^{(t+1)}_{\star}\right\Vert^2\right]\\
            \leq&{\dfrac{1}{c_2\mu_g}\mathbb{E}\left[\left\Vert \bar{z}^{(0)}-\bar{z}^{(0)}_{\star}\right\Vert^2\right]}+\dfrac{8L_{z^{\star}}^2\tau^2}{c_2^2\mu_g^2}\sum_{t=0}^{T-1} \alpha\mathbb{E}\left[\left\Vert\bar{h}^{(t)}\right\Vert^2\right]+\dfrac{64}{\mu_g^2}L^2\sum_{t=0}^{T-1} \alpha\mathbb{E}\left[\dfrac{1}{N}\Delta_t^2\right]\\
            &+\dfrac{32}{\mu_g^2}\left(\dfrac{L_{\nabla^2g}^2L_f^2}{\mu_g^2}+L_{\nabla f}^2\right)\sum_{t=0}^{T-1} \alpha\mathbb{E}\left[\left\Vert\bar{y}^{(t)}-\bar{y}^{(t)}_{\star}\right\Vert^2\right]+\dfrac{T\alpha^2}{N}\dfrac{8c_2}{\mu_g}\left(1+\dfrac{L_f^2}{\mu_g^2}{}\right)\sigma^2+T\alpha\iota^2\dfrac{16}{\mu_g^2}\\
           \leq&C_{zx}\sum_{t=0}^{T-1} \tau^2\alpha\mathbb{E}\left[\left\Vert\bar{h}^{(t)}\right\Vert^2\right]+C_{zD}\sum_{t=0}^{T-1} \alpha\mathbb{E}\left[\dfrac{1}{N}\Delta_t^2\right]+C_{z1}\dfrac{T\alpha^2}{N}+C_{z2}+C_{z3}T\alpha\iota^2,
        \end{aligned}
    \end{equation}
    where
\begin{align*}
    C_{zx}&=\dfrac{32}{\mu_g^2}\left(\dfrac{L_{\nabla^2g}^2L_f^2}{\mu_g^2}+L_{\nabla f}^2\right)C_{yx}+\dfrac{8L_{z^{\star}}^2}{c_2^2\mu_g^2},\\
    C_{zD}&=\dfrac{32}{\mu_g^2}\left[\left(\dfrac{L_{\nabla^2g}^2L_f^2}{\mu_g^2}+L_{\nabla f}^2\right)C_{yD}+2L^2\right],\\
    C_{z1}&=\dfrac{32}{\mu_g^2}\left(\dfrac{L_{\nabla^2g}^2L_f^2}{\mu_g^2}+L_{\nabla f}^2\right)C_{y1}+\dfrac{8c_2}{\mu_g}\left(1+\dfrac{L_f^2}{\mu_g^2}{}\right)\sigma^2,\\
    C_{z2}&=\dfrac{32}{\mu_g^2}\left(\dfrac{L_{\nabla^2g}^2L_f^2}{\mu_g^2}+L_{\nabla f}^2\right)C_{y2}+{\dfrac{1}{c_2\mu_g}\mathbb{E}\left[\left\Vert \bar{z}^{(0)}-\bar{z}^{(0)}_{\star}\right\Vert^2\right]},\\
    C_{z3}&=\dfrac{16}{\mu_g^2}.
\end{align*}
\end{proof}

Now we can obtain the estimation errors for $y,z$ in the following lemma:
\begin{lemma}
\label{lem:consensus1}
Suppose Assumption \ref{assumption:smooth}, \ref{assumption: data heterogeneity}, \ref{assumption:unbiased}, and \ref{assumption: gossip communication} hold. If the step-sizes $\beta, \gamma$ satisfy \eqref{eq:assumptiony}, \eqref{eq:assumptionz},  in Algorithm \ref{alg:DeMA-SOBA}  we have:
\begin{equation}
\label{eq:final11_xy_consesus1}
    \begin{aligned}
        &(1+\kappa^2)\sum_{t=0}^{T-1} \alpha\mathbb{E}\left[\left\Vert\bar{y}^{(t)}-\bar{y}^{(t)}_{\star}\right\Vert^2\right]+\sum_{t=0}^{T-1} \alpha\mathbb{E}\left[\left\Vert\bar{z}^{(t)}-\bar{z}^{(t)}_{\star}\right\Vert^2\right]\\
        \leq&((1+\kappa^2)C_{yx}+C_{zx})\sum_{t=0}^{T-1} \tau^2\alpha\mathbb{E}\left[\left\Vert\bar{h}^{(t)}\right\Vert^2\right]+((1+\kappa^2)C_{y1}+C_{z1})\dfrac{T\alpha^2}{N}+((1+\kappa^2)C_{y2}+C_{z2})\\
        &+C_{z3}T\alpha^2\iota^2+((1+\kappa^2)C_{yD}+C_{zD})\sum_{t=0}^{T-1} \alpha\mathbb{E}\left[\dfrac{1}{N}\Delta_t^2\right].
    \end{aligned}
\end{equation}

\end{lemma}

\subsection{Consensus Lemmas}
Then we analyze the consensus error  $\Delta_t^2$. In the beginning, we  present some notations. 
The iteration of $x$ in Algorithm \ref{alg:DeMA-SOBA} can be written as follows:
\begin{equation}
\label{matrix form}
\mathbf{x}^{(t+1)}=W\left(\mathbf{x}^{(t)}-\tau\alpha\mathbf{h}^{(t)}\right).
\end{equation}
By left-multiplying $\dfrac{1}{N}\mathbf{1}\mathbf{1}^\top $ on both sides of \eqref{matrix form} and using Assumption \ref{assumption: gossip communication}, we have:
\begin{equation}
\label{matrix form average}
\bar{\mathbf{x}}^{(t+1)}=\dfrac{1}{N}\mathbf{1}\mathbf{1}^\top \left(\mathbf{x}^{(t)}-\tau\alpha\mathbf{h}^{(t)}\right).
\end{equation}
Then, from \eqref{matrix form} and \eqref{matrix form average}, we have
\begin{equation}
\label{matrix11}
\begin{aligned}
\mathbf{x}^{(t+1)}-\bar{\mathbf{x}}^{(t+1)}=&\left(W-\dfrac{1}{N}\mathbf{1}\mathbf{1}^\top \right)\left(\mathbf{x}^{(t)}-\tau\alpha\mathbf{h}^{(t)}\right)\\
=&\left(W-\dfrac{1}{N}\mathbf{1}\mathbf{1}^\top \right)\left(\mathbf{x}^{(t)}-\bar{\mathbf{x}}^{(t)}-\tau\alpha\mathbf{h}^{(t)}\right).
\end{aligned}
\end{equation}

Similarly, the consensus term in the iteration of $y$ can be written as:
\begin{equation}
\label{matrix12}
\begin{aligned}
\mathbf{y}^{(t+1)}-\bar{\mathbf{y}}^{(t+1)}=\left(W-\dfrac{1}{N}\mathbf{1}\mathbf{1}^\top \right)\left(\mathbf{y}^{(t)}-\bar{\mathbf{y}}^{(t)}-\beta\mathbf{v}^{(t+1)}\right),
\end{aligned}
\end{equation}
and the consensus term in the iteration of $z$ can be written as:
\begin{equation}
\label{matrix13}
\begin{aligned}
\mathbf{z}^{(t+1)}-\bar{\mathbf{z}}^{(t+1)}=\left(W-\dfrac{1}{N}\mathbf{1}\mathbf{1}^\top \right)\left(\mathbf{z}^{(t)}-\bar{\mathbf{z}}^{(t)}-\gamma\mathbf{r}^{(t+1)}\right).
\end{aligned}
\end{equation}

The following lemma \ref{lem:consensusx} gives the upper bound of the consensus error of $x$:
\begin{lemma}
\label{lem:consensusx}
Suppose Assumption \ref{assumption:smooth}, \ref{assumption: data heterogeneity}, \ref{assumption:unbiased}, and \ref{assumption: gossip communication} hold. Then in Algorithm \ref{alg:DeMA-SOBA}, we have:
\begin{equation}
\label{eq:final-x-consesus2}
    \begin{aligned}
        &\mathbb{E}\left[\sum_{t=0}^{T-1} \left\Vert \mathbf{x}^{(t)}-\mathbf{\bar{x}}^{(t)}\right\Vert_F^2\right]\\
        \leq&\left(\dfrac{208\tau^2\alpha^2\rho^2L^2}{(1-\rho)^2}+\dfrac{6\tau^2\alpha^2\rho^2\sigma^2}{1-\rho}\right)\mathbb{E}\left[\sum_{t=0}^{T-1}\Delta_t^2\right]+{\dfrac{192\tau^2\alpha^2\rho^2}{(1-\rho)^2}NT\left(b_1^2+\dfrac{L_f^2}{\mu_g^2}b_2^2\right)}\\
        &+\left(\dfrac{192N\tau^2\alpha^2\rho^2}{(1-\rho)^2}L_{\nabla g}^2+\dfrac{6N\tau^2\alpha^2\rho^2}{1-\rho}\sigma^2\right)\sum_{t=0}^{T-1}\mathbb{E}\left[\left\Vert\bar{z}^{(t)}-\bar{z}^{(t)}_{\star}\right\Vert^2\right]+\dfrac{12\tau^2\alpha^2\rho^2}{(1-\rho)^2}NT\iota^2\\
        &+\dfrac{6\tau^2\alpha^2\rho^2}{1-\rho}NT\left(1+\dfrac{L_f^2}{\mu_g^2}\right)\sigma^2+{\dfrac{1}{1-\rho}\left\Vert \mathbf{x}^{(0)}-\mathbf{\bar{x}}^{(0)}\right\Vert_F^2}.
    \end{aligned}
\end{equation}
\end{lemma}

\begin{proof}
    From the definition of $\mathcal{F}^{(t)}$, we know that for all $i=1,2,\cdots,N$, the term $x^{(t)}_{i}$ is measureable with respect to $\mathcal{F}^{(t-1)}$. We have:
    \begin{equation}
    \begin{aligned}
        &\mathbb{E}_{t-1}\left[\left\Vert \left(\mathbf{x}^{(t+1)}-\bar{\mathbf{x}}^{(t+1)}\right)-(1-\theta)\left(\mathbf{x}^{(t)}-\bar{\mathbf{x}}^{(t)}\right) \right\Vert^2_F\right]\\
        =&\mathbb{E}_{t-1}\left[\left\Vert \left(W-\dfrac{1}{N}\mathbf{1}\mathbf{1}^\top \right)\left(\mathbf{x}^{(t)}-\bar{\mathbf{x}}^{(t)}-\tau\alpha\mathbf{h}^{(t)}\right)-(1-\theta)\left(\mathbf{x}^{(t)}-\bar{\mathbf{x}}^{(t)}\right) \right\Vert^2_F\right]\\
        =&\left\Vert \left(W-\dfrac{1}{N}\mathbf{1}\mathbf{1}^\top \right)\left(\mathbf{x}^{(t)}-\bar{\mathbf{x}}^{(t)}-\tau\alpha(1-\theta)\mathbf{h}^{(t-1)}-\tau\alpha\theta\mathbb{E}_{t-1}[{\omega}^{(t)}]\right)-(1-\theta)(\mathbf{x}^{(t)}-\bar{\mathbf{x}}^{(t)})\right\Vert^2_F\\
        &+\tau^2\alpha^2\theta^2\rho^2\sum_{i=1}^N\mathbb{E}_{t-1}\left[\left\Vert\omega^{(t)}_{i}-\mathbb{E}_{t-1}[\omega^{(t)}_{i}]\right\Vert^2\right].
    \end{aligned}
    \end{equation}
where the last equation is due to Assumption \ref{assumption:unbiased}. Note that
\begin{align}
    -\tau\alpha\left(W-\dfrac{1}{N}\mathbf{1}\mathbf{1}^\top \right)\mathbf{h}^{(t-1)}=\left(\mathbf{x}^{(t)}-\bar{\mathbf{x}}^{(t)}\right)-\left(W-\dfrac{1}{N}\mathbf{1}\mathbf{1}^\top \right)\left(\mathbf{x}^{(t-1)}-\bar{\mathbf{x}}^{(t-1)}\right).
\end{align}
It follows that:
    \begin{equation}
    \label{consis x 1}
    \begin{aligned}
        &\mathbb{E}_{t-1}\left[\left\Vert \left(\mathbf{x}^{(t+1)}-\bar{\mathbf{x}}^{(t+1)}\right)-(1-\theta)\left(\mathbf{x}^{(t)}-\bar{\mathbf{x}}^{(t)}\right) \right\Vert^2_F\right]\\
        \leq&\dfrac{1}{\rho}\left\Vert \left(W-\dfrac{1}{N}\mathbf{1}\mathbf{1}^\top \right)\left[\mathbf{x}^{(t)}-\bar{\mathbf{x}}^{(t)}-(1-\theta)\left(\mathbf{x}^{(t-1)}-\bar{\mathbf{x}}^{(t-1)}\right)\right]\right\Vert^2_F\\
        &+\dfrac{\tau^2\alpha^2\theta^2}{1-\rho}\left\Vert \left(W-\dfrac{1}{N}\mathbf{1}\mathbf{1}^\top \right)\left(\mathbb{E}_{t-1}[{\omega}^{(t)}]-\overline{\mathbb{E}_{t-1}[{\omega}^{(t)}]}\right)\right\Vert^2_F\\
        &+\tau^2\alpha^2\theta^2\rho^2\sum_{i=1}^N\mathbb{E}_{t-1}\left[\left\Vert\omega^{(t)}_{i}-\mathbb{E}_{t-1}[\omega^{(t)}_{i}]\right\Vert^2_F\right]\\
        \leq&\rho \left\Vert \left(\mathbf{x}^{(t)}-\bar{\mathbf{x}}^{(t)}\right)-(1-\theta)\left(\mathbf{x}^{(t-1)}-\bar{\mathbf{x}}^{(t-1)}\right) \right\Vert^2_F+\dfrac{\tau^2\alpha^2\theta^2\rho^2}{1-\rho}\sum_{i=1}^N\left\Vert\mathbb{E}_{t-1}[\omega_i^{(t)}]-\mathbb{E}_{t-1}[\bar{\omega}^{(t)}]\right\Vert^2\\
        &+\tau^2\alpha^2\theta^2\rho^2\sum_{i=1}^N\mathbb{E}_{t-1}\left[\left\Vert\omega^{(t)}_{i}-\mathbb{E}_{t-1}[\omega^{(t)}_{i}]\right\Vert^2\right],
    \end{aligned}
    \end{equation}
where the first inequality is due to Jensen's inequality and Assumption \ref{assumption: gossip communication}. Then taking expectation and summation on both sides of \eqref{consis x 1} from 0 to $T-1$ and using the fact that $x_i^{(0)}=x_i^{(1)}=0$, we have:
    \begin{equation}
    \label{consis x 10}
    \begin{aligned}
        &\mathbb{E}\left[\sum_{t=0}^{T-1} \left\Vert \left(\mathbf{x}^{(t)}-\bar{\mathbf{x}}^{(t)}\right)-(1-\theta)\left(\mathbf{x}^{(t-1)}-\bar{\mathbf{x}}^{(t-1)}\right) \right\Vert^2_F\right]\\
        \leq&\dfrac{\tau^2\alpha^2\theta^2\rho^2}{(1-\rho)^2}\sum_{t=0}^{T-1} \sum_{i=1}^N\mathbb{E}\left[\left\Vert\mathbb{E}_{t-1}[\omega_i^{(t)}]-\mathbb{E}_{t-1}[\bar{\omega}^{(t)}]\right\Vert^2\right]\\
        &+\dfrac{\tau^2\alpha^2\theta^2\rho^2}{1-\rho}\sum_{t=0}^{T-1} \sum_{i=1}^N\mathbb{E}\left[\left\Vert\omega^{(t)}_{i}-\mathbb{E}_{t-1}[\omega^{(t)}_{i}]\right\Vert^2_F\right]+{\dfrac{1}{1-\rho}\mathbb{E}\left[\left\Vert \mathbf{x}^{(0)}-\bar{\mathbf{x}}^{(0)} \right\Vert^2_F\right]}\\
        \leq&\dfrac{2\tau^2\alpha^2\theta^2\rho^2}{(1-\rho)^2}\sum_{t=0}^{T-1} \sum_{i=1}^N\mathbb{E}\left[\left\Vert\nabla_1f_i(x_i^{(t)},y_i^{(t)})-\dfrac{1}{N}\sum_{m=1}^N\nabla_1f_m(x_m^{(t)},y_m^{(t)})\right\Vert^2\right]\\
        &+\dfrac{2\tau^2\alpha^2\theta^2\rho^2}{(1-\rho)^2}\sum_{t=0}^{T-1} \sum_{i=1}^N\mathbb{E}\left[\left\Vert\mathbb{E}_{t-1}[p_{J,i}^{(t)}]-\mathbb{E}_{t-1}[\bar{p}_{J,i}^{(t)}]\right\Vert^2\right]\\
        &+\dfrac{\tau^2\alpha^2\theta^2\rho^2}{1-\rho}\sum_{t=0}^{T-1} \sum_{i=1}^N\mathbb{E}\left[\left\Vert\omega^{(t)}_{i}-\mathbb{E}_{t-1}[\omega^{(t)}_{i}]\right\Vert^2_F\right]+{\dfrac{1}{1-\rho}\mathbb{E}\left[\left\Vert \mathbf{x}^{(0)}-\bar{\mathbf{x}}^{(0)} \right\Vert^2_F\right]},
    \end{aligned}
    \end{equation}
    where we denote $\mathbf{x}^{(-1)}=\bar{\mathbf{x}}^{(-1)}={\mathbf{0}}_{N\times d}$

    Note that:
    \begin{equation}
    \begin{aligned}
        &\sum_{i=1}^N\mathbb{E}\left[\left\Vert\mathbb{E}_{t-1}[p_{J,i}^{(t)}]-\mathbb{E}_{t-1}[\bar{p}_{J,i}^{(t)}]\right\Vert^2\right]\\
        \leq&3\sum_{i=1}^N\mathbb{E}\left[\left\Vert\mathbb{E}_{t-1}[p_{J,i}^{(t)}]-\nabla_{12}g_i(x^{(t)}_{i},y^{(t)}_{i})z^{(t)}_{i}\right\Vert^2\right]\\
        &+3\sum_{i=1}^N\mathbb{E}\left[\left\Vert\nabla_{12}g_i(x^{(t)}_{i},y^{(t)}_{i})z^{(t)}_{i}-\dfrac{1}{N}\sum_{m=1}^N\nabla_{12}g_m(x^{(t)}_m,y^{(t)}_m)z^{(t)}_m\right\Vert^2\right]\\
        &+3\sum_{i=1}^N\mathbb{E}\left[\left\Vert\dfrac{1}{N}\sum_{m=1}^N\mathbb{E}_{t-1}[p_{J,m}^{(t)}]-\dfrac{1}{N}\sum_{m=1}^N\nabla_{12}g_m(x^{(t)}_m,y^{(t)}_m)z^{(t)}_m\right\Vert^2\right]\\
        \leq&6\sum_{i=1}^N\mathbb{E}\left[\left\Vert\mathbb{E}_{t-1}[p_{J,i}^{(t)}]-\nabla_{12}g_i(x^{(t)}_{i},y^{(t)}_{i})z^{(t)}_{i}\right\Vert^2\right]\\
        &+6\sum_{i=1}^N\mathbb{E}\left[\left\Vert\nabla_{12}g_i(x^{(t)}_{i},y^{(t)}_{i})z^{(t)}_{i}-\dfrac{1}{N}\sum_{m=1}^N\nabla_{12}g_m(x^{(t)}_m,y^{(t)}_m)z^{(t)}_m\right\Vert^2\right].
    \end{aligned}
    \end{equation}

    Then we can obtain:
    \begin{equation}
    \label{consis x 11}
    \begin{aligned}
        &\mathbb{E}\left[\sum_{t=1}^{T-1} \left\Vert \left(\mathbf{x}^{(t)}-\bar{\mathbf{x}}^{(t)}\right)-(1-\theta)\left(\mathbf{x}^{(t-1)}-\bar{\mathbf{x}}^{(t-1)}\right) \right\Vert^2_F\right]\\
        \leq&\dfrac{2\tau^2\alpha^2\theta^2\rho^2}{(1-\rho)^2}\sum_{t=1}^{T-1} \sum_{i=1}^N\mathbb{E}\left[\left\Vert\nabla_1f_i(x_i^{(t)},y_i^{(t)})-\dfrac{1}{N}\sum_{m=1}^N\nabla_1f_m(x_m^{(t)},y_m^{(t)})\right\Vert^2\right]\\
        &+\dfrac{12\tau^2\alpha^2\theta^2\rho^2}{(1-\rho)^2}\sum_{t=1}^{T-1} \sum_{i=1}^N\mathbb{E}\left[\left\Vert\mathbb{E}_{t-1}[p_{J,i}^{(t)}]-\nabla_{12}g_i(x^{(t)}_{i},y^{(t)}_{i})z^{(t)}_{i}\right\Vert^2\right]\\
        &+\dfrac{12\tau^2\alpha^2\theta^2\rho^2}{(1-\rho)^2}\sum_{t=1}^{T-1} \sum_{i=1}^N\mathbb{E}\left[\left\Vert \nabla_{12}g_i(x^{(t)}_{i},y^{(t)}_{i})z^{(t)}_{i}-\dfrac{1}{N}\sum_{m=1}^N\nabla_{12}g_m(x^{(t)}_m,y^{(t)}_m)z^{(t)}_m  \right\Vert^2\right]\\
        &+\dfrac{2\tau^2\alpha^2\theta^2\rho^2}{1-\rho}\sum_{t=1}^{T-1} \sum_{i=1}^N\left(\mathbb{E}\left[\left\Vert u_{i,x}^{(t)}-\mathbb{E}_t[u_{i,x}^{(t)}]\right\Vert^2\right]+\mathbb{E}\left[\left\Vert p^{(t)}_{J,i}-\mathbb{E}_{t-1}[\bar{p}^{(t)}_{J,i}]\right\Vert^2\right]\right).
    \end{aligned}
    \end{equation}
    
For the first term on the right-hand side of \eqref{consis x 11}, we have:
\begin{equation}
\label{omega-omega_bar}
    \begin{aligned}
        &\sum_{i=1}^N\left\Vert\nabla_1f_i(x_i^{(t)},y_i^{(t)})-\dfrac{1}{N}\sum_{m=1}^N\nabla_1f_m(x_m^{(t)},y_m^{(t)})\right\Vert^2\\
        =&\sum_{i=1}^N\left\Vert\left[\nabla_1f_i(x^{(t)}_{i},y^{(t)}_{i})-\nabla_1f(\bar{x}^{(t)},\bar{y}^{(t)})\right]-\dfrac{1}{N}\sum_{m=1}^N\left[\nabla_1f_m(x^{(t)}_m,y^{(t)}_m)-\nabla_1f(\bar{x}^{(t)},\bar{y}^{(t)})\right]\right\Vert^2\\
        \leq&4\sum_{i=1}^N\left\Vert\nabla_1f_i(x^{(t)}_{i},y^{(t)}_{i})-\nabla_1f(\bar{x}^{(t)},\bar{y}^{(t)})\right\Vert^2\\
        \leq&8\sum_{i=1}^N\left\Vert\nabla_1f_i(x^{(t)}_{i},y^{(t)}_{i})-\nabla_1f_i(\bar{x}^{(t)},\bar{y}^{(t)})\right\Vert^2+8\sum_{i=1}^N\left\Vert\nabla_1f_i(\bar{x}^{(t)},\bar{y}^{(t)})-\nabla_1f(\bar{x}^{(t)},\bar{y}^{(t)})\right\Vert^2\\
        \leq&8L_{\nabla f}^2\Delta_t^2+8N{b_1^2}.
    \end{aligned}
\end{equation}

For the third term of \eqref{consis x 11}, we have:
\begin{equation}
\label{error conx g}
    \begin{aligned}
        &\sum_{i=1}^N\left\Vert \nabla_{12}g_i(x^{(t)}_{i},y^{(t)}_{i})z^{(t)}_{i}-\dfrac{1}{N}\sum_{m=1}^N\nabla_{12}g_m(x^{(t)}_m,y^{(t)}_m)z^{(t)}_m  \right\Vert^2\\
        \leq&4\sum_{i=1}^N\left\Vert \nabla_{12}g_i(x^{(t)}_{i},y^{(t)}_{i})z^{(t)}_{i}-\nabla_{12}g(\bar{x}^{(t)},\bar{y}^{(t)})\bar{z}_{\star}^{(t)}\right\Vert^2\\
        \leq&16\sum_{i=1}^N\left\Vert\nabla_{12}g_i(x^{(t)}_{i},y^{(t)}_{i})\left(z^{(t)}_{i}-\bar{z}^{(t)}\right)\right\Vert^2+16\sum_{i=1}^N\left\Vert\nabla_{12}g_i(x^{(t)}_{i},y^{(t)}_{i})\left(\bar{z}^{(t)}-\bar{z}^{(t)}_{\star}\right)\right\Vert^2\\
        &+16\sum_{i=1}^N\left\Vert\left(\nabla_{12}g_i(x^{(t)}_{i},y^{(t)}_{i})-\nabla_{12}g_i(\bar{x}^{(t)},\bar{y}^{(t)})\right)\bar{z}^{(t)}_{\star}\right\Vert^2\\
        &+16\sum_{i=1}^N\left\Vert\left(\nabla_{12}g_i(\bar{x}^{(t)},\bar{y}^{(t)})-\nabla_{12}g(\bar{x}^{(t)},\bar{y}^{(t)})\right)\bar{z}^{(t)}_{\star}\right\Vert^2\\
        \leq&16L_{\nabla g}^2\left\Vert\mathbf{z}^{(t)}-\bar{\mathbf{z}}^{(t)}\right\Vert_F^2+16L_{\nabla g}^2\sum_{i=1}^N\left\Vert\bar{z}^{(t)}-\bar{z}^{(t)}_{\star}\right\Vert^2\\
        &+16\kappa^2L_{\nabla g}^2\left(\left\Vert\mathbf{x}^{(t)}-\bar{\mathbf{x}}^{(t)}\right\Vert^2_F+\left\Vert\mathbf{y}^{(t)}-\bar{\mathbf{y}}^{(t)}\right\Vert^2_F\right)+\dfrac{16L_f^2}{\mu_g^2}N{b_2^2}\\
        \leq&16L_{\nabla g}^2\sum_{i=1}^N\left\Vert\bar{z}^{(t)}-\bar{z}^{(t)}_{\star}\right\Vert^2+16L_{\nabla g}^2\Delta_t^2+\dfrac{16L_f^2}{\mu_g^2}N{b_2^2}.
    \end{aligned}
\end{equation}

Finally, we consider the last term of \eqref{consis x 11}. From Assumption \ref{assumption:unbiased} and Eq. \eqref{chafen_J}, we obtain:
\begin{equation}
    \begin{aligned}
    \label{error conx vector product}
        &\sum_{i=1}^N\left(\mathbb{E}\left[\left\Vert u_{i,x}^{(t)}-\mathbb{E}_t[u_{i,x}^{(t)}]\right\Vert^2_F\right]+\mathbb{E}\left[\left\Vert p^{(t)}_{J,i}-\mathbb{E}_{t-1}[\bar{p}^{(t)}_{J,i}]\right\Vert^2_F\right]\right)\\
        \leq&\sum_{i=1}^N\left(1+\left\Vert z_i^{(t)}\right\Vert^2\right)\sigma^2\\
        \leq&\sum_{i=1}^N\left(1+3\left\Vert z_i^{(t)}-\bar{z}^{(t)}\right\Vert^2+3\left\Vert \bar{z}^{(t)}-\bar{z}^{(t)}_{\star}\right\Vert^2+3\left\Vert \bar{z}^{(t)}_{\star}\right\Vert^2\right)\sigma^2\\
        \leq&3N\left(1+\dfrac{L_f^2}{\mu_g^2}\right)\sigma^2+3N\sigma^2\left\Vert \bar{z}^{(t)}-\bar{z}^{(t)}_{\star}\right\Vert^2+3\sigma^2\Delta_t^2.
    \end{aligned}
\end{equation}

Substituting \eqref{omega-omega_bar}, \eqref{error conx g}, and \eqref{error conx vector product} into \eqref{consis x 11}, then using Assumption \ref{assumption:unbiased} and Eq. \eqref{chafen_J}, we have:
\begin{equation}
\label{consis x 2}
\begin{aligned}
    &\mathbb{E}\left[\sum_{t=0}^{T-1} \left\Vert \left(\mathbf{x}^{(t)}-\bar{\mathbf{x}}^{(t)}\right)-(1-\theta)\left(\mathbf{x}^{(t-1)}-\bar{\mathbf{x}}^{(t-1)}\right) \right\Vert^2_F\right]\\
    \leq&\left(\dfrac{208\tau^2\alpha^2\theta^2\rho^2L^2}{(1-\rho)^2}+\dfrac{6\tau^2\alpha^2\theta^2\rho^2\sigma^2}{1-\rho}\right)\mathbb{E}\left[\sum_{t=0}^{T-1}\Delta_t^2\right]+{\dfrac{192\tau^2\alpha^2\theta^2\rho^2}{(1-\rho)^2}NT\left(b_1^2+\dfrac{L_f^2}{\mu_g^2}b_2^2\right)}\\
    &+\left(\dfrac{192N\tau^2\alpha^2\theta^2\rho^2}{(1-\rho)^2}L_{\nabla g}^2+\dfrac{6N\tau^2\alpha^2\theta^2\rho^2}{1-\rho}\sigma^2\right)\sum_{t=0}^{T-1}\mathbb{E}\left[\left\Vert\bar{z}^{(t)}-\bar{z}^{(t)}_{\star}\right\Vert^2\right]+\dfrac{12\tau^2\alpha^2\theta^2\rho^2}{(1-\rho)^2}NT\iota^2\\
    &+\dfrac{6\tau^2\alpha^2\theta^2\rho^2}{1-\rho}NT\left(1+\dfrac{L_f^2}{\mu_g^2}\right)\sigma^2+{\dfrac{1}{1-\rho}\mathbb{E}\left[\left\Vert \mathbf{x}^{(0)}-\bar{\mathbf{x}}^{(0)} \right\Vert^2_F\right]}.
\end{aligned}
\end{equation}
    
Next, for $t=1,2,\cdots,T-1$, we have the following inequality thanks to Jensen's inequality:
\begin{equation}
\label{drop 1-theta}
    \begin{aligned}
    \big\Vert\mathbf{x}^{(t)}-\mathbf{\bar{x}}^{(t)}\big\Vert_F^2\leq\dfrac{1}{\theta}\big\Vert (\mathbf{x}^{(t)}-\bar{\mathbf{x}}^{(t)})-(1-\theta)(\mathbf{x}^{(t-1)}-\bar{\mathbf{x}}^{(t-1)}) \big\Vert^2_F+(1-\theta)\big\Vert \mathbf{x}^{(t-1)}-\bar{\mathbf{x}}^{(t-1)}\big\Vert^2_F.
    \end{aligned}
    \end{equation}
Taking summation over $i=1,2,\cdots,N$ and $t=1,2,\cdots,T-1$ on both sides of \eqref{drop 1-theta}, then using the fact that $x_i^{(0)}=x_i^{(1)}=0$ and \eqref{consis x 2}, we have:
    \begin{equation}
    \label{final_con_x}
    \begin{aligned}
        &\mathbb{E}\left[\sum_{t=0}^{T-1} \left\Vert \mathbf{x}^{(t)}-\mathbf{\bar{x}}^{(t)}\right\Vert_F^2\right]\\
        \leq&\left(\dfrac{208\tau^2\alpha^2\rho^2L^2}{(1-\rho)^2}+\dfrac{6\tau^2\alpha^2\rho^2\sigma^2}{1-\rho}\right)\mathbb{E}\left[\sum_{t=0}^{T-1}\Delta_t^2\right]+{\dfrac{192\tau^2\alpha^2\rho^2}{(1-\rho)^2}NT\left(b_1^2+\dfrac{L_f^2}{\mu_g^2}b_2^2\right)}\\
        &+\left(\dfrac{192N\tau^2\alpha^2\rho^2}{(1-\rho)^2}L_{\nabla g}^2+\dfrac{6N\tau^2\alpha^2\rho^2}{1-\rho}\sigma^2\right)\sum_{t=0}^{T-1}\mathbb{E}\left[\left\Vert\bar{z}^{(t)}-\bar{z}^{(t)}_{\star}\right\Vert^2\right]+\dfrac{12\tau^2\alpha^2\rho^2}{(1-\rho)^2}NT\iota^2\\
        &+\dfrac{6\tau^2\alpha^2\rho^2}{1-\rho}NT\left(1+\dfrac{L_f^2}{\mu_g^2}\right)\sigma^2+{\dfrac{1}{1-\rho}\left\Vert \mathbf{x}^{(0)}-\mathbf{\bar{x}}^{(0)}\right\Vert_F^2}.
    \end{aligned}
    \end{equation}
\end{proof}

Then, the following lemma \ref{lem:consensusy} gives the upper bound of the consensus error of $y$. 
\begin{lemma}
\label{lem:consensusy}
Suppose Assumption \ref{assumption:smooth}, \ref{assumption: data heterogeneity}, \ref{assumption:unbiased}, and \ref{assumption: gossip communication} hold. Then in Algorithm \ref{alg:DeMA-SOBA}, we have:
\begin{equation}
\label{eq:final-y-consesus2}
    \begin{aligned}
            &\mathbb{E}\left[\sum_{t=0}^{T-1} \left\Vert \mathbf{y}^{(t)}-\mathbf{\bar{y}}^{(t)}\right\Vert_F^2\right]\leq\dfrac{3\rho^2\beta^2}{\kappa^2(1-\rho)^2}L^2_{\nabla g}\mathbb{E}\left[\sum_{t=0}^{T-1} \Delta_t^2\right]+{\dfrac{3N\rho^2\beta^2}{(1-\rho)^2}b_2^2}\\
            +&\dfrac{3N\rho^2\beta^2}{(1-\rho)^2}L_{\nabla g}^2\mathbb{E}\left[\sum_{t=0}^{T-1} \left\Vert\bar{y}^{(t)}-\bar{y}^{(t)}_{\star}\right\Vert^2\right]+\dfrac{NT\rho^2\beta^2}{1-\rho}\sigma^2+{\dfrac{1}{1-\rho}\left\Vert \mathbf{y}^{(0)}-\mathbf{\bar{y}}^{(0)}\right\Vert_F^2}.
    \end{aligned}
\end{equation}
\end{lemma}

\begin{proof}
We consider the conditional expectation of $\left\Vert \mathbf{y}^{(t+1)}-\mathbf{\bar{y}}^{(t+1)}\right\Vert_F^2$ with respect to $\mathcal{F}^{(t)}$:
\begin{equation}
\begin{aligned}
    \label{y-con-tiaojian}
    &\mathbb{E}_t\left[\left\Vert\mathbf{y}^{(t+1)}-\bar{\mathbf{y}}^{(t+1)}\right\Vert_F^2\right]\\
    =&\left\Vert\left(W-\dfrac{1}{N}\mathbf{1}\mathbf{1}^\top \right)\left[\left(\mathbf{y}^{(t)}-\bar{\mathbf{y}}^{(t)}\right)-\beta\nabla_2\mathbf{g}^{(t+1)}\right]\right\Vert_F^2+\rho^2\beta^2\mathbb{E}_t\left[\sum_{i=1}^N\left\Vert v^{(t+1)}_{2,i}-\nabla_2g_i(x^{(t)}_{i},y^{(t)}_{i})\right\Vert^2\right]\\
    \leq&\dfrac{1}{\rho}\left\Vert\left(W-\dfrac{1}{N}\mathbf{1}\mathbf{1}^\top \right)\left(\mathbf{y}^{(t)}-\bar{\mathbf{y}}^{(t)}\right)\right\Vert_F^2+\dfrac{\beta^2}{1-\rho}\left\Vert\left(W-\dfrac{1}{N}\mathbf{1}\mathbf{1}^\top \right)\nabla_2\mathbf{g}^{(t+1)}\right\Vert_F^2+N\rho^2\beta^2\sigma^2\\
    \leq&\rho\left\Vert \mathbf{y}^{(t)}-\mathbf{\bar{y}}^{(t)}\right\Vert_F^2+\dfrac{\rho^2\beta^2}{1-\rho}\sum_{i=1}^N\left\Vert \nabla_2g_i(x^{(t)}_{i},y^{(t)}_{i})\right\Vert^2+N\rho^2\beta^2\sigma^2,
\end{aligned}
\end{equation}
where the first inequality is from Assumption \ref{assumption:unbiased}. For the second term on the right-hand side of \eqref{y-con-tiaojian}, we have:
\begin{equation}
\begin{aligned}
    &\sum_{i=1}^N\left\Vert \nabla_2g_i(x^{(t)}_{i},y^{(t)}_{i})\right\Vert^2\\
    \leq&3\sum_{i=1}^N\left\Vert\nabla_2g_i(x^{(t)}_{i},y^{(t)}_{i})-\nabla_2g_i(\bar{x}^{(t)},\bar{y}^{(t)})\right\Vert^2+3N\left\Vert\nabla_2g(\bar{x}^{(t)},\bar{y}^{(t)})-\nabla_2g(\bar{x}^{(t)},\bar{y}^{(t)}_{\star})\right\Vert^2\\
    &+3\sum_{i=1}^N\left\Vert\nabla_2g_i(\bar{x}^{(t)},\bar{y}^{(t)})-\dfrac{1}{N}\sum_{m=1}^N\nabla_2g_m(\bar{x}^{(t)},\bar{y}^{(t)})\right\Vert^2\\
    \leq&3L_{\nabla g}^2\sum_{i=1}^N\left[\left\Vert x_i^{(t)}-\bar{x}^{(t)}\right\Vert^2+\left\Vert y_i^{(t)}-\bar{y}^{(t)}\right\Vert^2\right]+3N{b_2^2}+3NL_{\nabla g}^2\left\Vert\bar{y}^{(t)}-\bar{y}^{(t)}_{\star}\right\Vert^2\\
    \leq&\frac{3L_{\nabla g}^2}{\kappa^2}\Delta_t^2+3N{b_2^2}+3NL_{\nabla g}^2\left\Vert\bar{y}^{(t)}-\bar{y}^{(t)}_{\star}\right\Vert^2,
\end{aligned}
\end{equation}
where the first inequality is due to the fact that $\nabla_2g(\bar{x}^{(t)},\bar{y}^{(t)})=\frac{1}{N}\sum_{i=1}^N\nabla_2g_i(\bar{x}^{(t)},\bar{y}^{(t)})$ as well as $\nabla_2g(\bar{x}^{(t)},\bar{y}^{(t)}_{\star})=0$.

Finally, taking summation and expectation on both sides of \eqref{y-con-tiaojian}, we have:
\begin{equation}
\begin{aligned}
    &\mathbb{E}\left[\sum_{t=0}^{T-1} \left\Vert\mathbf{y}^{(t+1)}-\bar{\mathbf{y}}^{(t+1)}\right\Vert_F^2\right]\\
    \leq&\dfrac{\rho^2\beta^2}{(1-\rho)^2}\mathbb{E}\left[\sum_{t=0}^{T-1} \sum_{i=1}^N\left\Vert \nabla_2g_i(x^{(t)}_{i},y^{(t)}_{i})\right\Vert^2\right]+\dfrac{N\rho^2\beta^2\sigma_{g1}^2}{1-\rho}+{\dfrac{1}{1-\rho}\left\Vert \mathbf{y}^{(0)}-\mathbf{\bar{y}}^{(0)}\right\Vert_F^2}\\
    \label{final-con-y}
    \leq&\dfrac{3\rho^2\beta^2}{\kappa^2(1-\rho)^2}L^2_{\nabla g}\mathbb{E}\left[\sum_{t=0}^{T-1} \Delta_t^2\right]+{\dfrac{3N\rho^2\beta^2}{(1-\rho)^2}b_2^2}+\dfrac{3NT\rho^2\beta^2}{(1-\rho)^2}L_{\nabla g}^2\mathbb{E}\left[\sum_{t=0}^{T-1} \left\Vert\bar{y}^{(t)}-\bar{y}^{(t)}_{\star}\right\Vert^2\right]\\
    &+\dfrac{NT\rho^2\beta^2}{1-\rho}\sigma^2+{\dfrac{1}{1-\rho}\left\Vert \mathbf{y}^{(0)}-\mathbf{\bar{y}}^{(0)}\right\Vert_F^2}.
\end{aligned}
\end{equation}
\end{proof}

Moreover, the following lemma \ref{lem:consensusz} gives the upper bound of the consensus error of $z$. 
\begin{lemma}
\label{lem:consensusz}
Suppose Assumption \ref{assumption:smooth}, \ref{assumption: data heterogeneity}, \ref{assumption:unbiased}, and \ref{assumption: gossip communication} hold. Then in Algorithm \ref{alg:DeMA-SOBA}, we have:
\begin{equation}
\label{eq:final-z-consesus2}
    \begin{aligned}
&\mathbb{E}\left[\sum_{t=0}^{T-1} \left\Vert \mathbf{z}^{(t)}-\mathbf{\bar{z}}^{(t)}\right\Vert_F^2\right]\\
\leq&\left(\dfrac{64\rho^2\gamma^2}{(1-\rho)^2}L^2+\dfrac{6\rho^2\gamma^2}{1-\rho}\sigma^2\right)\mathbb{E}\left[\sum_{t=0}^{T-1}\Delta_t^2\right]+{\dfrac{66\rho^2\gamma^2}{(1-\rho)^2}NT\left(b_1^2+\dfrac{L_f^2}{\mu_g^2}b_2^2\right)}\\
&+\dfrac{8N\rho^2\gamma^2}{(1-\rho)^2}\left(L_{\nabla f}^2+\dfrac{L_{\nabla^2g}^2L_f^2}{\mu_g^2}\right)\mathbb{E}\left[\sum_{t=0}^{T-1}\left\Vert\bar{y}^{(t)}-\bar{y}_{\star}^{(t)}\right\Vert^2\right]+{\dfrac{1}{1-\rho}\left\Vert \mathbf{z}^{(0)}-\mathbf{\bar{z}}^{(0)}\right\Vert_F^2}\\
&+\left(\dfrac{32N\rho^2\gamma^2}{(1-\rho)^2}L^2+\dfrac{6N\rho^2\gamma^2}{1-\rho}\sigma^2\right)\mathbb{E}\left[\sum_{t=0}^{T-1}\left\Vert \bar{z}^{(t)}-\bar{z}^{(t)}_{\star}\right\Vert^2\right]+\dfrac{6NT\rho^2\gamma^2}{1-\rho}\left(1+\dfrac{L_f^2}{\mu_g^2}\right)\sigma^2.
    \end{aligned}
\end{equation}
\end{lemma}

\begin{proof}
Considering the conditional expectation of $\left\Vert \mathbf{z}^{(t+1)}-\mathbf{\bar{z}}^{(t+1)}\right\Vert_F^2$ with respect to $\mathcal{F}^{(t)}$:
\begin{equation}
\label{consensus z start111}
\begin{aligned}
&\mathbb{E}_t\left[\left\Vert\mathbf{z}^{(t+1)}-\bar{\mathbf{z}}^{(t+1)}\right\Vert_F^2\right]\\
\leq&\dfrac{1}{\rho}\left\Vert\left(W-\dfrac{1}{N}\mathbf{1}\mathbf{1}^\top\right)\left(\mathbf{z}^{(t)}-\bar{\mathbf{z}}^{(t)}\right)\right\Vert_F^2\\
&+\dfrac{\rho^2\gamma^2}{1-\rho}\sum_{i=1}^N\left\Vert\mathbb{E}_t[p_{H,i}^{(t)}-u_{2,i}^{(t)}]-\left(\nabla_{22}^2g_i(x_i^{(t)},y_i^{(t)})z_i^{(t)}-\nabla_2f_i(x_i^{(t)},y_i^{(t)})\right)\right\Vert^2\\
&+2\rho^2\gamma^2\sum_{i=1}^N\left\Vert p_{H,i}^{(t)}-\mathbb{E}_t[p_{H,i}^{(t)}]\right\Vert^2+2\rho^2\gamma^2\sum_{i=1}^N\left\Vert u_{2,i}^{(t)}-\mathbb{E}_t[u_{2,i}^{(t)}]\right\Vert^2\\
\leq&\rho\left\Vert\mathbf{z}^{(t)}-\bar{\mathbf{z}}^{(t)}\right\Vert_F^2+\dfrac{2\rho^2\gamma^2}{1-\rho}\sum_{i=1}^N\left\Vert\mathbb{E}_t[p_{H,i}^{(t)}]-\nabla_{22}^2g_i(x_i^{(t)},y_i^{(t)})z_i^{(t)}\right\Vert^2\\
&+2\rho^2\gamma^2\sum_{i=1}^N\mathbb{E}_t\left[\left\Vert p_{H,i}^{(t)}-\mathbb{E}_t[p_{H,i}^{(t)}]\right\Vert^2\right]+2\rho^2\gamma^2\sum_{i=1}^N\mathbb{E}_t\left[\left\Vert u_{2,i}^{(t)}-\mathbb{E}_t[u_{2,i}^{(t)}]\right\Vert^2\right],
\end{aligned}
\end{equation}
where the second inequality is due to $\mathbb{E}_t[u_{2,i}^{(t)}]=\nabla_2f_i(x_i^{(t)},y_i^{(t)})$.

Then using $\nabla_{22}^2g(\bar{x}^{(t)},\bar{y}^{(t)}_{\star})\bar{z}^{(t)}_{\star}-\nabla_2f(\bar{x}^{(t)},\bar{y}^{(t)}_{\star})=0$, we have:
\begin{equation}
    \begin{aligned}
        \label{consensus z start0}
        &\sum_{i=1}^N\left\Vert\nabla_{22}g_i(x^{(t)}_{i},y^{(t)}_{i})z^{(t)}_{i}-\nabla_{2}f_i(x^{(t)}_{i},y^{(t)}_{i})\right\Vert^2\\
        \leq&4\sum_{i=1}^N\left\Vert\nabla_{22}g_i(x^{(t)}_{i},y^{(t)}_{i})z^{(t)}_{i}-\nabla_{22}g(\bar{x}^{(t)},\bar{y}^{(t)})\bar{z}^{(t)}_{\star}\right\Vert^2\\
        &+4\sum_{i=1}^N\left\Vert\nabla_{2}f_i(x^{(t)}_{i},y^{(t)}_{i})-\nabla_{2}f(\bar{x}^{(t)},\bar{y}^{(t)})\right\Vert^2\\
        &+4N\left\Vert\left(\nabla_{22}g(\bar{x}^{(t)},\bar{y}^{(t)})-\nabla_{22}g(\bar{x}^{(t)},\bar{y}_{\star}^{(t)})\right)\bar{z}^{(t)}_{\star}\right\Vert^2+4N\left\Vert\nabla_{2}f(\bar{x}^{(t)},\bar{y}^{(t)})-\nabla_{2}f(\bar{x}^{(t)},\bar{y}_{\star}^{(t)})\right\Vert^2\\
        \leq&4\sum_{i=1}^N\left\Vert\nabla_{22}g_i(x^{(t)}_{i},y^{(t)}_{i})z^{(t)}_{i}-\nabla_{22}g(\bar{x}^{(t)},\bar{y}^{(t)})\bar{z}^{(t)}_{\star}\right\Vert^2\\
        &+4\sum_{i=1}^N\left\Vert\nabla_{2}f_i(x^{(t)}_{i},y^{(t)}_{i})-\nabla_{2}f(\bar{x}^{(t)},\bar{y}^{(t)})\right\Vert^2+4N\left(L_{\nabla f}^2+\dfrac{L_{\nabla^2g}^2L_f^2}{\mu_g^2}\right)\left\Vert\bar{y}^{(t)}-\bar{y}_{\star}^{(t)}\right\Vert^2.
    \end{aligned}
\end{equation}

For the first term of the right-hand side of \eqref{consensus z start0}, we have:
\begin{equation}
    \begin{aligned}
        &\nabla_{22}g_i(x^{(t)}_{i},y^{(t)}_{i})z^{(t)}_{i}-\nabla_{22}g(\bar{x}^{(t)},\bar{y}^{(t)})\bar{z}^{(t)}_{\star}\\
        =&\nabla_{22}g_i(x^{(t)}_{i},y^{(t)}_{i})(z^{(t)}_{i}-\bar{z}^{(t)})+\nabla_{22}g_i(x^{(t)}_{i},y^{(t)}_{i})(\bar{z}^{(t)}-\bar{z}^{(t)}_{\star})\\
        &+(\nabla_{22}g_i(x^{(t)}_{i},y^{(t)}_{i})-\nabla_{22}g_i(\bar{x}^{(t)},\bar{y}^{(t)}))\bar{z}^{(t)}_{\star}+(\nabla_{22}g_i(\bar{x}^{(t)},\bar{y}^{(t)})-\nabla_{22}g(\bar{x}^{(t)},\bar{y}^{(t)}))\bar{z}^{(t)}_{\star}.
    \end{aligned}
\end{equation}

Taking the norm on both sides and using Assumption \ref{assumption: data heterogeneity}, we have:
\begin{equation}
    \begin{aligned}
    \label{consensus z guji1}
        &\sum_{i=1}^N\left\Vert\nabla_{22}g_i(x^{(t)}_{i},y^{(t)}_{i})z^{(t)}_{i}-\nabla_{22}g(\bar{x}^{(t)},\bar{y}^{(t)})\bar{z}^{(t)}_{\star}\right\Vert^2\\
        \leq&4L_{\nabla g}^2\sum_{i=1}^N\left\Vert z^{(t)}_{i}-\bar{z}^{(t)}\right\Vert^2+4NL_{\nabla g}^2\left\Vert \bar{z}^{(t)}-\bar{z}^{(t)}_{\star}\right\Vert^2\\
        &+\dfrac{4L_{\nabla^2g}^2L_f^2}{\mu_g^2}\left(\left\Vert\mathbf{x}^{(t)}-\mathbf{\bar{x}}^{(t)}\right\Vert_F^2+\left\Vert\mathbf{y}^{(t)}-\mathbf{\bar{y}}^{(t)}\right\Vert_F^2\right)+\dfrac{4L_f^2}{\mu_g^2}N{b_2^2}.
    \end{aligned}
\end{equation}

For the second term on the right-hand side of \eqref{consensus z start0}, we have: 
\begin{equation}
    \begin{aligned}
    \label{consensus z guji2}
        &\sum_{i=1}^N\left\Vert\nabla_{2}f_i(x^{(t)}_{i},y^{(t)}_{i})-\nabla_{2}f(\bar{x}^{(t)},\bar{y}^{(t)})\right\Vert^2
        \leq4L_{\nabla f}^2\left(\left\Vert\mathbf{x}^{(t)}-\mathbf{\bar{x}}^{(t)}\right\Vert_F^2+\left\Vert\mathbf{y}^{(t)}-\mathbf{\bar{y}}^{(t)}\right\Vert_F^2\right)+4N{b_1^2}.
    \end{aligned}
\end{equation}

Substituting \eqref{consensus z start0}, \eqref{consensus z guji1} and \eqref{consensus z guji2} into \eqref{consensus z start111}, then using \eqref{chafen_H}, we have:
\begin{equation}
\label{consensus z start01}
\begin{aligned}
&\mathbb{E}_t\left[\left\Vert\mathbf{z}^{(t+1)}-\bar{\mathbf{z}}^{(t+1)}\right\Vert_F^2\right]\\
\leq&\rho\left\Vert\mathbf{z}^{(t)}-\bar{\mathbf{z}}^{(t)}\right\Vert_F^2+\dfrac{8N\rho^2\gamma^2}{1-\rho}\left(L_{\nabla f}^2+\dfrac{L_{\nabla^2g}^2L_f^2}{\mu_g^2}\right)\left\Vert\bar{y}^{(t)}-\bar{y}_{\star}^{(t)}\right\Vert^2\\
&+\left(\dfrac{32N\rho^2\gamma^2}{1-\rho}L^2+6N\rho^2\gamma^2\sigma^2\right)\left\Vert \bar{z}^{(t)}-\bar{z}^{(t)}_{\star}\right\Vert^2+\left(\dfrac{64\rho^2\gamma^2}{1-\rho}L^2+6\rho^2\gamma^2\sigma^2\right)\Delta_t^2\\
&+{\dfrac{66\rho^2\gamma^2}{1-\rho}N\left(b_1^2+\dfrac{L_f^2}{\mu_g^2}b_2^2\right)}+\dfrac{2N\rho^2\gamma^2}{1-\rho}\iota^2+6N\rho^2\gamma^2\left(1+\dfrac{L_f^2}{\mu_g^2}\right)\sigma^2.
\end{aligned}
\end{equation}

Taking the expectation and summation 
 on both sides of \eqref{consensus z start01}, we can get:
\begin{equation}
\label{consensus z start02}
\begin{aligned}
&\mathbb{E}\left[\sum_{t=0}^{T-1} \left\Vert \mathbf{z}^{(t)}-\mathbf{\bar{z}}^{(t)}\right\Vert_F^2\right]\\
\leq&\left(\dfrac{64\rho^2\gamma^2}{(1-\rho)^2}L^2+\dfrac{6\rho^2\gamma^2}{1-\rho}\sigma^2\right)\mathbb{E}\left[\sum_{t=0}^{T-1}\Delta_t^2\right]+{\dfrac{66\rho^2\gamma^2}{(1-\rho)^2}NT\left(b_1^2+\dfrac{L_f^2}{\mu_g^2}b_2^2\right)}+\dfrac{2NT\rho^2\gamma^2}{1-\rho}\iota^2\\
&+\dfrac{8N\rho^2\gamma^2}{(1-\rho)^2}\left(L_{\nabla f}^2+\dfrac{L_{\nabla^2g}^2L_f^2}{\mu_g^2}\right)\mathbb{E}\left[\sum_{t=0}^{T-1}\left\Vert\bar{y}^{(t)}-\bar{y}_{\star}^{(t)}\right\Vert^2\right]+{\dfrac{1}{1-\rho}\left\Vert \mathbf{z}^{(0)}-\mathbf{\bar{z}}^{(0)}\right\Vert_F^2}\\
&+\left(\dfrac{32N\rho^2\gamma^2}{(1-\rho)^2}L^2+\dfrac{6N\rho^2\gamma^2}{1-\rho}\sigma^2\right)\mathbb{E}\left[\sum_{t=0}^{T-1}\left\Vert \bar{z}^{(t)}-\bar{z}^{(t)}_{\star}\right\Vert^2\right]+\dfrac{6NT\rho^2\gamma^2}{1-\rho}\left(1+\dfrac{L_f^2}{\mu_g^2}\right)\sigma^2.
\end{aligned}
\end{equation}

\end{proof}

Through Lemma \ref{lem:consensusx}, Lemma \ref{lem:consensusy}, and Lemma \ref{lem:consensusz}, we can present an upper bound of the consensus error during the iterations in Algorithm \ref{alg:DeMA-SOBA}, which is concluded by the following lemma \ref{lem:consensus xyz}:

\begin{lemma}
\label{lem:consensus xyz}
Suppose Assumption \ref{assumption:smooth}, \ref{assumption: data heterogeneity}, \ref{assumption:unbiased}, and \ref{assumption: gossip communication} hold, and the step-sizes $\alpha, \beta,\gamma$ satisfy:
\begin{align}
\label{eq:assumptioncon con}
    208(\kappa^2\tau^2\alpha^2+\kappa^2\beta^2+\gamma^2)\rho^2\left(\dfrac{L^2}{(1-\rho)^2}+\dfrac{\sigma^2}{1-\rho}\right)\leq\dfrac{1}{3}.
\end{align}
Then in Algorithm \ref{alg:DeMA-SOBA}, we have:
\begin{equation}
\label{eq:final11-xy-consesus11}
    \begin{aligned}
        \sum_{t=0}^{T-1} \alpha\mathbb{E}\left[\dfrac{1}{N}\Delta_t^2\right]\leq&\dfrac{C_{Dy}\rho^2\alpha^2}{(1-\rho)^2}\sum_{t=0}^{T-1} \alpha\mathbb{E}\left[\left\Vert\bar{y}^{(t)}-\bar{y}^{(t)}_{\star}\right\Vert^2\right]+\dfrac{C_{Dz}\rho^2\alpha^2}{(1-\rho)^2}\sum_{t=0}^{T-1} \alpha\mathbb{E}\left[\left\Vert\bar{z}^{(t)}-\bar{z}^{(t)}_{\star}\right\Vert^2\right]\\
        &+{\dfrac{C_{D1}T\rho^2\alpha^3(b_1^2+\kappa^2b_2^2)}{(1-\rho)^2}}+\dfrac{C_{D2}T\rho^2\alpha^3\sigma^2}{1-\rho}+\dfrac{C_{D3}T\rho^2\alpha^3\iota^2}{(1-\rho)^2}+{\dfrac{\alpha}{1-\rho}\cdot\dfrac{1}{N}\Delta_0^2}.
    \end{aligned}
\end{equation}
where the constants are defined as:
\begin{align*}
C_{Dy}=C_{Dz}=&384(\tau^2\kappa^2+\kappa^2c_1^2+c_2^2)(L^2+\sigma^2)=\mathcal{O}(\tau^2\kappa^2+\kappa^2c_1^2+c_2^2),\\ 
{C_{D1}=}&{384(\tau^2\kappa^2+\kappa^2c_1^2+c_2^2)=\mathcal{O}(\tau^2\kappa^2+\kappa^2c_1^2+c_2^2)},\\
C_{D2}=&384(\tau^2\kappa^2+\kappa^2c_1^2+c_2^2)\left(1+\dfrac{L_f^2}{\mu_g^2}\right)=\mathcal{O}(\kappa^2(\tau^2\kappa^2+\kappa^2c_1^2+c_2^2)),\\ 
C_{D3}=&24(\tau^2\kappa^2+c_2^2)=\mathcal{O}(\tau^2\kappa^2+c_2^2).
\end{align*}
\end{lemma}

\begin{proof}
    From \eqref{eq:final-x-consesus2}, \eqref{eq:final-y-consesus2}, and \eqref{eq:final-z-consesus2}, we have:
    \begin{equation}
    \label{final consensus start}
        \begin{aligned}
            &\mathbb{E}\left[\sum_{t=0}^{T-1} \Delta_t^2\right]\\
            \leq&\kappa^2\mathbb{E}\left[\sum_{t=0}^{T-1} \left\Vert \mathbf{x}^{(t)}-\mathbf{\bar{x}}^{(t)}\right\Vert_F^2\right]+\kappa^2\mathbb{E}\left[\sum_{t=0}^{T-1} \left\Vert \mathbf{y}^{(t)}-\mathbf{\bar{y}}^{(t)}\right\Vert_F^2\right]+\mathbb{E}\left[\sum_{t=0}^{T-1} \left\Vert \mathbf{z}^{(t)}-\mathbf{\bar{z}}^{(t)}\right\Vert_F^2\right]\\
            \leq&208(\kappa^2\tau^2\alpha^2+\kappa^2\beta^2+\gamma^2)\rho^2\left(\dfrac{L^2}{(1-\rho)^2}+\dfrac{\sigma^2}{1-\rho}\right)\mathbb{E}\left[\sum_{t=0}^{T-1} \Delta_t^2\right]\\
            &+192(\kappa^2\tau^2\alpha^2+\kappa^2\beta^2+\gamma^2)\rho^2\left(\dfrac{L^2}{(1-\rho)^2}+\dfrac{\sigma^2}{1-\rho}\right)N\sum_{t=0}^{T-1}\mathbb{E}\left[\left\Vert\bar{y}^{(t)}-\bar{y}^{(t)}_{\star}\right\Vert^2+\left\Vert\bar{z}^{(t)}-\bar{z}^{(t)}_{\star}\right\Vert^2\right]\\
            &+{\dfrac{192(\kappa^2\tau^2\alpha^2+\kappa^2\beta^2+\gamma^2)\rho^2}{(1-\rho)^2}NT\left(b_1^2+\dfrac{L_f^2}{\mu_g^2}b_2^2\right)}+\dfrac{12(\kappa^2\tau^2\alpha^2+\gamma^2)\rho^2}{(1-\rho)^2}NT\iota^2\\
            &+\dfrac{192(\kappa^2\tau^2\alpha^2+\kappa^2\beta^2+\gamma^2)\rho^2}{1-\rho}\left(1+\dfrac{L_f^2}{\mu_g^2}\right)NT\sigma^2+{\dfrac{1}{1-\rho}\Delta_0^2}.
        \end{aligned}
    \end{equation}
    Then, multiplying $\dfrac{\alpha}{N}$ on both sides of \eqref{final consensus start} and substituting \eqref{eq:assumptioncon con} into it, we finally achieve:
    \begin{equation}
        \begin{aligned}
            \sum_{t=0}^{T-1} \alpha\mathbb{E}\left[\dfrac{1}{N}\Delta_t^2\right]\leq&\dfrac{C_{Dy}\rho^2\alpha^2}{(1-\rho)^2}\sum_{t=0}^{T-1} \alpha\mathbb{E}\left[\left\Vert\bar{y}^{(t)}-\bar{y}^{(t)}_{\star}\right\Vert^2\right]+\dfrac{C_{Dz}\rho^2\alpha^2}{(1-\rho)^2}\sum_{t=0}^{T-1} \alpha\mathbb{E}\left[\left\Vert\bar{z}^{(t)}-\bar{z}^{(t)}_{\star}\right\Vert^2\right]\\
        &+{\dfrac{C_{D1}T\rho^2\alpha^3(b_1^2+\kappa^2b_2^2)}{(1-\rho)^2}}+\dfrac{C_{D2}T\rho^2\alpha^3\sigma^2}{1-\rho}+\dfrac{C_{D3}T\rho^2\alpha^3\iota^2}{(1-\rho)^2}+{\dfrac{\alpha}{1-\rho}\cdot\dfrac{1}{N}\Delta_0^2},
        \end{aligned}
    \end{equation}
where $C_{Dy}, C_{Dz}, C_{D1}, C_{D2}, C_{D3}$ have already been defined above.
\end{proof}

Finally, we combine the conclusion of Lemma \ref{lem:consensus1} and   Lemma \ref{lem:consensus xyz} together, obtaining the following lemma:
\begin{lemma}
    Suppose Assumption \ref{assumption:smooth}, \ref{assumption: data heterogeneity}, \ref{assumption:unbiased}, and \ref{assumption: gossip communication} hold. If $\alpha, \beta, \gamma$ satisfy \eqref{eq:assumptiony}, \eqref{eq:assumptionz}, \eqref{eq:assumptioncon con}, and
    \begin{align}
    \label{eq:assumptioncon}
    \dfrac{\rho^2\alpha^2}{(1-\rho)^2}\tilde{C}_D((1+\kappa^2)C_{yD}+C_{zD}{ +1})\leq\dfrac{1}{3},
    \end{align}
    where $\tilde{C}_D=\max\{C_{Dy},C_{Dz}\}$, in Algorithm \ref{alg:DeMA-SOBA}  we have:
\begin{equation}
\label{eq:final11_xy_consesus0}
    \begin{aligned}
        &(1+\kappa^2)\sum_{t=0}^{T-1} \alpha\mathbb{E}\left[\left\Vert\bar{y}^{(t)}-\bar{y}^{(t)}_{\star}\right\Vert^2\right]+\sum_{t=0}^{T-1} \alpha\mathbb{E}\left[\left\Vert\bar{z}^{(t)}-\bar{z}^{(t)}_{\star}\right\Vert^2\right]\\
        \leq&2((1+\kappa^2)C_{yx}+C_{zx})\sum_{t=0}^{T-1} \alpha\mathbb{E}\left[\left\Vert\bar{h}^{(t)}\right\Vert^2\right]+2((1+\kappa^2)C_{y1}+C_{z1})\dfrac{T\alpha^2}{N}\\
        &+2((1+\kappa^2)C_{y2}+C_{z2})\alpha+2C_{z3}T\alpha^2\iota^2+{\dfrac{2((1+\kappa^2)C_{yD}+C_{zD})C_{D1}T\rho^2\alpha^3(b_1^2+\kappa^2b_2^2)}{(1-\rho)^2}}\\
        &+\dfrac{2((1+\kappa^2)C_{yD}+C_{zD})C_{D2}T\rho^2\alpha^3\sigma^2}{1-\rho}+\dfrac{2((1+\kappa^2)C_{yD}+C_{zD})C_{D3}T\rho^2\alpha^3\iota^2}{(1-\rho)^2}\\
        &+{\dfrac{2((1+\kappa^2)C_{yD}+C_{zD})\alpha}{1-\rho}\cdot\dfrac{1}{N}\Delta_0^2}.
    \end{aligned}
\end{equation}
\end{lemma}
\begin{proof}
Combining Eq. \eqref{eq:final11_xy_consesus1} and \eqref{eq:final11-xy-consesus11} together, we have:  
\begin{equation}
\label{eq:final11_xy_consesus2}
    \begin{aligned}
        &(1+\kappa^2)\sum_{t=0}^{T-1} \alpha\mathbb{E}\left[\left\Vert\bar{y}^{(t)}-\bar{y}^{(t)}_{\star}\right\Vert^2\right]+\sum_{t=0}^{T-1} \alpha\mathbb{E}\left[\left\Vert\bar{z}^{(t)}-\bar{z}^{(t)}_{\star}\right\Vert^2\right]\\
        \leq&((1+\kappa^2)C_{yx}+C_{zx})\sum_{t=0}^{T-1} \alpha\mathbb{E}\left[\left\Vert\bar{h}^{(t)}\right\Vert^2\right]+((1+\kappa^2)C_{y1}+C_{z1})\dfrac{T\alpha^2}{N}\\
        &+((1+\kappa^2)C_{y2}+C_{z2})\alpha+C_{z3}T\alpha^2\iota^2\\
        &+\dfrac{((1+\kappa^2)C_{yD}+C_{zD})\tilde{C}_{D}\rho^2\alpha^2}{(1-\rho)^2}\left(\sum_{t=0}^{T-1} \alpha\mathbb{E}\left[\left\Vert\bar{y}^{(t)}-\bar{y}^{(t)}_{\star}\right\Vert^2\right]+\sum_{t=0}^{T-1} \alpha\mathbb{E}\left[\left\Vert\bar{z}^{(t)}-\bar{z}^{(t)}_{\star}\right\Vert^2\right]\right)\\
        &+{\dfrac{((1+\kappa^2)C_{yD}+C_{zD})C_{D1}T\rho^2\alpha^3(b_1^2+\kappa^2b_2^2)}{(1-\rho)^2}}+\dfrac{((1+\kappa^2)C_{yD}+C_{zD})C_{D2}T\rho^2\alpha^3\sigma^2}{1-\rho}\\
        &+\dfrac{((1+\kappa^2)C_{yD}+C_{zD})C_{D3}T\rho^2\alpha^3\iota^2}{(1-\rho)^2}+{\dfrac{((1+\kappa^2)C_{yD}+C_{zD})\alpha}{1-\rho}\cdot\dfrac{1}{N}\Delta_0^2}.
    \end{aligned}
\end{equation}

Then, using \eqref{eq:assumptioncon}, we finish the proof of this lemma.

\end{proof}

\subsection{The step-sizes selection and the conclusion}
In the end, the following lemma presents the proof of Eq. \eqref{eq:convergence_MADeSOBA}
\begin{lemma}
\label{final_lemma}
    Suppose Assumption \ref{assumption:smooth}, \ref{assumption: data heterogeneity}, \ref{assumption:unbiased}, and \ref{assumption: gossip communication} hold. There exist $c_0,c_1,c_2,c_3>0$ such that if $\alpha=c_0\sqrt{{N}/{T}}$, and $\alpha_t\equiv\alpha$, $\beta_t\equiv c_1\alpha$, $\gamma_t\equiv c_2\alpha$, $\theta_t\equiv c_3\alpha$ for any $t=0,1,\cdots,T-1$,  the iterates $\bar{x}^{(t)}$ in Algorithm \ref{alg:DeMA-SOBA} satisfy:
\begin{align}
    &\ \dfrac{1}{T}\sum_{t=0}^{T-1}\EE\left[\left\|\nabla\Phi(\bar{x}^{(t)})\right\|^2\right] \nonumber \\
    \lesssim&\ \dfrac{\kappa^5}{\sqrt{NT}} + \dfrac{\rho^{\frac{2}{3}}\kappa^{6}}{(1-\rho)^{\frac{1}{3}}T^{\frac{2}{3}}}+{\dfrac{\rho^{\frac{2}{3}}(b_1^{\frac{2}{3}}\kappa^{\frac{16}{3}}+b_2^{\frac{2}{3}}\kappa^{6})}{(1-\rho)^{\frac{2}{3}}T^{\frac{2}{3}}}}+\dfrac{\rho\kappa^{6}}{(1-\rho)T}+{\dfrac{\kappa^4}{(1-\rho)NT}}+\dfrac{\kappa^{13}}{T}+\dfrac{\kappa^7}{NT},
\end{align}
\end{lemma}

\begin{proof}
We initially assume that the conclusions of all preceding lemmas are valid, and provide explicit expressions of $\alpha,
\beta,\gamma$ at the conclusion of the proof.

Assume that the step-size $\alpha$ satisfies:
    \begin{equation}
        \label{as:final}
        \dfrac{(1+c_3)\sigma^2\alpha}{N}\leq L^2,
    \end{equation}

From \eqref{eq:final_avgh11}, we have:
 \begin{equation}\label{H1}
        \begin{aligned}
        &\alpha\sum_{t=0}^{T-1} \mathbb{E}\left[\left\Vert\bar{h}^{(t)}\right\Vert^2\right]+\alpha\sum_{t=0}^{T} \mathbb{E}\left[\left\Vert\bar{h}^{(t)}-\nabla\Phi(\bar{x}^{(t)})\right\Vert^2\right]\\
        \leq&\dfrac{2}{\tau}\mathbb{E}[E_0]+{\dfrac{1}{c_3}\mathbb{E}\left[\left\Vert\bar{h}^{(0)}-\nabla\Phi(\bar{x}^{(0)})\right\Vert^2\right]}+\dfrac{1}{2}\cdot\alpha\sum_{t=0}^{T-1} \mathbb{E}\left[\left\Vert\bar{h}^{(t)}-\nabla\Phi(\bar{x}^{(t)})\right\Vert^2\right]\\
        &+\dfrac{4L_{\nabla\Phi}^2\tau^2}{c_3^2}\alpha\sum_{t=0}^{T-1} \mathbb{E}\left[\left\Vert\bar{h}^{(t)}\right\Vert^2\right]+6L^2\dfrac{1}{N}\cdot\alpha\sum_{t=0}^{T-1}\mathbb{E}[\Delta_t^2]+\dfrac{6(1+c_3)T\alpha^2}{N}\left(1+\dfrac{L_f^2}{\mu_g^2}\right)\sigma^2\\
        &+10T\alpha\iota^2+66L^2(1+\kappa^2)\sum_{t=0}^{T-1}\alpha\mathbb{E}\left[\left\Vert \bar{y}^{(t)}-\bar{y}^{(t)}_{\star}\right\Vert^2\right]+66L^2\sum_{t=0}^{T-1}\alpha\mathbb{E}\left[\left\Vert \bar{z}^{(t)}-\bar{z}^{(t)}_{\star}\right\Vert^2\right].
        \end{aligned}
    \end{equation}

Then, plugging \eqref{eq:final11-xy-consesus11} and \eqref{eq:final11_xy_consesus0} into \eqref{H1}, we get: 
 \begin{equation}\label{H21}
        \begin{aligned}
        &\alpha\sum_{t=0}^{T-1} \mathbb{E}\left[\left\Vert\bar{h}^{(t)}\right\Vert^2\right]+\alpha\sum_{t=0}^{T} \mathbb{E}\left[\left\Vert\bar{h}^{(t)}-\nabla\Phi(\bar{x}^{(t)})\right\Vert^2\right]\\
        \leq&\dfrac{2}{\tau}\mathbb{E}[E_0]+{\dfrac{1}{c_3}\mathbb{E}\left[\left\Vert\bar{h}^{(0)}-\nabla\Phi(\bar{x}^{(0)})\right\Vert^2\right]}+\dfrac{4L_{\nabla\Phi}^2\tau^2}{c_3^2}\alpha\sum_{t=0}^{T-1} \mathbb{E}\left[\left\Vert\bar{h}^{(t)}\right\Vert^2\right]\\
        &+\dfrac{1}{2}\cdot\alpha\sum_{t=0}^{T-1} \mathbb{E}\left[\left\Vert\bar{h}^{(t)}-\nabla\Phi(\bar{x}^{(t)})\right\Vert^2\right]+\left(\dfrac{6(1+c_3)T\alpha^2}{N}\left(1+\dfrac{L_f^2}{\mu_g^2}\right)+\dfrac{66L^2C_{D2}T\rho^2\alpha^3}{1-\rho}\right)\sigma^2\\
        &+\left(10T\alpha+\dfrac{66L^2C_{D3}T\rho^2\alpha^3}{(1-\rho)^2}\right)\iota^2+{\dfrac{66L^2C_{D1}T\rho^2\alpha^3}{(1-\rho)^2}(b_1^2+\kappa^2b_2^2)}+{\dfrac{66L^2\alpha}{1-\rho}\cdot\dfrac{1}{N}\Delta_0^2}\\
        &+132L^2\left((1+\kappa^2)\sum_{t=0}^{T-1} \alpha\mathbb{E}\left[\left\Vert\bar{y}^{(t)}-\bar{y}^{(t)}_{\star}\right\Vert^2\right]+\sum_{t=0}^{T-1} \alpha\mathbb{E}\left[\left\Vert\bar{z}^{(t)}-\bar{z}^{(t)}_{\star}\right\Vert^2\right]\right).
        \end{aligned}
    \end{equation}
Then, we have:
 \begin{equation}\label{H2}
        \begin{aligned}
        &\alpha\sum_{t=0}^{T-1} \mathbb{E}\left[\left\Vert\bar{h}^{(t)}\right\Vert^2\right]+\alpha\sum_{t=0}^{T} \mathbb{E}\left[\left\Vert\bar{h}^{(t)}-\nabla\Phi(\bar{x}^{(t)})\right\Vert^2\right]\\
        \leq&\dfrac{2}{\tau}\mathbb{E}[E_0]+{\dfrac{1}{c_3}\mathbb{E}\left[\left\Vert\bar{h}^{(0)}-\nabla\Phi(\bar{x}^{(0)})\right\Vert^2\right]}+\dfrac{1}{2}\cdot\alpha\sum_{t=0}^{T-1} \mathbb{E}\left[\left\Vert\bar{h}^{(t)}-\nabla\Phi(\bar{x}^{(t)})\right\Vert^2\right]\\
        &+\left(\dfrac{4L_{\nabla\Phi}^2}{c_3^2}+264L^2((1+\kappa^2)C_{yx}+C_{zx})\right)\sum_{t=0}^{T-1}\alpha\mathbb{E}\left[\left\Vert\bar{h}^{(t)}\right\Vert^2\right]+264L^2((1+\kappa^2)C_{y2}+C_{z2})\tau^2\alpha\\
        &+\left(6(1+c_3)(1+\kappa^2)\sigma^2+264L^2((1+\kappa^2)C_{y1}+C_{z1})\right)\dfrac{T\alpha^2}{N}+\left(10+264L^2C_{z3}\right)T\alpha\iota^2\\
        &+{\hat{C}_D\left(\dfrac{C_{D1}T\rho^2\alpha^3(b_1^2+\kappa^2b_2^2)}{(1-\rho)^2}+\dfrac{C_{D2}T\rho^2\alpha^3\sigma^2}{1-\rho}+\dfrac{C_{D3}T\rho^2\alpha^3\iota^2}{(1-\rho)^2}+\dfrac{\alpha\Delta_0^2}{(1-\rho)N}\right)},
        \end{aligned}
    \end{equation}
{where $\hat{C}_D:=66L^2+264L^2((1+\kappa^2)C_{yD}+C_{zD})$}, and the last inequality is due to $\dfrac{\tilde{C}_{D}\rho^2\alpha^2}{(1-\rho)^2}\leq\dfrac{1}{3}$. Then, if we take:
\begin{equation}
\label{eq:constants_c}
\begin{aligned}
      C_x=&\dfrac{4L_{\nabla\Phi}^2\tau^2}{c_3^2}+264L^2\tau^2((1+\kappa^2)C_{yx}+C_{zx})=\mathcal{O}\left(\dfrac{\kappa^8\tau^2}{c_1^2}+\dfrac{\kappa^8\tau^2}{c_2^2}+\dfrac{\kappa^6\tau^2}{c_3^2}\right),\\
    C_0=&264L^2((1+\kappa^2)C_{y2}+C_{z2})+\dfrac{2}{\tau}\mathbb{E}[E_0]+{\dfrac{1}{c_3}\mathbb{E}\left[\left\Vert\bar{h}^{(0)}-\nabla\Phi(\bar{x}^{(0)})\right\Vert^2\right]}\\
    =&\mathcal{O}\left(\dfrac{\kappa^5}{c_1}+\dfrac{\kappa}{c_2}+\dfrac{1}{c_3}+\dfrac{1}{\tau}\right),\\
    C_1=&6(1+c_3)(1+\kappa^2)\sigma^2+264L^2((1+\kappa^2)C_{y1}+C_{z1})=\mathcal{O}(c_1\kappa^5+c_2\kappa^3+c_3\kappa^2+1),\\
    C_2=&10+264L^2C_{z3}=\mathcal{O}(\kappa^2),\\
    \hat{C}_D=&66L^2+264L^2((1+\kappa^2)C_{yD}+C_{zD})=\mathcal{O}(\kappa^4),  
\end{aligned}
\end{equation}
we can get:
 \begin{equation}\label{H}
        \begin{aligned}
            &\alpha\sum_{t=0}^{T-1} \mathbb{E}\left[\left\Vert\bar{h}^{(t)}\right\Vert^2\right]+\alpha\sum_{t=0}^{T} \mathbb{E}\left[\left\Vert\bar{h}^{(t)}-\nabla\Phi(\bar{x}^{(t)})\right\Vert^2\right]\\
            \le&C_x\sum_{t=0}^{T-1} \alpha\mathbb{E}\left[\left\Vert\bar{h}^{(t)}\right\Vert^2\right]+\dfrac{1}{2}\sum_{t=0}^{T-1} \alpha\mathbb{E}\left[\left\Vert\bar{h}^{(t)}-\nabla\Phi(\bar{x}^{(t)})\right\Vert^2\right]+C_1\dfrac{T\alpha^2}{N}+C_2T\alpha^2\iota^2+C_0\alpha\\
            &+ \hat{C}_D\left[{\dfrac{C_{D1}T\rho^2\alpha^3(b_1^2+\kappa^2b_2^2)}{(1-\rho)^2}}+\dfrac{C_{D2}T\rho^2\alpha^3\sigma^2}{1-\rho}+\dfrac{C_{D3}T\rho^2\alpha^3\iota^2}{(1-\rho)^2}+{\dfrac{\alpha\Delta_0^2}{(1-\rho)N}}\right].
        \end{aligned}
    \end{equation}
    
Define
\begin{align*}
    \alpha_1&=\left(\dfrac{C_0N}{C_1T}\right)^{\frac{1}{2}},\quad{\alpha_{21}=\left(\dfrac{C_0(1-\rho)^2}{\hat{C}_DC_{D1}\rho^2Tb_1^2}\right)^{\frac{1}{3}},\quad\alpha_{22}=\left(\dfrac{C_0(1-\rho)^2}{\hat{C}_DC_{D1}\rho^2T\kappa^2b_2^2}\right)^{\frac{1}{3}}},\\
    \alpha_3&=\left(\dfrac{C_0(1-\rho)}{\hat{C}_DC_{D2}\rho^2T}\right)^{\frac{1}{3}}, \alpha_4=\left(\dfrac{(1-\rho)^2}{208\rho^2(\kappa^2\tau^2+\kappa^2c_1^2+c_2^2)\left(L^2+\sigma^2\right)}\right)^{\frac{1}{2}},\\
    \alpha_5&=\left(\dfrac{(1-\rho)^2}{3\rho^2\tilde{C}_D((1+\kappa^2)C_{yD}+C_{zD}{ +1})}\right)^{\frac{1}{2}},\\
    C_{\alpha1}&={c_1(\mu_g+L_{\nabla g})}+2{c_2\mu_g}+\dfrac{2\tau(c_3L_{\nabla\Phi}+L_{\nabla\eta})}{c_3}+\dfrac{20c_3}{\tau^2}+1,\\
C_{\alpha2}&=\dfrac{144c_2\sigma^2}{N\mu_g}+\dfrac{c_2\mu_g^2\sigma^2}{NL^2}+\dfrac{(1+c_3)\sigma_{g2}^2}{NL_{\nabla^2g}^2}.
\end{align*}
Then take
\begin{align}
\label{def:alpha}
    \alpha=\left(\dfrac{1}{\alpha_1}+{\dfrac{1}{\alpha_{21}}+\dfrac{1}{\alpha_{22}}}+\dfrac{1}{\alpha_3}+\dfrac{1}{\alpha_4}+\dfrac{1}{\alpha_5}+C_{\alpha1}+C_{\alpha2}\right)^{-1}. 
\end{align}

Review that the step-sizes are restricted by:
\begin{itemize}
    \item $\text{Eq. \eqref{eq:assumptionhyper}: \quad}
\alpha\leq\min\left\{\dfrac{c_3}{2\tau(c_3L_{\nabla\Phi}+L_{\nabla\eta})},\dfrac{\tau^2}{20c_3},1\right\},\quad\gamma\mu_g\leq1$;
    \item $\text{Eq. \eqref{eq:assumptiony}: \quad}
\beta(\mu_g+L_{\nabla g})\leq1$;
    \item $\text{Eq. \eqref{eq:assumptionz}: \quad}
-\dfrac{1}{6}\gamma\mu_g+ {\dfrac{6}{N}}\gamma^2\sigma^2+\dfrac{1}{2}\gamma^3\mu_g^3\leq0,\quad \gamma\mu_g\leq\min\left\{1,\dfrac{NL^2}{\mu_g\sigma^2}\right\}$;
    \item $\text{Eq. \eqref{eq:assumptioncon con}: \quad}
624(\kappa^2\tau^2\alpha^2+\kappa^2\beta^2+\gamma^2)\rho^2\left(\dfrac{L^2}{(1-\rho)^2}+\dfrac{\sigma^2}{1-\rho}\right)\leq\dfrac{1}{3}$;
    \item $\text{Eq. \eqref{eq:assumptioncon}: \quad}
\dfrac{\rho^2\alpha^2}{(1-\rho)^2}\tilde{C}_D((1+\kappa^2)C_{yD}+C_{zD}{ +1})\leq\dfrac{1}{3}$;
    \item $\text{Eq. \eqref{as:final}: \quad}
\dfrac{(1+c_3)\sigma_{g2}^2\alpha}{N}\leq L^2$.
\end{itemize}

With \eqref{def:alpha}, all above restrictions hold by the following validation:
\begin{itemize}
    \item Validation of Eq. \eqref{eq:assumptionhyper}: \quad As $\max\left\{\dfrac{2\alpha\tau(c_3L_{\nabla\Phi}+L_{\nabla\eta})}{c_3},\dfrac{20c_3\alpha}{\tau^2},2\gamma\mu_g,\alpha\right\}\leq\alpha C_{\alpha1}\leq1$, thus $\alpha\leq\min\left\{\dfrac{c_3}{2\tau(c_3L_{\nabla\Phi}+L_{\nabla\eta})},\dfrac{\tau^2}{20c_3},1\right\}$ and $\gamma\mu_g\leq\dfrac{1}{2}<1$.
    \item Validation of Eq. \eqref{eq:assumptiony}: \quad $\beta(\mu_g+L_{\nabla g})\leq\alpha C_{\alpha1}\leq1$.
    \item Validation of Eq. \eqref{eq:assumptionz}: \quad As $\dfrac{\gamma\mu_g^2\sigma^2}{NL^2}\leq\alpha C_{\alpha2}\leq1$,  $\gamma\mu_g\leq\min\left\{1,\dfrac{NL^2}{\mu_g\sigma^2}\right\}$. 

    Moreover, it holds that $\gamma\mu_g\leq\dfrac{1}{2}$. Then, $\dfrac{1}{6}\gamma\mu_g-\dfrac{1}{2}\gamma^3\mu_g^3\geq\dfrac{1}{6}\gamma\mu_g-\dfrac{1}{8}\gamma\mu_g=\dfrac{1}{24}\gamma\mu_g$. And from $\dfrac{144\gamma\sigma^2}{N\mu_g}\leq\alpha C_{\alpha2}\leq1$, we can obtain that ${\dfrac{6}{N}}\gamma^2\sigma^2\leq\dfrac{1}{24}\gamma\mu_g$. Thus Eq. \eqref{eq:assumptionz} holds as $ {\dfrac{6}{N}}\gamma^2\sigma^2\leq\dfrac{1}{6}\gamma\mu_g-\dfrac{1}{2}\gamma^3\mu_g^3$.
    \item Validation of Eq. \eqref{eq:assumptioncon con}: \quad As $\alpha/\alpha_4\leq1$, it holds that $\dfrac{624\rho^2(\kappa^2\tau^2+\kappa^2c_1^2+c_2^2)\left(L^2+\sigma^2\right)}{(1-\rho)^2}\leq1$. Thus Eq. \eqref{eq:assumptioncon con} holds.
    \item Validation of Eq. \eqref{eq:assumptioncon}: \quad As $\alpha/\alpha_5\leq1$, it holds that $\dfrac{3\rho^2\alpha^2\tilde{C}_D((1+\kappa^2)C_{yD}+C_{zD})}{(1-\rho)^2}\leq1$. Thus Eq. \eqref{eq:assumptioncon} holds.
    \item Validation of Eq. \eqref{as:final}: \quad As $\alpha C_{\alpha2}\leq1$, it holds that $\dfrac{(1+c_3)\sigma_{g2}^2\alpha}{NL^2}\leq 1$. Then Eq. \eqref{as:final} holds.
\end{itemize}

Thus, all these restrictions hold. Consequently, all preceding lemmas hold. 

Taking {$c_1=\Theta(1)$, $c_2=\Theta(1)$, $c_3=\Theta(1)$, $\tau=\Theta(\kappa^{-4})$} such that $C_x<\dfrac{1}{2}$, we get:
\begin{equation}\label{kap}
    \begin{aligned}
    C_x&=\mathcal{O}(1),\quad C_0=\mathcal{O}\left(\kappa^5\right),\quad
    C_1=\mathcal{O}(\kappa^5),\quad C_2=\mathcal{O}(\kappa^2),\\
    C_{D1}&={\mathcal{O}(\kappa^2)},\quad C_{D2}=\mathcal{O}(\kappa^2),\quad \hat{C}_D=\mathcal{O}\left(\kappa^4\right).
    \end{aligned}
\end{equation}

Then the order of $\kappa$ in the following terms can be provided as:
\begin{equation}\label{kap1}
    \begin{aligned}
    \alpha_1&=\mathcal{O}(\kappa^0),\quad{\alpha_{21}=\mathcal{O}(\kappa^{-\frac{1}{3}}),\quad\alpha_{22}=\mathcal{O}(\kappa^{-1})},\quad\alpha_3=\mathcal{O}(\kappa^{-1}),\quad\alpha_4=\mathcal{O}(\kappa^{-1}),\\
    {\alpha_5}&={\mathcal{O}(\kappa^{-8})},\quad C_{\alpha1}=\mathcal{O}(\kappa^8),\quad C_{\alpha2}=\mathcal{O}(\kappa^2/N).
    \end{aligned}
\end{equation}

From \eqref{H}, we get:
\begin{equation}\label{final11Phi}
\begin{aligned}
    &\dfrac{1}{T}\sum_{t=0}^{T-1} \mathbb{E}\left[\left\Vert\nabla\Phi(\bar{x}^{(t)})\right\Vert^2\right]\\
    \le&{\dfrac{2}{T}}\sum_{t=0}^{T-1} \mathbb{E}\left[\left\Vert\bar{h}^{(t)}\right\Vert^2\right]+\dfrac{2}{T}\sum_{t=0}^{T-1} \mathbb{E}\left[\left\Vert\bar{h}^{(t)}-\nabla\Phi(\bar{x}^{(t)})\right\Vert^2\right]\\
    \leq&2C_1\dfrac{\alpha}{N}+2C_2\alpha\iota^2+\dfrac{2C_0}{\alpha T}+{2\hat{C}_DC_{D1}\dfrac{\rho^2\alpha^2(b_1^2+\kappa^2b_2^2)}{(1-\rho)^2}}+2\hat{C}_DC_{D2}\dfrac{\rho^2\alpha^2}{1-\rho}+2\hat{C}_DC_{D3}\dfrac{\rho^2\alpha^2\iota^2}{(1-\rho)^2}\\
    &+{\dfrac{2\hat{C}_D\Delta_0^2}{(1-\rho)NT}}.
\end{aligned}
\end{equation}

Substituting \eqref{def:alpha} into \eqref{final11Phi} and taking $\iota$ satisfying \[C_2\alpha\iota^2+\hat{C}_DC_{D3}\dfrac{\rho^2\alpha^2\iota^2}{(1-\rho)^2}\lesssim\dfrac{1}{{NT}},\]we obtain:
\begin{equation}
\begin{aligned}
\label{final0Phi}
    &\dfrac{1}{T}\sum_{t=0}^{T-1} \mathbb{E}\left[\left\Vert\nabla\Phi(\bar{x}^{(t)})\right\Vert^2\right]\\
    \leq&\dfrac{2C_0}{T}\left(\dfrac{1}{\alpha_1}+{\dfrac{1}{\alpha_{21}}+\dfrac{1}{\alpha_{22}}}+\dfrac{1}{\alpha_3}+\dfrac{1}{\alpha_4}+\dfrac{1}{\alpha_5}+C_{\alpha1}+C_{\alpha2}\right)+2C_1\dfrac{\alpha_1}{N}\\
    &+{\dfrac{2\hat{C}_DC_{D1}\rho^2\alpha_2^2}{(1-\rho)^2}(b_1^2+\kappa^2b_2^2)}+2\hat{C}_DC_{D2}\dfrac{\rho^2\alpha_3^2}{1-\rho}+2C_2\alpha\iota^2+2\hat{C}_DC_{D3}\dfrac{\rho^2\alpha^2\iota^2}{(1-\rho)^2}+{\dfrac{2\hat{C}_D\Delta_0^2}{(1-\rho)NT}}\\
    \lesssim&\dfrac{\kappa^5}{\sqrt{NT}}+\dfrac{\rho^{\frac{2}{3}}\kappa^{6}}{(1-\rho)^{\frac{1}{3}}T^{\frac{2}{3}}}+{\dfrac{\rho^{\frac{2}{3}}(b_1^{\frac{2}{3}}\kappa^{\frac{16}{3}}+b_2^{\frac{2}{3}}\kappa^{6})}{(1-\rho)^{\frac{2}{3}}T^{\frac{2}{3}}}}+\dfrac{\rho\kappa^{6}}{(1-\rho)T}+{\dfrac{\kappa^4}{(1-\rho)NT}}+\dfrac{\kappa^{13}}{T}+\dfrac{\kappa^7}{NT}.
\end{aligned}
\end{equation}
\end{proof}

\subsection{Proof of Corollary \ref{thm:transient time of DeMA-SOBA}}
\label{proof of transient time of D-SOBA}
\begin{proof}
    According to \eqref{eq:convergence_MADeSOBA}, the linear speed up stage can be achieved if $T$ is sufficiently large to make $1/\sqrt{NT}$ dominate the convergence rate, \ie,:
    \begin{align}
    \dfrac{1}{\sqrt{NT}}\lesssim\dfrac{1}{(1-\rho)^{\frac{1}{3}}T^{\frac{2}{3}}},\quad
    \dfrac{1}{\sqrt{NT}}\lesssim\dfrac{b^{\frac{2}{3}}}{(1-\rho)^{\frac{2}{3}}T^{\frac{2}{3}}},\quad
    \dfrac{1}{\sqrt{NT}}\lesssim\dfrac{1}{(1-\rho)T}.
    \end{align}
    The above constraints are equivalent to $T\lesssim\max\left\{ \dfrac{N^3}{(1-\rho)^2}, \dfrac{N^3b^2}{(1-\rho)^4}\right\}$.
\end{proof}

\subsection{Proof of Corollary \ref{thm:convergence of deterministic}}
\label{proof of deterministic}
\begin{proof}
    From \eqref{final11Phi}, if the stochastic noises are all equal to 0,  $C_1=C_{D2}=0$ but $C_0,\hat{C}_D,C_{D1}\not=0$, the average of square norm of the hyper-gradients, \ie, $\frac{1}{T}\sum_{t=0}^T\left\Vert\nabla\Phi(\bar{x}^{(t)})\right\Vert^2$, can be bounded as:
    \begin{align}
    \label{final11Phi_deter}
        &\dfrac{1}{T}\sum_{t=0}^{T-1} \left\Vert\nabla\Phi(\bar{x}^{(t)})\right\Vert^2\\
        \leq&\dfrac{2C_0}{\alpha T}+{2\hat{C}_DC_{D1}\dfrac{\rho^2\alpha^2(b_1^{2}+\kappa^2b_2^{2})}{(1-\rho)^2}}+2C_2\alpha\iota^2+2\hat{C}_DC_{D3}\dfrac{\rho^2\alpha^2\iota^2}{(1-\rho)^2}+{\dfrac{2\hat{C}_D\Delta_0^2}{(1-\rho)NT}}.
    \end{align}

    Define 
    \begin{align*}
        {\alpha_{21}'=\left(\dfrac{C_0(1-\rho)^2}{\hat{C}_DC_{D1}\rho^2b_1^2T}\right)^{\frac{1}{3}}=\mathcal{O}(\kappa^{\frac{1}{3}}),\quad
        \alpha_{22}'=\left(\dfrac{C_0(1-\rho)^2}{\hat{C}_DC_{D1}\rho^2b_2^2\kappa^2T}\right)^{\frac{1}{3}}=\mathcal{O}(\kappa^{-1})},
    \end{align*}

    and set
    \begin{align}
    \label{def:alpha_deter}
        \alpha=\left({\dfrac{1}{\alpha_{21}'}+\dfrac{1}{\alpha_{22}'}}+\dfrac{1}{\alpha_4}+\dfrac{1}{\alpha_5}+C_{\alpha1}+C_{\alpha2}\right)^{-1}.
    \end{align}
    Then \eqref{eq:assumptionhyper}, \eqref{eq:assumptiony}, \eqref{eq:assumptionz}, \eqref{eq:assumptioncon con}, \eqref{eq:assumptioncon}, and \eqref{as:final} hold. Taking $\iota$ satisfying \[2C_2\alpha\iota^2+2\hat{C}_DC_{D3}\dfrac{\rho^2\alpha^2\iota^2}{(1-\rho)^2}\lesssim\dfrac{1}{T},\] substituting \eqref{def:alpha_deter} into \eqref{final11Phi_deter} and using \eqref{kap}, we have:
    \begin{equation}
    \begin{aligned}
    \label{final0Phi}
        &\dfrac{1}{T}\sum_{t=0}^{T-1} \left\Vert\nabla\Phi(\bar{x}^{(t)})\right\Vert^2\\
        \leq&\dfrac{2C_0}{T}\left({\dfrac{1}{\alpha_{21}'}+\dfrac{1}{\alpha_{22}'}}+\dfrac{1}{\alpha_4}+\dfrac{1}{\alpha_5}+C_{\alpha1}+C_{\alpha2}\right)+{2\hat{C}_DC_{D1}\dfrac{\rho^2(\alpha_{21}')^2}{(1-\rho)^2}b^2+2\hat{C}_DC_{D1}\dfrac{\rho^2(\alpha_{22}')^2}{(1-\rho)^2}\kappa^2b_2^2}\\
    &+2C_2\alpha\iota^2+2\hat{C}_DC_{D3}\dfrac{\rho^2\alpha^2\iota^2}{(1-\rho)^2}+{\dfrac{2\hat{C}_D\Delta_0^2}{(1-\rho)NT}}\\
        \lesssim&{\dfrac{\rho^{\frac{2}{3}}(b_1^{\frac{2}{3}}\kappa^{\frac{16}{3}}+b_2^{\frac{2}{3}}\kappa^{6})}{T^{\frac{2}{3}}(1-\rho)^{\frac{2}{3}}}}+\dfrac{\rho\kappa^6}{(1-\rho)T}+{\dfrac{\kappa^4}{(1-\rho)NT}}+\dfrac{\kappa^{13}}{T}.
    \end{aligned}
    \end{equation}

\end{proof}
{
\section{Proof of the non-asymptotic concensus error}
\label{Proof of the non-asymptotic concensus error}
In this section, we present the proof of the non-asymptotic concensus error, which has been present in Corollary \ref{thm:concensus error} .
\begin{lemma}
\label{lemma: concensus}Suppose Assumptions~\ref{assumption:smooth}, \ref{assumption: data heterogeneity}, \ref{assumption:unbiased}, and \ref{assumption: gossip communication} hold. If we take the hyperparameters as Theorem \ref{thm:concensus error}, then the concensus error of \ours satisfies that:
\begin{equation}
\label{eq:final11-xy-consesus11_new1}
    \begin{aligned}
        \dfrac{1}{T}\sum_{t=0}^{T-1} \mathbb{E}\left[\dfrac{1}{N}\Delta_t^2\right]\lesssim&\dfrac{N}{T}\left(\dfrac{b^2}{(1-\rho)^2}+\dfrac{1}{1-\rho}\right)+\dfrac{N^{\frac{1}{2}}}{T^{\frac{3}{2}}(1-\rho)^2}.
    \end{aligned}
\end{equation} 
\end{lemma}
\begin{proof}
From Eq. \eqref{eq:final11_xy_consesus1}, it holds that:
\begin{equation}
\label{eq:final11_xy_consesus1_new}
    \begin{aligned}
        &(1+\kappa^2)\sum_{t=0}^{T-1} \alpha\mathbb{E}\left[\left\Vert\bar{y}^{(t)}-\bar{y}^{(t)}_{\star}\right\Vert^2\right]+\sum_{t=0}^{T-1} \alpha\mathbb{E}\left[\left\Vert\bar{z}^{(t)}-\bar{z}^{(t)}_{\star}\right\Vert^2\right]\\
        \leq&((1+\kappa^2)C_{yx}+C_{zx})\sum_{t=0}^{T-1} \tau^2\alpha\mathbb{E}\left[\left\Vert\bar{h}^{(t)}\right\Vert^2\right]+((1+\kappa^2)C_{y1}+C_{z1})\dfrac{T\alpha^2}{N}+((1+\kappa^2)C_{y2}+C_{z2})\\
        &+C_{z3}T\alpha^2\iota^2+((1+\kappa^2)C_{yD}+C_{zD})\sum_{t=0}^{T-1} \alpha\mathbb{E}\left[\dfrac{1}{N}\Delta_t^2\right].
    \end{aligned}
\end{equation}

Moreover, from Eq. \eqref{eq:final_avgh11}, the definition of $C_x$ in Eq. \eqref{eq:constants_c},  $C_x\leq\dfrac{1}{2}$, and the step-sizes taken in Lemma \ref{final_lemma}, it holds that:
\begin{equation}
\label{eq:final_avgh11_new}
    \begin{aligned}
        &\alpha\sum_{t=0}^{T-1} \mathbb{E}\left[\left\Vert\bar{h}^{(t)}\right\Vert^2\right]\\
        \leq&\dfrac{4}{\tau}\mathbb{E}[E_0]+\dfrac{2}{c_3}\mathbb{E}\left[\left\Vert\bar{h}^{(0)}-\nabla\Phi(\bar{x}^{(0)})\right\Vert^2\right]+132L^2\dfrac{1}{N}\cdot\alpha\sum_{t=0}^{T-1}\mathbb{E}[\Delta_t^2]\\
        &+\dfrac{12(1+c_3)T\alpha^2}{N}\left(1+\dfrac{L_f^2}{\mu_g^2}\right)\sigma^2+20T\alpha\iota^2\\
        &+132L^2\left((1+\kappa^2)\sum_{t=0}^{T-1}\alpha\mathbb{E}\left[\left\Vert \bar{y}^{(t)}-\bar{y}^{(t)}_{\star}\right\Vert^2\right]+\sum_{t=0}^{T-1}\alpha\mathbb{E}\left[\left\Vert \bar{z}^{(t)}-\bar{z}^{(t)}_{\star}\right\Vert^2\right]\right).
    \end{aligned}
\end{equation}

Then, substituting Eq. \eqref{eq:final_avgh11_new} into Eq. \eqref{eq:final11_xy_consesus1_new}, using the step-sizes taken in Lemma \ref{final_lemma} and the fact that
$$132L^2\tau^2((1+\kappa^2)C_{yx}+C_{zx})\leq\dfrac{1}{2},$$
 it holds that:
\begin{equation}
\label{eq:final11_xy_consesus1_new1}
    \begin{aligned}
        &(1+\kappa^2)\sum_{t=0}^{T-1} \alpha\mathbb{E}\left[\left\Vert\bar{y}^{(t)}-\bar{y}^{(t)}_{\star}\right\Vert^2\right]+\sum_{t=0}^{T-1} \alpha\mathbb{E}\left[\left\Vert\bar{z}^{(t)}-\bar{z}^{(t)}_{\star}\right\Vert^2\right]\\
        \leq&\dfrac{8C_{con}}{\tau}\mathbb{E}[E_0]+\dfrac{4C_{con}}{c_3}\mathbb{E}\left[\left\Vert\nabla\bar{h}^{(0)}-\Phi(\bar{x}^{(0)})\right\Vert^2\right]+\dfrac{24C_{con}(1+c_3)T\alpha^2}{N}\left(1+\dfrac{L_f^2}{\mu_g^2}\right)\sigma^2\\
        &+40C_{con}T\alpha\iota^2+2((1+\kappa^2)C_{y1}+C_{z1})\dfrac{T\alpha^2}{N}+2((1+\kappa^2)C_{y2}+C_{z2})+2C_{z3}T\alpha^2\iota^2\\
        &+2((1+\kappa^2)C_{yD}+C_{zD}+1)\sum_{t=0}^{T-1} \alpha\mathbb{E}\left[\dfrac{1}{N}\Delta_t^2\right],
    \end{aligned}
\end{equation}
where $C_{con}:=((1+\kappa^2)C_{yx}+C_{zx})\tau^2$.

Then, with Eq. \eqref{eq:assumptioncon}, it holds from \eqref{eq:final11_xy_consesus1_new1} and \eqref{eq:final11-xy-consesus11} that
\begin{equation}
\label{eq:final11-xy-consesus11_new}
    \begin{aligned}
        \dfrac{1}{T}\sum_{t=0}^{T-1} \mathbb{E}\left[\dfrac{1}{N}\Delta_t^2\right]\lesssim&\dfrac{\rho^2\alpha}{T(1-\rho)^2}+\dfrac{\rho^2\alpha^3}{N(1-\rho)^2}+\dfrac{\rho^2\alpha^2b^2}{(1-\rho)^2}+\dfrac{\rho^2\alpha^2}{1-\rho}+\dfrac{\rho^2\alpha^2\iota^2}{(1-\rho)^2}+\dfrac{1}{(1-\rho)NT}.
    \end{aligned}
\end{equation} 
As we have taken $\iota$ satisfying $C_2\alpha\iota^2+\hat{C}_DC_{D3}\dfrac{\rho^2\alpha^2\iota^2}{(1-\rho)^2}\lesssim\dfrac{1}{{NT}}$ in Lemma \ref{final_lemma}. Thus with the step-sizes taken in Lemma \ref{final_lemma}, the non-asymptotic concensus error of D-SOBA has an upper bound of:
\begin{equation*}
    \begin{aligned}
        \dfrac{1}{T}\sum_{t=0}^{T-1} \mathbb{E}\left[\dfrac{1}{N}\Delta_t^2\right]\lesssim&\dfrac{N}{T}\left(\dfrac{b^2}{(1-\rho)^2}+\dfrac{1}{1-\rho}\right)+\dfrac{N^{\frac{1}{2}}}{T^{\frac{3}{2}}(1-\rho)^2}+\dfrac{1}{(1-\rho)NT}.
    \end{aligned}
\end{equation*}  
\end{proof}
}

\section{Asymptotic Rate and Transient Complexity Analysis of MDBO}
\label{analysis of gao}
In this section, we give the theoretical analysis of the asymptotic convergence rate when $T\to\infty$ and the transient time of MDBO \citep[Algorithm 1]{gao2023convergence} shown in Table \ref{table:comparison}, which is not explicitly given in their paper.

Firstly, if we take the constraints of \citep[Theorem 1]{gao2023convergence}, then the convergence rate of MDBO satisfies:
\begin{align}
    \dfrac{1}{T}\sum_{t=0}^{T-1}\mathbb{E}\left[\left\Vert\nabla\Phi(\bar{x}^{(t)})\right\Vert^2\right]\leq\dfrac{C_1}{(1-\rho)^2\eta T}+C_2\eta,
\end{align}
where $C_1,C_2,C_3$ are constants and $\eta$ denotes the step-size. Here we remove the $\mathcal{O}(\varepsilon)$ terms as $\varepsilon$ can be set sufficiently small.
Let $\eta=\left(\dfrac{C_1}{C_2(1-\rho)^2T}\right)^{\frac{1}{2}}$, then we have:
\begin{align}
    \dfrac{1}{T}\sum_{t=0}^{T-1}\mathbb{E}\left[\left\Vert\nabla\Phi(\bar{x}^{(t)})\right\Vert^2\right]\lesssim\dfrac{1}{(1-\rho)\sqrt{T}}.
\end{align}

Similarly, if we take the constraints of \citep[Theorem 2]{gao2023convergence}, then the convergence rate of MDBO satisfies:
\begin{align}
    \dfrac{1}{T}\sum_{t=0}^{T-1}\mathbb{E}\left[\left\Vert\nabla\Phi(\bar{x}^{(t)})\right\Vert^2\right]\leq\dfrac{C_1}{\eta T}+\dfrac{C_2\eta^2}{(1-\rho)^4}+C_3\dfrac{\eta}{N},
\end{align}
where $C_1,C_2,C_3,C_4$ are constants. Taking
\begin{align*}
    \eta_1 = \left( \dfrac{C_1(1-\rho)^4}{C_2T} \right)^{\frac{1}{3}},\quad \eta_2 =  \left( \dfrac{C_1N}{C_3T} \right)^{\frac{1}{2}},\quad \eta=\left(\frac{1}{\eta_1}+\frac{1}{\eta_2}\right)^{-1},
\end{align*}
we have:
\begin{align}
    \dfrac{1}{T}\sum_{t=0}^{T-1}\mathbb{E}\left[\left\Vert\nabla\Phi(\bar{x}^{(t)})\right\Vert^2\right]\lesssim\dfrac{1}{\sqrt{NT}}+\dfrac{1}{(1-\rho)^\frac{4}{3}{T}^\frac{2}{3}}.
\end{align}
Thus, it has an asymptotic convergence rate of $1/\sqrt{NT}$. Moreover, to achieve the linear speedup stage, the following condition is required:
\begin{align}
    \dfrac{1}{(1-\rho)^\frac{4}{3}{T}^\frac{2}{3}}\lesssim\dfrac{1}{\sqrt{NT}}.
\end{align}
The above constraint is equivalent to $T\lesssim\dfrac{N^3}{(1-\rho)^8}$.

\vskip 0.2in
\bibliography{reference.bib}

\end{document}